\documentclass[11pt,thmsa,a4paper,reqno]{article}

\usepackage{amssymb}
\usepackage{amsthm}
\usepackage{amsxtra}
\usepackage{mathrsfs}
\usepackage{upgreek}
\usepackage{bbm}
\usepackage{mathtools}
\usepackage{stmaryrd}

\usepackage{extarrows}
\usepackage{enumitem}
\usepackage{tikz-cd}
\usepackage{url}
\usepackage{booktabs}
\usepackage{graphicx}
\usepackage{resizegather}
\usepackage{arydshln}
\usepackage{floatrow}
\usepackage{soul}
\usepackage{relsize,etoolbox}
 \usepackage[all,cmtip,2cell]{xy}
\UseAllTwocells
\xyoption{2cell}

\usepackage{eucal}

\numberwithin{equation}{section}

\usepackage[scr,scaled=1.1]{rsfso}

\usepackage[pdftex,
                paper=a4paper,
                portrait=true,
                textwidth=165mm,
                textheight=235mm,
                tmargin=2.5cm,
                marginratio=1:1]{geometry}
                
\usepackage{pdfrender}

\makeatletter
\newtheorem*{rep@theorem}{\rep@title}
\newcommand{\newreptheorem}[2]{%
\newenvironment{rep#1}[1]{%
 \def\rep@title{#2 \ref{##1}}%
 \begin{rep@theorem}}%
 {\end{rep@theorem}}}
\makeatother

\newtheorem*{theoremO}{Theorem}
\newtheorem*{theoremA}{Theorem A}
\newtheorem*{theoremB}{Theorem B}
\newtheorem*{theoremC}{Theorem C}
\newtheorem*{theoremD}{Theorem D}
\newtheorem{theorem}{Theorem}[section]
\newreptheorem{theorem}{Theorem}
\newtheorem{proposition}[theorem]{Proposition}
\newtheorem{lemma}[theorem]{Lemma}

\newtheorem{corollary}[theorem]{Corollary}
\newreptheorem{corollary}{Corollary}

\theoremstyle{definition}

\newtheorem{example}[theorem]{Example}

\theoremstyle{remark}

\newtheorem*{acknowledgements}{Acknowledgements}

\newcommand{\RR}{\ensuremath{\mathbbmss{R}}}
\newcommand{\TT}{\ensuremath{\mathbbmss{T}}} 
\newcommand{\ZZ}{\ensuremath{\mathbbmss{Z}}} 
\newcommand{\II}{\ensuremath{\mathbbmss{1}}}

\newcommand{\Ccal}{\mathcal{C}}
\newcommand{\Dcal}{\mathcal{D}}

\newcommand{\Pcal}{\mathcal{P}}

\def\Dscr{\mathscr{D}}
\def\Iscr{\mathscr{I}}

\def\VEscr{\mathscr{VE}}
\def\DRscr{\mathscr{DR}}

\newcommand{\VE}{\mathrm{VE}}

\newcommand{\gfrak}{\mathfrak{g}}

\newcommand{\Ad}{\operatorname{Ad}}
\newcommand{\ad}{\operatorname{ad}}

\newcommand{\End}{\operatorname{End}}
\newcommand{\Hom}{\operatorname{Hom}}

\newcommand{\CE}{\operatorname{CE}}
\newcommand{\HC}{\operatorname{HC}}
\newcommand{\HH}{\operatorname{HH}}
\newcommand{\Tot}{\operatorname{Tot}}

\newcommand{\dW}{d_{\uW}}
\newcommand{\dCE}{\delta_{\CE}}

\newcommand{\Rep}{\operatorname{\mathbf{Rep}}}
\newcommand{\DGRep}{\operatorname{\mathbf{DGRep}}}
\newcommand{\Vect}{\operatorname{\mathbf{Vect}}}
\newcommand{\DGVect}{\operatorname{\mathbf{DGVect}}}
\newcommand{\BSS}{\operatorname{\mathbf{BSS}}}
\newcommand{\Mod}{\operatorname{\mathbf{Mod}}}
\newcommand{\DGMod}{\operatorname{\mathbf{DGMod}}}

\newcommand{\Ob}{\operatorname{Ob}}
\newcommand{\A}{\mathsf{A}}

\def\uC{{\mathrm C}}
\def\uW{{\mathrm W}}
\def\uS{{\mathrm S}}
\def\bas{{\rm bas}}
\def\ev{{\rm ev}}

\def\id{{\rm id}}
\def\pr{{\rm pr}}

\def\us{{\sf s}}
\def\uu{{\sf u}}
\def\uq{{\sf q}}
\def\ue{{\sf e}}
\def\Csf{{\sf C}}
\def\Ssf{{\sf S}}
\def\DR{{\sf{DR}}}
\def\AM{{\sf{AM}}}
\def\AW{{\sf{AW}}}
\def\EZ{{\sf{EZ}}}
\def\VE{{\sf{VE}}}

\def\sym{{\rm sym}}

\def\VEsrc{{\mathscr{VE}}}
\def\AMsrc{{\mathscr{AM}}}
\def\DRsrc{{\mathscr{DR}}}

\let\originalleft\left
\let\originalright\right
\renewcommand{\left}{\mathopen{}\mathclose\bgroup\originalleft}
\renewcommand{\right}{\aftergroup\egroup\originalright}

\DeclareFontFamily{U}{matha}{\hyphenchar\font45}
\DeclareFontShape{U}{matha}{m}{n}{
      <5> <6> <7> <8> <9> <10> gen * matha
      <10.95> matha10 <12> <14.4> <17.28> <20.74> <24.88> matha12
      }{}
\DeclareSymbolFont{matha}{U}{matha}{m}{n}
\DeclareFontSubstitution{U}{matha}{m}{n}
\DeclareMathSymbol{\abxcup}{\mathbin}{matha}{'131}

\newcommand*{\sbullet}{\raisebox{0.1ex}{\scalebox{0.6}{$\bullet$}}}
\newcommand*{\ssbullet}{\raisebox{-0.01ex}{\scalebox{0.6}{$\bullet$}}}
\newcommand*{\usbullet}{\raisebox{0.35ex}{\scalebox{0.6}{$\bullet$}}}

\DeclareMathAlphabet{\mathbfit}{OML}{cmm}{b}{it}
\newcommand*{\bigcdot}{\raisebox{-0.25ex}{\scalebox{1.3}{\,$\cdot$\,}}}

\begin{document}

\title{Singular chains on Lie groups and the Cartan relations II}

\author{Camilo Arias Abad\footnote{Escuela de Matem\'{a}ticas, Universidad Nacional de Colombia Sede Medell\'{i}n, email: camiloariasabad@gmail.com} { and} Alexander Quintero V\'{e}lez\footnote{Escuela de Matem\'{a}ticas, Universidad Nacional de Colombia Sede Medell\'{i}n, email: aquinte2@unal.edu.co}}

 \maketitle
 
 \begin{abstract}
Let $G$ be a simply connected Lie group with Lie algebra $\gfrak$ and denote by $\uC_{\sbullet}(G)$ the DG Hopf algebra of smooth singular chains on $G$. In a companion paper it was shown that the category of `sufficiently smooth’ modules over $\uC_{\sbullet}(G)$ is equivalent to the category of representations of $\TT \gfrak$, the DG Lie algebra which is universal for the Cartan relations. In this paper we show that, if $G$ is compact, this equivalence of categories can be extended to an $\A_{\infty}$-quasi-equivalence of the corresponding DG categories. As an intermediate step we construct an $\A_{\infty}$-quasi-isomorphism between the Bott-Shulman-Stasheff DG algebra associated to $G$ and the DG algebra of Hochschild cochains on $\uC_{\sbullet}(G)$. The main ingredients in the proof are the Van Est map and Gugenheim’s $\A_{\infty}$ version of De~Rham’s theorem. 
\end{abstract}

\tableofcontents

\section{Introduction}
This paper is a continuation of previous work by the first named author \cite{AriasAbad2019}, where an infinitesimal description of the category of modules over the algebra of singular chains on a Lie group is presented. Our main result is that in the compact case, the correspondence of \cite{AriasAbad2019} can be promoted to an $\A_\infty$ quasi-equivalence of DG categories.

Let $G$ be a Lie group with Lie algebra $\gfrak$, and denote by $\uC_{\sbullet}(G)$ the space of smooth singular chains on $G$. This space carries the structure of a DG Hopf algebra, where the product is induced by the Eilenberg-Zilber map and the coproduct by the Alexander-Whitney map. We write $\Mod(\uC_{\sbullet}(G))$ for the category of  sufficiently smooth modules over this DG Hopf algebra. We also denote by  $\TT \gfrak$ the DG Lie algebra which is universal for the Cartan relations on $\gfrak$,  and write $\Rep(\TT \gfrak)$ for the corresponding category of representations. The main result of \cite{AriasAbad2019} is the following.

\begin{theoremO}
Suppose that $G$ is a simply connected Lie group. There exists a differentiation functor \linebreak $\Dscr \colon  {\Mod(\uC_{\sbullet}(G)) \to \Rep(\TT \gfrak)}$ and an integration functor $\Iscr \colon \Rep(\TT \gfrak) \to \Mod(\uC_{\sbullet}(G))$ which are inverses to one another. In particular, the categories $\Mod(\uC_{\sbullet}(G))$ and $\Rep(\TT \gfrak)$ are equivalent as symmetric monoidal categories. 
\end{theoremO}

Let us now explain the content of the present work.
The category $\Mod(\uC_{\sbullet}(G))$ admits an enhancement to a DG category, which we shall denote by $\DGMod(\uC_{\sbullet}(G))$, and whose spaces of morphisms are Hochschild complexes for $\uC_{\sbullet}(G)$. Similarly, the category $\Rep(\TT\gfrak)$ may be enhanced to a DG category, which we denote by $\DGRep(\TT\gfrak)$, and whose spaces of morphisms are constructed in terms of the Weil algebra of $\gfrak$. Our main result is the following.

\begin{theoremA}
Suppose that $G$ is compact and simply connected. There exists a zig-zag of $\A_{\infty}$-quasi-equivalences that connects $\DGRep(\TT\gfrak)$ to $\DGMod(\uC_{\sbullet}(G))$. In particular, the DG categories $\DGRep(\TT\gfrak)$ and $\DGMod(\uC_{\sbullet}(G))$ are $\A_{\infty}$-quasi-equivalent. 
\end{theoremA}

For the proof we introduce an intermediate DG category $\BSS(G)$ and an ivariant version of it $\BSS^G(G)$, whose morphism spaces are defined by twisting the Bott-Shulman-Stasheff DG algebra introduced in \cite{BSS}. The following diagram, where each arrow represents an $\A_\infty$ equivalence of DG categories, summarizes the structure of the paper: 
\[
\xymatrix{
\DGRep(\TT\gfrak) && \BSS(G)\ar[rd]&\\
 &\ar[ul] \ar[ur]\BSS^G(G)&& \DGMod(\uC_{\sbullet}(G)).
}\]
The second arrow is an inclusion of categories which is a quasi-equivalence when $G$ is compact.\\

The comparison between $\BSS^G(G)$ and $\DGRep(\TT\gfrak)$
uses the Van Est map \cite{VanEst1953, AriasAbad-Crainic2011,Li-Bland-Meinrenken2015,Crainic2003,Xu-Weinstein1991,meinrenken2019van} and the noncommutative Weil algebra of Alekseev and Meinrenken \cite{Alekseev-Meinrenken2005,alekseev2000non}. We prove the following results.

\begin{theoremB}
Let $G$ be a compact and simply connected Lie group. There exists a DG functor \[\VEscr \colon \BSS^{G}(G) \to \DGRep(\TT\gfrak)\]which is a quasi-equivalence.  
\end{theoremB}

In order to compare $\BSS(G)$ and $\DGMod(\uC_{\sbullet}(G))$, we construct an $\A_\infty$ quasi-isomorphism between the Bott-Shulman-Stasheff algebra and the algebra of Hochschild cochains on singular chains on $G$.

\begin{theoremC}\label{thm:3.17aa}
Let $G$ be a Lie group. There is an $\A_{\infty}$-morphism \[\DR^{\Theta} \colon \Tot (\Omega^{\sbullet}(BG_{\sbullet})) \to \HC^{\sbullet}(\uC_{\sbullet}(G))\]
which is a quasi-isomorphism.
\end{theoremC}

This construction uses Chen's iterated integrals and combines a version of Gugenheim's $\A_{\infty}$ De~Rham's theorem for the classifying space $BG$ with the Eilenberg-Zilberg map. 

 \begin{theoremD}
 Let $G$ be a Lie group. There exists an $\A_{\infty}$-functor 
 $$
 \DRscr \colon \BSS(G) \longrightarrow \DGMod(\uC_{\sbullet}(G)),
 $$
which is an $\A_{\infty}$-quasi-equivalence of DG categories. 
\end{theoremD}

The $\A_{\infty}$-quasi-equivalence of DG categories between $\DGRep(\TT\gfrak)$ and $\DGMod(\uC_{\sbullet}(G))$ may be interpreted in terms of Chern-Weil theory for $\infty$-local systems on the classifying space $BG$, as studied in \cite{Block-Smith2014,Rivera-Zeinalian2018,Holstein2015,Abad-Schatz2016,AriasAbad-QuinteroVelez-VelezVasquez2019,brav2016relative}. Indeed,  $\infty$-local systems on a  topological space $X$ are described as objects of the DG category $\DGMod(\uC_{\sbullet}(\Omega X))$, where $\Omega X$ denotes the Moore based loop space of $X$. In  case  $X$ is  $BG$, the monoid of Moore loops on $BG$ is $\A_\infty$ equivalent to $G$. Thus, the DG category $\DGMod(\uC_{\sbullet}(G))$ can be thought of as that which parametrises $\infty$-local systems on $BG$. The $\A_{\infty}$-quasi-equivalence between $\DGRep(\TT\gfrak)$ and $\DGMod(\uC_{\sbullet}(G))$ is then an extension of the Chern-Weil computation of the cohomology of $BG$ for the trivial $\infty$-local system to the case of arbitrary $\infty$-local systems. The explicit construction of a Chern-Weil DG functor categorifying the Chern-Weil homomorphism is the subject of a forthcoming work \cite{AriasAbad-Pineda-QuinteroVelez2020}. 

The structure of the paper is as follows. In Section~\ref{sec:2}, we collect some preliminaries on DG categories, Hochschild complexes, Gugenheim's $\A_{\infty}$ De~Rham theorem, the Alexander-Whitney and Eilenberg-Zilber maps, representations of the DG Lie algebra $\TT \gfrak$, and the main result of \cite{AriasAbad2019} concerning the equivalence between $\Rep(\TT\gfrak)$ and $\Mod(\uC_{\sbullet}(G))$. Section~\ref{sec:3} is devoted to the study of the properties of the Van~Est map that are used in the proof of our main results, and the construction of the $\A_{\infty}$-quasi-isomorphism between the Bott-Shulman-Stasheff DG algebra $\Omega^{\sbullet}(BG_{\sbullet})$ and the DG algebra of Hochschild cochains $\HC^{\sbullet}(\uC_{\sbullet}(G))$. We conclude in Section~\ref{sec:4} with a discussion of the DG enhanced categories $\DGRep(\TT\gfrak)$ and $\BSS(G)$ and the proof of Theorems~B and C, which together imply our main result, Theorem~A. We also present some examples.

\begin{acknowledgements}
We would like to acknowledge the support of Colciencias through  their grant ``Estructuras lineales en topolog\'ia y geometr\'ia'', with contract number FP44842-013-2018.  We also thank  the Alexander von Humboldt foundation which supported our work through the Humboldt Institutspartnerschaftet ``Representations of Gerbes and higher holonomies''. We are grateful to Anders Kock for pointing us to his book \cite{Kock2006}, where the image of the De~Rham map is described. We would like to thank Konrad Waldorf for his hospitality during a visit to Greifswald, where part of this work was completed. We are also grateful to Manuel Rivera for several conversations related to this work.
\end{acknowledgements}

\subsection*{Notation and Conventions}
All vector spaces and algebras are defined over the field of real numbers $\RR$. If $V = \bigoplus_{k \in \ZZ} V^k$ is a graded vector space, we denote by $\us V$ its suspension, that is, the graded vector space with grading defined by
$$
(\us V)^{k} = V^{k+1},
$$
and by $\uu V$ its unsuspension, that is, the graded vector space with grading defined by
$$
(\uu V)^{k} = V^{k-1}.
$$
All our complexes will be cochain complexes, meaning that the differentials increase the degree by one. For each $n \geq 1$, we write $\Delta_n$ for the standard $n$-simplex. The geometric realisation of $\Delta_n$ that we take is
$$
\Delta_n = \left\{ (t_1, \dots, t_{n}) \in \RR^{n} \mid 1 \geq t_1 \geq \cdots \geq t_n \geq 0 \right\}.  
$$ 
If $M$ is a smooth manifold, we respectively denote by $\Omega^{\sbullet}(M)$, $\uC^{\sbullet}(M)$ and $\uC_{\sbullet}(M)$ the spaces of differential forms, smooth singular cochains and smooth singular chains defined on $M$.


\section{Preliminaries}\label{sec:2}
In this section, we review the basic definitions and results that will be needed in the sequel, in an attempt at making our paper as self-contained as possible. For a more detailed exposition on some of the topics covered in Sections \ref{sec:2.1}, \ref{sec:2.2}, \ref{sec:2.4} and \ref{sec:2.5}, the reader may consult \cite{Seidel2008}, \cite{Maclane1963} and \cite{Meinrenken2013}. 
\subsection{Hochschild chain and cochain complexes}\label{sec:2.1}
Let $A$ be a DG algebra and let $M$ be a DG bimodule over $A$. The \emph{Hochschild chain complex} of $A$ with values in $M$ is the graded vector space
\begin{equation}
\HC_{\sbullet}(A,M) = \bigoplus_{n \geq 0}M \otimes (\uu A)^{\otimes n},
\end{equation}
equipped with a differential $b$ which is the sum of two components $b_1$ and $b_2$ defined by the formulas
\begin{align}
\begin{split}
&b_1(m  \otimes  \uu a_1 \otimes \cdots  \otimes \uu a_n) = d_M m \otimes \uu a_1 \otimes \cdots \otimes \uu a_n \\
&\qquad  + \sum_{i=1}^{n} (-1)^{\vert m \vert + \sum_{j=1}^{i-1}\vert a_j \vert - i} m \otimes \uu a_1 \otimes \cdots \otimes \uu a_{i-1} \otimes \uu d a_i \otimes \uu a_{i+1} \otimes \cdots \otimes \uu a_n, 
\end{split}
\end{align}
and
\begin{align}
\begin{split}
& b_2 (m \otimes  \uu a_0 \otimes \cdots \otimes \uu a_n) = (-1)^{\vert m \vert + 1} m a_1 \otimes \uu a_2 \otimes \cdots \otimes \uu a_n \\
&\qquad +\sum_{i=1}^{n-1}(-1)^{\vert m \vert + \sum_{j=1}^{i}\vert a_j \vert - i+1} m \otimes \uu a_1 \otimes \cdots \otimes \uu a_{i-1} \otimes \uu (a_i a_{i+1}) \otimes \uu a_{i+2} \otimes  \cdots \otimes \uu a_n \\
&\qquad +(-1)^{(\vert m \vert + \sum_{j=1}^{n-1}\vert a_j\vert -n -1)(\vert a_n \vert -1)} a_n m \otimes \uu a_1 \otimes \cdots \otimes \uu a_{n-1},
\end{split}
\end{align}
for homogeneous elements $m \in M$ and $a_1,\dots, a_n \in A$.  The resulting cohomology is called the \emph{Hochschild homology of $A$ with values in $M$}. 
In the special case where $M = \RR$ is the trivial bimodule, we shall write $\HC_{\sbullet}(A)$ instead of $\HC_{\bullet}(A,\RR)$. 


The \emph{Hochschild cochain complex} of $A$ with values in $M$ is the cochain complex
\begin{equation}
\HC^{\sbullet}(A,M) = \Hom (\HC_{\sbullet}(A), M) = \bigoplus_{n \geq 0} \Hom ((\uu A)^{\otimes n}, M),
\end{equation}
with differential $b$ characterised by
\begin{align}\label{eqn:2.6aa}
\begin{split}
&(b \varphi) (\uu a_1 \otimes \cdots \otimes \uu a_{n})  =  d_M (\varphi(\uu a_1 \otimes \cdots  \otimes \uu a_n))  - (-1)^{\vert \varphi \vert} \varphi (b(\uu a_1 \otimes \cdots \otimes \uu a_n))  \\
 &\qquad   + (-1)^{\vert \varphi \vert(\vert a_1 \vert +1)} a_1 \varphi(\uu a_2 \otimes \cdots \otimes \uu a_n) - (-1)^{\vert \varphi \vert + \sum_{j=1}^{n-1}\vert a_j \vert + n-1} \varphi (\uu a_1 \otimes \cdots \otimes \uu a_{n-1})  a_n,
 \end{split}
\end{align}
for homogeneous elements $\varphi \in \Hom ((\uu A)^{\otimes n}, M)$ and $a_1,\dots, a_n \in A$. The resulting cohomology is called the \emph{Hochschild cohomology of $A$ with values in $M$}. In case $M$ is the trivial module $\mathbb{R}$, we will  write $\HC^{\sbullet}(A)$ instead of $\HC^{\sbullet}(A,M)$.

A case of special interest arises in the following way. Let $V$, $V'$ and $V''$ be DG modules over $A$, so that the hom-complexes $\Hom(V,V')$ and $\Hom(V',V'')$ are naturally DG bimodules over $A$. Then there is a cup product
$$
\abxcup \colon \HC^{\sbullet}(A,\Hom(V',V'')) \otimes \HC^{\sbullet}(A,\Hom(V,V')) \longrightarrow \HC^{\sbullet}(A,\Hom(V,V'')),
$$
which is defined by
\begin{align}\label{eqn:2.7aa}
\begin{split}
&(\psi \abxcup \varphi) (\uu a_1 \otimes \cdots \otimes \uu a_{m+n}) \\
 &\qquad\quad\,\,= (-1)^{\vert \varphi \vert (\sum_{i=1}^{m}\vert a_i \vert - m)} \psi (\uu a_1 \otimes \cdots \otimes \uu a_{m}) \circ \varphi (\uu a_{m+1} \otimes \cdots \otimes \uu a_{m+n}),
\end{split}
\end{align}
for homogeneous elements $\varphi \in \HC^{\sbullet}(A,\Hom(V,V'))$, $\psi \in \HC^{\sbullet}(A,\Hom(V',V''))$ and $a_1,\dots, a_{m+n} \in A$. This cup product is compatible with the differential $b$ in the sense that it satisfies the Leibniz rule. Given a DG algebra $A$, one can form the DG category of DG modules over $A$, which we denote by $\DGMod(A)$. Its objects are, of course, DG modules over $A$. For any two such objects $V$ and $V'$, the space of morphisms is the complex $\HC^{\sbullet}(A,\Hom(V,V'))$ and the composition law is the cup product. 

\subsection{DG categories, DG funtors and $\mathsf{A}_{\infty}$-functors}\label{sec:2.2}
A \emph{DG category} (where DG stands for ``differential graded'') over a field $K$ is a $K$-linear category $\Ccal$ such that for every two objects $X$ and $Y$ the space of arrows $\Hom_{\Ccal}(X,Y)$ is equipped with a structure of a cochain complex of $K$-vector spaces, and for every three objects $X$, $Y$ and $Z$ the composition map $\Hom_{\Ccal}(Y,Z) \otimes_K \Hom_{\Ccal}(X,Y) \to \Hom_{\Ccal}(X,Z)$ 
is a morphism of cochain complexes. Thus, by definition, 
$$
\Hom_{\Ccal}(X,Y) = \bigoplus_{n \in \ZZ} \Hom_{\Ccal}^{n}(X,Y)
$$
is a graded $K$-vector space with a differential $d \colon \Hom_{\Ccal}^{n}(X,Y) \to \Hom_{\Ccal}^{n+1}(X,Y)$. The elements $f \in \Hom_{\Ccal}^{n}(X,Y)$ are called \emph{homogeneous of degree $n$}, and we write $\vert f \vert = n$. We shall denote the set of objects of $\Ccal$ by $\Ob \Ccal$.

The fundamental example of a DG category is the category of cochain complexes of $K$-vector spaces, which we denote by $\DGVect_K$. Its objects are cochain complexes of $K$-vector spaces and the morphism spaces $\Hom_{\DGVect_K}(X,Y)$ are endowed with the differential defined as
$$
d (f) = d_{Y} \circ f - (-1)^n f \circ d_{X},
$$
for any homogeneous element $f$ of degree $n$.   

Given a DG category $\Ccal$ one defines an ordinary category $\mathbf{Ho}(\Ccal)$ by keeping the same set of objects and replacing each $\Hom$ complex by its $0$th cohomology. We call $\mathbf{Ho}(\Ccal)$ the \emph{homotopy category} of $\Ccal$. 
If $\Ccal$ and $\Dcal$ are DG categories, a DG functor $F \colon \Ccal \to \Dcal$ is an $K$-linear functor whose associated map for $X, Y \in \Ob \Ccal$,
$$
F_{X,Y} \colon \Hom_{\Ccal}(X,Y) \to \Hom_{\Dcal}(F(X),F(Y)),
$$ 
is a morphism of cochain complexes. Notice that any DG functor $F \colon \Ccal \to \Dcal$ induces an ordinary functor 
$$
\mathbf{Ho}(F) \colon \mathbf{Ho}(\Ccal) \to \mathbf{Ho}(\Dcal)
$$
between the corresponding homotopy categories. A DG functor $F \colon \Ccal \to \Dcal$ is said to be \emph{quasi fully faithful} if for every pair of objects $X, Y \in \Ob \Ccal$ the morphism $F_{X,Y}$ is a quasi-isomorphism. Moreover, the DG functor $F $ is said to be \emph{quasi essentially surjective} if $\mathbf{Ho}(F)$ is essentially surjective. A DG functor which is both quasi fully faithful and quasi essentially surjective is called a \emph{quasi-equivalence}. 

There is a more general notion of functor between DG categories, that of an $\mathsf{A}_{\infty}$-functor, where the composition is preserved only up to an infinite sequence of coherence conditions. It will be useful to introduce first the Hochschild chain complex of a DG category. 

Let $\Ccal$ be a small DG category. The \emph{Hochschild cochain complex} of $\Ccal$ is the complex 
$$
\bigoplus_{X_0,\dots, X_n}\us\! \Hom_{\Ccal}(X_{n-1},X_n) \otimes_K \cdots \otimes_K \us\! \Hom_{\Ccal}(X_{0},X_1),
$$
where $X_0,\dots,X_n$ range through the objects of $\Ccal$, and whose differential $b$ is the sum of two components $b_1$ and $b_2$ given by the formulas
$$
b_1 (f_{n-1} \otimes \cdots \otimes f_0) = \sum_{i=0}^{n-1} (-1)^{\sum_{j=i+1}^{n-1}\vert f_{j} \vert  + n - i - 1}  f_{n-1} \otimes \cdots \otimes d f_i \otimes \cdots\otimes f_0
$$
and 
$$
b_2 (f_{n-1} \otimes \cdots \otimes f_0) = \sum_{i=0}^{n-2} (-1)^{\sum_{j=i+2}^{n-1}\vert f_{j} \vert + n - i } f_{n-1} \otimes \cdots \otimes (f_{i+1} \circ f_i) \otimes \cdots\otimes f_0
$$
for homogeneous elements $f_0 \in  \us\! \Hom_{\Ccal}(X_{0},X_1), \dots, f_{n-1} \in \us\! \Hom_{\Ccal}(X_{n-1},X_n)$.  Here $d$ denotes indistinctly the differential in any of the spaces $\Hom_{\Ccal} (X_{i},X_{i+1})$. It is easy to check that indeed $b^2=0$, by cancellation of terms with opposite signs. 

With this in mind, the formal definition of an $\mathsf{A}_{\infty}$-functor is given as follows. Let $\Ccal$ and $\Dcal$ be DG categories. An \emph{$\mathsf{A}_{\infty}$-functor} $F \colon \Ccal \to \Dcal$ is the datum of a map of sets $F_0 \colon \Ob \Ccal \to \Ob \Dcal$ and a collection of $K$-linear maps of degree $0$
$$
F_n \colon \us\! \Hom_{\Ccal}(X_{n-1},X_n) \otimes_K \cdots \otimes_K \us\! \Hom_{\Ccal}(X_{0},X_1) \to \Hom_{\Dcal}(F_0(X_0),F_0(X_n))
$$
for every collection $X_0,\dots,X_n \in \Ob \Ccal$, such that the relation
\begin{align*}
b_1 \circ  F_n  + \sum_{i+j = n} b_2  \circ (F_i \otimes F_j ) = \sum_{i + j + 1= n} F_n \circ (\id^{\otimes i} \otimes b_1 \otimes \id^{\otimes j}) + \sum_{i + j + 2 = n} F_{n-1} \circ (\id^{\otimes i} \otimes b_2 \otimes \id^{\otimes j}) 
\end{align*} 
is satisfied for any $n \geq 1$. One also requires that $F_1(\id_{X}) = \id_{F_0(X)}$ for all objects $A$ in $\Ccal$, as well as $F_n(f_{n-2} \otimes \cdots \otimes f_{i} \otimes \id_{X_{i}} \otimes f_{i-1} \otimes \cdots \otimes f_0 ) =0$ for any $n \geq 1$, any $0 \leq i \leq n-2$, and any chain of morphisms $f_0 \in  \us\! \Hom_{\Ccal}(X_{0},X_1), \dots, f_{n-2} \in \us\! \Hom_{\Ccal}(X_{n-2},X_{n-1})$. 

The above relation when $n=1$ implies that $F_1$ is a morphism of cochain complexes. On the oner hand, for $n=2$ we find that $F_1$ preserves the compositions on $\Ccal$ and $\Dcal$, up to a homotopy defined by $F_2$. In particular, a DG functor between $\Ccal$ and $\Dcal$ is the same as an $\mathsf{A}_{\infty}$-functor such that $F_n = 0$ for $n \geq 2$. It also follows that $F_1$ induces and ordinary functor 
$$
\mathbf{Ho}(F_1) \colon  \mathbf{Ho}(\Ccal)  \to \mathbf{Ho}(\Dcal). 
$$
An $\mathsf{A}_{\infty}$-functor $F \colon \Ccal \to \Dcal$ is called $\A_{\infty}$-\emph{quasi fully faithfull} if $F_1$ is a quasi-isomorphism, and it is called $\A_{\infty}$-\emph{quasi essentially surjective} if $\mathbf{Ho}(F_1)$ is essentially surjective. Finally, an $\mathsf{A}_{\infty}$-functor $F$ is called a $\A_{\infty}$-\emph{quasi-equivalence} if it is both quasi fully faithfull and quasi essentially surjective.  We say that two DG categories are $\A_\infty$ equivalent if they can be connected by a zig-zag of $\A_\infty$ quasi-equivalences.

\subsection{Gugenheim's $\A_{\infty}$ De~Rham theorem} \label{sec:2.3}
The usual De~Rham map, which sends a differential form to a singular cochain by integration, is not an algebra map. However, it induces an isomorphism of algebras in cohomology. A more complete explanation of this fact is due to Gugenheim. In \cite{Gugenheim1977}, this author uses Chen's iterated integrals \cite{Chen1977} to extend the De~Rham map to an $\A_{\infty}$-quasi-isomorphism of DG algebras. Here we will review this construction, which will be needed later. We follow the presentation in \cite{Abad-Schatz2013}. 

Let us start with some background. For a smooth manifold $M$ we denote by $\Pcal M$ the path space of $M$, that is, the space of all smooth maps from $I$ to $M$ which we regard as a diffeological space. Given another manifold $X$, one says that a map $f \colon X \to \Pcal M$ is smooth if the map $\widehat{f} \colon [0,1] \times X \to M$ defined for any $t \in I$ and $x \in X$ by
\begin{equation}
\widehat{f} (t, x) = f(x)(t),
\end{equation}
is smooth. With this in mind, we may define differential forms on $\Pcal M$ as follows. We first consider the category $C^{\infty}(-,\Pcal M)$ whose objects are pairs $(X, f)$ where $X$ is a smooth manifold and $f$ is a smooth map from $X$ to $\Pcal M$ and whose morphisms from one such pair $(X,f)$ to another $(Y,g)$ are smooth maps $h \colon X \to Y$ such that $f = g \circ h$. Next, if $\Vect_{\RR}$ denotes the category of real vector spaces, we consider the functor $\RR(-)$ from $C^{\infty}(-,\Pcal M)$ to $\Vect_{\RR}$ which sends any object in $C^{\infty}(-,\Pcal M)$ to $\RR$ and every morphism to the identity, along with the functor $\Omega^{\sbullet}(-)$ from $C^{\infty}(-,\Pcal M)$ to $\Vect_{\RR}$ sending an object $(X,f)$ to $\Omega^{\sbullet}(X)$ and a morphism $h$ to its pullback $h^*$. Then, a differential form on $\Pcal M$ is a natural transformation from $\RR(-)$ to $\Omega^{\sbullet}(-)$. This definition simply means that we declare a differential form on $\Pcal M$ to be determined by its pullback along smooth maps from a smooth manifold. 
We shall now explain Chen's iterated integrals taking values on differential forms on the path space $\Pcal M$. First we need the following piece of notation. If $\Delta_n$ denotes the $n$-simplex, we write $\ev \colon \Delta_n \times \Pcal M \to M^{n}$ for the evaluation map defined as
\begin{equation}
\ev((t_1,\dots, t_n), \gamma) = (\gamma(t_1),\dots, \gamma(t_n)),
\end{equation}
for $(t_1,\dots,t_n) \in \Delta_n$ and $\gamma \in \Pcal M$. Further, we let $p_i$ stand for the $i$-th projection from $M^n$ to $M$ for any $i = 1,\dots, n$, and $\pi$ for the projection from $\Delta_n \times M$ to $M$. Then, Chen's map $\Csf \colon (\us \Omega^{\sbullet}(M))^{\otimes n} \to \Omega^{\sbullet}(\Pcal M)$ is defined by setting
\begin{equation}
\Csf (\omega_1 \otimes \cdots \otimes \omega_n) = (-1)^{\sum_{j=1}^{n}\vert \omega_j \vert (n-j)} \pi_{*} (\ev^{*}(p_1^* \omega_1 \wedge \cdots \wedge p_n^* \omega_n )),  
\end{equation}
for homogeneous elements $\omega_1 ,\dots ,\omega_n \in \us \Omega^{\sbullet}(M)$, where here $\pi_{*} \colon \Omega^{\sbullet}(\Delta_n \times M) \to \Omega^{\sbullet}(M)$ denotes the pushforward along the projection $\pi$.\footnote{In the case when $M$ is compact and oriented, the pushforward $\pi_*$ is characterized by the property that
$$
\int_{M} \pi_* \omega = \int_{\Delta_n \times M} \omega,
$$ 
for all $\omega \in \Omega^{\sbullet}(\Delta_n \times M)$. }

Besides Chen's map, Gugenheim's construction uses some combinatorial maps that we now describe. For each $n \geq 1$, let $\lambda_n \colon I^{n-1} \to \Pcal I^{n}$ be the map  that sends every element $(s_1,\dots,s_{n-1})$ of $I^{n-1}$ to the piecewise linear path which goes backwards through the $n+1$ points 
$$
0 \longleftarrow  s_1 e_1 \longleftarrow  s_1 e_1 + s_2 e_2 \longleftarrow \dots \longleftarrow  s_1 e_1 + \cdots + s_{n-1} e_{n-1} \longleftarrow s_1 e_1 + \cdots + s_{n-1} e_{n-1} + e_n,
$$
with $e_1,\dots, e_n$ being the standard basis of $\RR^{n}$, and $\pi_n \colon I^{n} \to \Delta_n$ the map given by
\begin{equation}
\pi_n (s_1,\dots,s_n) = (t_1,\dots,t_n), 
\end{equation}
for $(s_1,\dots, s_n) \in I^{n}$, with $t_i = \max\{s_i,\dots,s_n\}$ for any $i = 1,\dots, n$. We then obtain, for each $n \geq 1$, a map $\theta_n \colon I^{n-1} \to \Pcal \Delta_n$ which is defined as the composition
$$
I^{n-1} \xlongrightarrow{\phantom{a}\lambda_n\phantom{a}} \Pcal I^{n} \xlongrightarrow{\phantom{a}\Pcal\pi_n\phantom{a}} \Pcal \Delta_n,
$$
where $\Pcal \pi_n$ is the map induced on path spaces by $\pi_n$. We also, by convention, set $\theta_0$ to be the map from a point to a point. \\

Using the above notation, we consider the map $\Ssf \colon \Omega^{\sbullet}(\Pcal M) \to \us \uC^{\sbullet}(M)$ from the de~Rham complex of the path space $\Pcal M$ to the unsuspension of $\uC^{\sbullet}(M)$, obtained as the composition of the map $\Omega^{\sbullet}(\Pcal M) \to \uC^{\sbullet}(M)$ given, for each $\varphi \in \Omega^{\sbullet}(\Pcal M)$, by 
$$
(\sigma \colon \Delta_n \to M) \longmapsto \int_{I^{n-1}} \theta_n^* \Pcal \sigma^* \varphi \in \RR,
$$
followed by the unsuspension $\us \colon \uC^{\sbullet}(M) \to \us\uC^{\sbullet}(M)$. We then proceed to define, for $n \geq 1$, a sequence of linear maps $\DR_n \colon (\us \Omega^{\sbullet}(M))^{\otimes n} \to \us \uC^{\sbullet}(M)$ in the following way. For $n = 1$, we set
\begin{equation}\label{eqn:2.12}
\DR_1(\omega)(\sigma) = \int_{\Delta_k} \sigma^*\omega,
\end{equation}
for $\omega \in \us \Omega^{\sbullet}(M)$ and $\sigma \in \uC_k (M)$. For $n > 1$, we set
\begin{equation}
\DR_n (\omega_1 \otimes \cdots \otimes \omega_n) = (-1)^{\sum_{j=1}^{n}\vert \omega_j \vert + n} (\Ssf \circ \Csf)(\omega_1 \otimes \cdots \otimes \omega_n),
\end{equation}
for homogeneous elements $\omega_1,\dots,\omega_n \in \us\Omega^{\sbullet}(M)$. In term of these, we are now in position to state Gugenheim's main result.   

\begin{theorem}\label{thm:2.1}
For $n \geq 1$, the sequence of maps $\DR_n \colon (\us \Omega^{\sbullet}(M))^{\otimes n} \to \us \uC^{\sbullet}(M)$ determines an $\A_{\infty}$-morphism from the de~Rham complex of $M$ to the singular cochain complex of $M$, both viewed as DG algebras. Moreover, this $\A_{\infty}$-morphism is an $\A_{\infty}$-quasi-isomorphism which is natural with respect to pullbacks along smooth maps. 
\end{theorem}

\subsection{Alexander-Whitney and Eilenberg-Zilber maps}\label{sec:2.4}
In this subsection we recall the definitions of the Alexander-Whitney and Eilenberg-Zilber maps. These will enable us to give the singular chain complex of a Lie group the structure of a DG Hopf algebra. We begin with the Alexander-Whitney map. For $p \leq n$, the inclusions of the standard $p$-simplex as the front and the back $p$-th face of the standard $n$-simplex will be denoted respectively by $f_{p}^{n} \colon \Delta_p \to \Delta_n$ and $b_{p}^{n}\colon \Delta_p \to \Delta_n$. Explicitly, 
\begin{align}\label{eqn:2.7}
\begin{split}
f_{p}^{n}(t_1,\dots, t_p) &= (t_1,\dots, t_p, 0,\dots, 0), \\
b_{p}^{n}(t_1,\dots, t_p) &= (1,\dots, 1, t_1,\dots, t_p).
\end{split}
\end{align}
Also, for two fixed smooth manifolds $X$ and $Y$, we let $\pi_1 \colon X \times Y \to X$ and $\pi_2 \colon X \times Y \to Y$ denote the two natural projections. Then, the Alexander-Whitney map $\AW \colon \uC_{\sbullet}(X \times Y) \to \uC_{\sbullet}(X) \otimes \uC_{\sbullet}(Y)$ is the chain map given, for each singular $n$-simplex $\sigma \colon \Delta_n \to X \times Y$, by the formula
\begin{equation}
\AW (\sigma) = \sum_{p + q = n} (\sigma_1 \circ f_{p}^{n}) \otimes (\sigma_2 \circ b_{q}^{n}), 
\end{equation}
where $\sigma_1 = \pi_1 \circ \sigma$ and $\sigma_2 = \pi_2 \circ \sigma$. For us, the most important property of this map is that it allows us to define a DG coalgebra structure on the space of singular chains $\uC_{\sbullet}(X)$. The coproduct $\Delta \colon \uC_{\sbullet} (X) \to \uC_{\sbullet}(X) \otimes \uC_{\sbullet}(X)$ is formed by composition of the  map $\uC_{\sbullet}(X) \to \uC_{\sbullet}(X \times X)$ induced by the diagonal $X \to X \times X$ with the Alexander-Whitney map $\AW \colon \uC_{\sbullet}(X \times X) \to \uC_{\sbullet}(X) \otimes \uC_{\sbullet}(X)$. The counit $\varepsilon \colon \uC_{\sbullet}(X) \to \RR$ is induced by the projection map which collapses $X$ to a point. 

Now we turn to the Eilenberg-Zilber map. Such a map is based on the simple fact that a cube of dimension $n$ is the union of $n!$ simplices. 
\begin{center}
\includegraphics[scale=0.5]{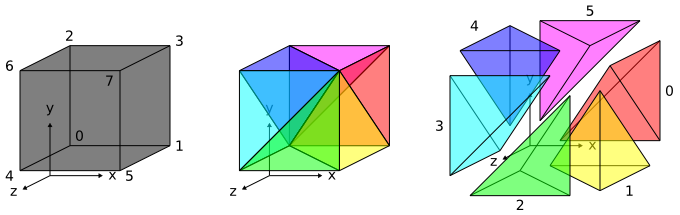}
\end{center}
We fix again two smooth manifolds $X$ and $Y$. The Eilenberg-Zilber map $\EZ \colon \uC_{\sbullet}(X) \otimes \uC_{\sbullet}(Y) \to \uC_{\sbullet}(X \times Y)$ is the chain map given, for each singular $m$-simplex $\sigma \colon \Delta_m \to X$ and each singular $n$-simplex $\tau \colon \Delta_n \to Y$, by the formula
\begin{equation}\label{eqn:2.9}
\EZ (\sigma \otimes \tau) = \sum_{\chi \in \mathfrak{S}_{m,n}} (-1)^{\chi} (\sigma \times \tau) \circ \chi_*,
\end{equation}
where, as the notation implies, the sum over $\chi$ is taken over all $(m,n)$ shuffles and $\chi_* \colon \Delta_{m+n} \to \Delta_m \times \Delta_n$ is the map defined by
\begin{equation}
\chi_* (t_1, \dots, t_{m+n}) = ((t_{\chi(1)}, \dots, t_{\chi(m)}), (t_{\chi(m+1)}, \dots, t_{\chi(m+n)}))
\end{equation}
We state without proof the key properties of the Eilenberg-Zilber map (see \cite{Maclane1963} and \cite{Eilenberg-Moore1966}).

\begin{proposition}\label{prop:2.2}
The Eilenberg-Zilber map $\EZ \colon  \uC_{\sbullet}(X) \otimes \uC_{\sbullet}(Y) \to \uC_{\sbullet}(X \times Y)$ satisfies:
\begin{enumerate}
\item It is associative, that is, given a third smooth manifold $Z$, the following diagram commutes
$$
\xymatrix@C=8ex{\uC_{\sbullet}(X) \otimes \uC_{\sbullet}(Y) \otimes \uC_{\sbullet}(Z)  \ar[r]^-{\id \otimes \EZ} \ar[d]_-{\EZ \otimes \id}& \uC_{\sbullet}(X \times Y) \otimes \uC_{\sbullet}(Z) \ar[d]^-{\EZ} \\
\uC_{\sbullet}(X) \otimes \uC_{\sbullet}(Y \times Z) \ar[r]^-{\EZ} & \uC_{\sbullet}(X \times Y \times Z);}
$$
\item It is a map of DG coalgebras.
\item $\EZ$ and $\AW$ are inverses up to natural chain homotopies. 
\end{enumerate}
\end{proposition}

From the associativity property, it follows that if $X_1,\dots, X_r$ are smooth manifolds, then there is an $r$-fold Eilenberg-Zilber map $\EZ \colon \uC_{\sbullet}(X_1) \otimes \cdots \otimes \uC_{\sbullet}(X_r) \to \uC_{\sbullet}(X_1 \times \cdots \times X_r)$ which is obtained by applying the binary Eilenberg-Zilber maps $r-1$ times. Explicitly, this map is defined as follows. Given simplices $\sigma_i \colon \Delta_{n_i} \to X_i$ with $i=1,\dots, r$, one has:
\begin{equation}\label{eqn:2.18}
\EZ (\sigma_1 \otimes \cdots \otimes \sigma_r) = \sum_{\chi \in \mathfrak{S}_{n_1,\dots,n_r}} (-1)^{\chi} (\sigma_1 \times \cdots \times \sigma_r) \circ \chi_*, 
\end{equation}
where the sum over $\chi$ is taken over all $(n_1,\dots,n_r)$-shuffles and $\chi_* \colon \Delta_{n_1 + \cdots + n_r} \to \Delta_{n_1} \times \cdots \times  \Delta_{n_r}$ now denotes the map defined by
\begin{equation}
\chi_* (t_1 ,\dots, t_{n_1 + \cdots + n_r}) = ((t_{\chi(1)},\dots, t_{\chi(n_1)} ), \dots, (t_{\chi(n_1 + \cdots + n_{r-1}+1)},\dots, t_{\chi(n_1 + \cdots + n_r)} ) ).
\end{equation}

We will now specialize the discussion by replacing $X$ with a Lie group $G$. In this case, the space of singular chains $\uC_{\sbullet}(G)$ acquires the structure of a DG Hopf algebra. The product $m \colon \uC_{\sbullet}(G) \otimes \uC_{\sbullet}(G) \to \uC_{\sbullet}(G)$ is formed by composition of the Eilenberg-Zilber map $\EZ \colon \uC_{\sbullet}(G) \otimes \uC_{\sbullet}(G) \to \uC_{\sbullet}(G \times G)$ with the map $\mu_{*} \colon \uC_{\sbullet}(G \times G) \to \uC_{\sbullet}(G)$ induced by the multiplication map $\mu \colon G \times G \to G$. The unit $u \colon \RR \to \uC_{\sbullet}(G)$ is induced by the inclusion of a point as the identity element of $G$, and the antipode $S \colon \uC_{\sbullet}(G) \to \uC_{\sbullet}(G)$ is induced by the inversion map $\iota \colon G \to G$.

\subsection{Representations of the DG Lie algebra $\TT\gfrak$}\label{sec:2.5}
For a Lie algebra $\gfrak$, consider the DG Lie algebra $\TT\gfrak = \uu \gfrak \oplus \gfrak$ with degree $-1$ generators $i(x) \in \uu \gfrak$ and degree $0$ generators $L(x) \in \gfrak$ for $x \in \gfrak$. The Lie bracket of $\TT\gfrak$ is induced by the Lie bracket of $\gfrak$, and the differential is defined by
\begin{align}
\begin{split}
d (i(x)) &= L(x), \\
d (L(x) ) &= 0.
\end{split}
\end{align}
The generators $i(x)$ and $L(x)$ satisfy the Cartan relations
\begin{align}
\begin{split}
[i(x),i(y)] &= 0,\\
 [L(x),L(y)] &= L([x,y]), \\
  [L(x),i(y)] &=i([x,y]).
 \end{split}
\end{align}
By a \emph{representation} of $\TT \gfrak$ we mean a cochain complex $V$ together with a DG Lie algebra homomorphism $\TT\gfrak \to \End(V)$. That is, it consists of a representation of $\TT\gfrak$ on $V$, where the operators $i_{x}\in \End(V)^{-1}$ and $L_{x} \in \End(V)^0$ corresponding to $i(x),L(x) \in \TT\gfrak$ satisfy the relations
\begin{align}\label{eqn:2.22}
\begin{split}
[i_{x},\delta] &= L_{x}, \\
[L_{x},\delta]&= 0, \\
[i_{x},i_{y}] &= 0, \\
[L_{x},L_{y}] &= L_{[x,y]}, \\
 [L_{x},i_{y}]&= i_{[x,y]},
 \end{split}
\end{align}
with $\delta$ being the differential of $V$. The operators $i_{x}$ are called \emph{contractions} and the operators $L_{x}$ are called \emph{Lie derivatives}. 

An important example of a representation of $\TT\gfrak$ is the Chevalley-Eilenberg complex $\CE(\mathfrak{g})$. As a graded algebra is the exterior algebra $\Lambda^{\sbullet}\mathfrak{g}^*$, where $\gfrak^*$ has degree $1$. The differential $\dCE$ is the unique derivation such that for $\xi \in \Lambda^1\gfrak^*$, $\dCE \xi$ is the element in $\Lambda^2\gfrak^*$ defined by
\begin{equation}
(\dCE \xi)(x,y) = - \xi([x,y]).
\end{equation}
It follows from the Jacobi identity that $\dCE$ defined in this manner is a differential. The contraction $i_{x}$ and Lie derivatives $L_{x}$ are the unique derivations such that for $\xi \in \Lambda^1\gfrak^*$,

\begin{align}
\begin{split}
i_{x} \xi &= \langle \xi, x \rangle, \\
L_{x} \xi &= \ad_{x}^* \xi,
\end{split}
\end{align} 
where $\ad_{x}^*$ denotes the infinitesimal coadjoint action of the element $x$. 

Explicit formulas for these various maps, which will be useful later on, are obtained by introducing a basis for $\gfrak^*$. Let $e_a$ be a basis for $\gfrak$ with dual basis $e^{a}$ and structure constants $f^{a}_{\phantom{a}bc}= \langle e^{a},[e_b,e_c] \rangle$, and write $i_a$ and $L_a$ for the contraction $i_{e_a}$ and the Lie derivative $L_{e_a}$ acting on $\CE(\gfrak)$. Then the explicit formulas for $\dCE$, $i_{a}$ and $L_{a}$ are the following:
\begin{align}
\begin{split}
\dCE e^{a} &= - \frac{1}{2} f^{a}_{\phantom{a}bc} e^{b} \wedge e^{c}, \\
i_{b} e^a &= \delta_{b}^{\phantom{b}a}, \\
L_{b} e^{a} &= - f^{a}_{\phantom{a}bc}e^{c}. 
\end{split}
\end{align} 
Here the convention that repeated indices are summed over is in place.

Another example of a representation of $\TT\gfrak$ is the Weil algebra $\uW\gfrak$. The underlying graded commutative algebra of $\uW\gfrak$ is the tensor product
\begin{equation}
\uW\gfrak = \Lambda^{\sbullet}\gfrak^* \otimes \uS^{\sbullet}\gfrak^*,
\end{equation}
where $\uS^{\sbullet}\gfrak^*$ is the symmetric algebra of $\gfrak^*$ and where we associate to each $\xi \in \gfrak^*$ the degree $1$ generators $t(\xi) \in \Lambda^1 \gfrak^*$ and the degree $2$ generators $w(\xi) \in \uS^1\gfrak^*$. The differential on $\uW\gfrak$ is given by \begin{align}
\begin{split}
\dW (t(\xi)) &=  w(\xi) + \dCE (t(\xi)), \\
\dW (w(\xi)) &= \dCE (w(\xi)),
\end{split}
\end{align}
where $\dCE$ is the differential of the Chevalley-Eilenberg complex $\CE(\gfrak)$. The operators $i_x$ and $L_x$ are the unique derivations such that
\begin{align}
\begin{split}
i_x (t(\xi)) &= \langle t(\xi),x \rangle,\\
i_x (w(\xi)) &= 0, \\
L_x(t(\xi)) &= \ad_x^*(t(\xi)), \\
L_x(w(\xi)) &= \ad_x^*(w(\xi)).
\end{split}
\end{align}

As for $\CE(\gfrak)$, it will be useful to express the differential $\dW$ and the operator $i_x$ and $L_x$ in terms of a dual basis $e^{a}$ of $\gfrak^*$ and the structure constants $f^{a}_{\phantom{a}bc}$ of $\gfrak$. If we write $t^{a} = t(e^{a})$ and $w^{a} = w(e^{a})$, they are as follows:
\begin{align}\label{eqn:2.29}
\begin{split}
\dW  t^{a} &= w^{a} - \frac{1}{2} f^{a}_{\phantom{a}bc} t^{b} t^{c}, \\
\dW w^{a} &= f^{a}_{\phantom{a}bc} w^{b} t^{c}, \\
i_b t^{a} &= \delta_{b}^{\phantom{b}a}, \\
i_b w^{a} &= 0, \\
L_{b} t^{a} &= - f^{a}_{\phantom{a}bc} t^{c}, \\
L_{b} w^{a} &=- f^{a}_{\phantom{a}bc} w^{c}. 
\end{split}
\end{align}
Clearly, the elements $t^{a}$ and $\dW t^{a}$  also generate $\uW\gfrak$ freely, which implies that the
Weil algebra is acyclic.\\

If $V$ and $W$ are representations of $\TT \gfrak$, a \emph{homomorphism} $f \colon V \to W$ is a morphism of cochain complexes commuting with the operators $i_x$ and $L_x$. It is cleat that the identity map of a representation of $\TT \gfrak$ onto itself is a homomorphism, and that the composition of two homomorphisms is again a homomorphism. Thus, representations of $\TT \gfrak$ and their homomorphisms form a category which we denote by $\Rep (\TT \gfrak)$. For later purposes, we note that this category is symmetric monoidal with tensor product $V \otimes V'$ of two objects $V$ and $V'$ given by the tensor product of the underlying cochain complexes equipped with the actions of $i_x$ and $L_x$ defined by the formulas
\begin{align}
\begin{split}
i_x (v \otimes v') &= i_x v \otimes v' + (-1)^{\vert v \vert } v \otimes i_x v', \\
L_x (v \otimes v') &= L_x v \otimes v' +  v \otimes L_x v',
\end{split}
\end{align}
for homogeneous elements $v \in V$ and $v' \in V'$. Clearly, the unit object is the trivial representation $\underline{\RR}$ viewed as a complex concentrated in degree zero. 

\subsection{Differentiation and integration functors between $\Mod(\uC_{\bullet}(G))$ and $\Rep(\TT \gfrak)$}\label{sec:2.6}
Let  $G$ be a simply connected Lie group with Lie algebra $\gfrak$. The main result of \cite{AriasAbad2019} states the existence of differentiation and integration functors between the monoidal categories $\Mod(\uC_{\sbullet}(G))$ and $\Rep(\TT \gfrak)$ which are inverses to one another. This extends the well-known correspondence between representations of $G$ and representations of $\gfrak$. Let us explain briefly the construction of these functors.  

We begin with the differentiation functor, which we will write as $\Dscr \colon \Mod(\uC_{\sbullet}(G)) \to \Rep(\TT \gfrak)$. For an element $x$ of $\gfrak$, take the singular $1$-simplex $\sigma[x] \colon \Delta_1 \to G$ defined by 
\begin{equation}
\sigma[x] (s) = \exp(s x),
\end{equation}
where $\exp \colon \gfrak \to G$ is the exponential map of $G$. Then, given an object $V$ in $\Mod(\uC_{\sbullet}(G))$ with structure homomorphism $\rho \colon \uC_{\bullet}(G) \to \End(V)$, the corresponding object $\Dscr(V)$ in $\Rep(\TT\gfrak)$ is $V$ with structure homomorphism $\Dscr(\rho) \colon \TT \gfrak \to \End(V)$ determined by 
\begin{align}
\begin{split}
\Dscr(\rho)(i(x)) &= \frac{d}{d t}\bigg\vert_{t=0} \rho\left(\sigma[tx]\right), \\
\Dscr(\rho)(L(x)) &=\frac{d}{d t}\bigg\vert_{t=0} \rho\left(\exp(tx)\right).
\end{split}
\end{align}
In addition, the functor $\Dscr$ acts as the identity on morphisms. Under this definition, it is not difficult to verify that $\Dscr$ is indeed monoidal. For details, see Theorem~3.3 of \cite{AriasAbad2019}.

Next we turn to the integration functor, which we write as $\Iscr \colon \Rep(\TT \gfrak) \to \Mod(\uC_{\sbullet}(G))$. First we need to record an important preliminary notion. In the category $\Rep(\TT\gfrak)$ let us fix an object $V$, and for each $k \geq 0$, let us call $\Phi^{(k)}_V \in \Omega^k(G) \otimes \End(V)^{-k}$ the unique left-equivariant form on $G$ with values in $\End(V)$ such that
\begin{equation}\label{eqn:2.33}
\Phi^{(k)}_V(e)(x_1,\dots,x_k) = i_{x_1} \circ \cdots \circ i_{x_k},
\end{equation}
for all $x_1,\dots,x_k \in \gfrak$. With this definition, it can be concluded that the forms $\Phi^{(k)}_V$ satisfy the ``descent equations''
\begin{equation}\label{eqn:2.34}
d \Phi^{(k)}_V = (-1)^k \delta \Phi^{(k+1)}_V,
\end{equation}
where, as before, we write $\delta$ for the the differential of $V$. Furthermore, if $\mu \colon G \times G \to G$ denotes the multiplication map for $G$ and $\pi_1,\pi_2 \colon G \times G \to G$ are the projection onto the first and second component, respectively, we get the relation
\begin{equation}\label{eqn:2.35}
\mu^{*} \Phi^{(k)}_V = \sum_{i + j = k}(-1)^{ij} \pi_1^* \Phi^{(i)}_V \wedge \pi_2^* \Phi^{(j)}_V. 
\end{equation}
We note finally that $\Phi^{(0)}_V \in \Omega^{0}(G) \otimes \End(V)^{0}$ is a representation of $G$ on $V$ and, in particular,
\begin{equation}\label{eqn:2.36}
\Phi^{(0)}_V(e) = \id_V
\end{equation}
In general, an element $\Phi$ of $\Omega^{\sbullet}(G) \otimes \End(V)$ is refereed to as a \emph{left-equivariant representation form} for $V$ if it can be decomposed as 
\begin{equation}
\Phi = \sum_{k \geq 0} \Phi^{(k)},
\end{equation}
where the forms $\Phi^{(k)} \in \Omega^k(G) \otimes \End(V)^{-k}$ satisfy the conditions \eqref{eqn:2.34}, \eqref{eqn:2.35} and \eqref{eqn:2.36}. In this way we set up a bijective correspondence between objects of $\Rep(\TT \gfrak)$ and their associated left-equivariant representation forms (see Proposition~3.18 of \cite{AriasAbad2019}). We can now define the integration functor $\Iscr \colon \Rep(\TT\gfrak) \to \Mod(\uC_{\sbullet}(G))$ as follows. Given an object $V$ in $\Rep(\TT \gfrak)$ with structure homomorphism $\rho \colon \TT \gfrak \to \End(V)$, the corresponding object $\Iscr(V)$ in $\Mod(\uC_{\sbullet}(G))$ is $V$ with structure homomorphism $\Iscr(\rho) \colon \uC_{\sbullet}(G) \to \End(V)$ determined on a singular $k$-simplex $\sigma \colon \Delta_k \to G$ by
\begin{equation}\label{eqn:2.38}
\Iscr(\rho)(\sigma) = \int_{\Delta_k} \sigma^*\Phi_{V},
\end{equation}
where $\Phi_{V} = \sum_{k \geq 0} \Phi^{(k)}_V$. Moreover, the functor $\Iscr$ acts  as the identity on morphisms. Under this definition, it is not hard to see that $\Dscr$ is simultaneously left and right inverse to $\Iscr$. All the details can be found in \S3.3 of \cite{AriasAbad2019}. 


\section{$\A_{\infty}$-quasi-isomorphisms of DG algebras}\label{sec:3}
In this section, we prove several technical results concerning the Van~Est map and the Hoschschild-De~Rham $\A_{\infty}$-quasi-isomorphism in the context of classifying spaces. These results are key components in the proof of our main theorem. They may also be of independent interest. Throughout the discussion, $G$ denotes a simply connected Lie group with Lie algebra $\gfrak$. 

\subsection{The Van~Est map}\label{sec:3.1}
Here we consider the Van~Est map from the  Bott-Shulman-Stasheff algebra $\Omega^\bullet(G_\bullet)$ to the Weil algebra of $\gfrak$. We follow the conventions  in \cite{Li-Bland-Meinrenken2015}. The Bott-Shulman-Stasheff algebra computes the cohomology of BG while the Van Est map is contractible, so the Van Est map models the pull-back map of the universal bundle.
Our goal here is to show that if  $G$ is connected and compact, there is a natural subalgebra of  $\Omega^\bullet(G_\bullet)$ such that the restriction of the Van Est map to it lands on the basic part of the Weil algebra and is a quasi-isomorphism.

Let us consider the universal $G$-bundle $\pi \colon EG_{\sbullet} \to BG_{\sbullet}$ as in \cite{Segal1968}. Recall that $EG_{\sbullet}$ is the simplicial manifold with $EG_{p} = G \times \cdots \times G$ ($p+1$ copies) where the face operators $\overline{\varepsilon}_{i} \colon EG_{p} \to EG_{p-1}$ are given by
\begin{equation}
\overline{\varepsilon}_{i}(g_0,\dots,g_{p}) = (g_0,\dots, g_{i-1},g_{i+1},\dots, g_{p}),
\end{equation}
for $0 \leq i \leq p$. Similarly, $BG_{\sbullet}$ is defined by $BG_{p} = G \times \cdots \times G$ ($p$ copies) but here the face operators $\varepsilon_{i} \colon BG_{p} \to BG_{p-1}$ are given by
\begin{equation}\label{eqn:3.2}
\varepsilon_i (g_1,\dots,g_p) = \begin{cases} (g_2,\dots,g_p) & \text{if $i = 0$,} \\
(g_1,\dots,g_{i-1}, g_ig_{i+1},g_{i+2},\dots, g_p) & \text{if $0 < i < p$,} \\
(g_1,\dots,g_{p-1}) & \text{if $i = p$.} \end{cases}
\end{equation}
Finally, view each $EG_{p}$ as a principal $G$-bundle over $BG_{p}$, with action the diagonal action of $G$ from the right, and quotient map $\pi \colon EG_{p} \to BG_{p}$ given by
\begin{equation}
\pi (g_0,\dots,g_p) = (g_0 g_1^{-1}, g_1 g_2^{-1},\dots, g_{p-1}g_{p}^{-1}). 
\end{equation}
By definition, the total space of the universal $G$-bundle $EG$ is the thick geometric realisation of the simplicial manifold $EG_{\sbullet}$. From this it is easy to see that the classifying space $BG$ may be identified with the thick geometric realisation of $BG_{\sbullet}$.  This means the cohomology of $BG$ may be computed as the ``De~Rham cohomology'' of $BG_{\sbullet}$, which is defined as the cohomology of the following double complex $\Omega^{\sbullet}(BG_{\sbullet})$:
$$
\xymatrix{\vdots & \vdots & \vdots &  \\
\Omega^2(BG_0) \ar[u]^-{\bar{d}}  \ar[r]^-{\partial}& \Omega^2(BG_1) \ar[u]^-{\bar{d}}  \ar[r]^-{\partial}& \Omega^2 (BG_2) \ar[u]^-{\bar{d}}  \ar[r]^-{\partial} & \cdots \\
\Omega^1(BG_0) \ar[u]^-{\bar{d}}  \ar[r]^-{\partial} & \Omega^1(BG_1) \ar[u]^-{\bar{d}}  \ar[r]^-{\partial}& \Omega^1 (BG_2) \ar[u]^-{\bar{d}}  \ar[r]^-{\partial} & \cdots \\
\Omega^0(BG_0) \ar[u]^-{\bar{d}}  \ar[r]^-{\partial} & \Omega^0(BG_1) \ar[u]^-{\bar{d}}  \ar[r]^-{\partial} & \Omega^0 (BG_2) \ar[u]^-{\bar{d}}  \ar[r]^-{\partial} & \cdots. \\
}
$$
Here the vertical differential $\bar{d} \colon \Omega^q(BG_p) \to \Omega^{q+1}(BG_p)$ is $(-1)^p$ times the usual de exterior differential $d$ and the horizontal differential $\partial \colon \Omega^q(BG_p) \to \Omega^q(BG_{p+1})$ is given by
\begin{equation}\label{eqn:3.4}
\partial = \sum_{i=0}^{p+1} (-1)^{i} \varepsilon_{i}^*.
\end{equation}
 We note that $\Omega^{\sbullet}(BG_{\sbullet})$ has a graded ring structure with respect to the cup product defined as follows. For any $\omega \in \Omega^{q}(BG_{p})$ and $\omega' \in \Omega^{q'}(BG_{p'})$, let $\omega \abxcup \omega' \in \Omega^{q + q'}(BG_{p + p'})$ be the differential form given by
\begin{equation}\label{eqn:3.5}
\omega \abxcup \omega' = (-1)^{q p'} \pr^* \omega \wedge \pr'^*\omega',
\end{equation}
where $\pr \colon BG_{p + p'} \to BG_{p}$ is the front face projection 
\begin{equation}\label{eqn:3.6}
\pr(g_1,\dots,g_{p+p'}) =(g_1,\dots,g_p) ,
\end{equation}
and $\pr' \colon BG_{p + p'} \to BG_{p'}$ is the back face projection
\begin{equation}\label{eqn:3.7}
\pr'(g_1,\dots,g_{p+p'}) =  (g_{p+1},\dots,g_{p+p'}) .
\end{equation}
Both the vertical and horizontal differentials $\bar{d}$ and $\partial$ are graded derivations relative to the cup product, and we regard Bott-Shulman-Stasheff complex $\Omega^{\sbullet}(BG_{\sbullet})$ as a DG algebra with respect to the total differential.  

Now we turn to a discussion of the Van~Est map. For this purpose, let us consider the action $\gamma_i (g)$ of elements $g$ of $G$ on $BG_{p}$ defined by
\begin{equation}\label{eqn:3.8}
\gamma_i(g) (g_1,\dots, g_p) = (g_1,\dots, g_{i-1},g_i g^{-1},g g_{i+1}, g_{i+2},\dots, g_p),
\end{equation}
where $1 \leq i \leq p$. For each $x \in \gfrak$, we denote by $x^{i,\sharp}$ the vector field on $BG_{p}$ generated by this action. We also regard the Weil algebra of $\gfrak$ as a bigraded algebra $\uW^{\sbullet,\sbullet} \gfrak$ with
\begin{equation}
\uW^{p,q}\gfrak = \Lambda^{p-q} \gfrak^* \otimes \uS^{q}\gfrak^*. 
\end{equation}
Notice that any $x \in \gfrak$ defines two kinds of contraction operators $i_{\Lambda}(x)$ and $i_{\uS}(x)$ on $\uW^{\sbullet,\sbullet} \gfrak$ of bidegrees $(-1,0)$ and $(-1,-1)$, corresponding to the contractions on $\Lambda^{\sbullet}\gfrak$ and $\uS^{\sbullet} \gfrak$, respectively. For elements $\xi \in \uW^{p,q}\gfrak$ and $x_1,\dots, x_p \in \gfrak$ we put
\begin{equation}
\xi(x_1,\dots,x_q, \overline{x}_{q+1},\dots, \overline{x}_p) = i_{\Lambda}(x_p) \cdots i_{\Lambda}(x_{q+1}) i_{\uS}(x_q) \cdots i_{\uS}(x_1) \xi. 
\end{equation}
With these definitions, the Van~Est map $\VE \colon \Omega^{\sbullet}(BG_{\sbullet}) \to \uW^{\sbullet,\sbullet} \gfrak$ is the map of DG algebras given by the following formula, for $\omega \in \Omega^{q}(BG_p)$ and $x_1,\dots, x_p \in \gfrak$,
\begin{equation}\label{eqn:3.11}
\VE(\omega) (x_1,\dots,x_q, \overline{x}_{q+1},\dots, \overline{x}_p) =\sum_{\sigma \in \mathfrak{S}_p} \varepsilon(\sigma) \left( i_{x^{1,\sharp}_{\sigma(1)}} \cdots i_{x^{q,\sharp}_{\sigma(q)}} L_{x^{q+1,\sharp}_{\sigma(q+1)}} \cdots L_{x^{p,\sharp}_{\sigma(p)}}  \omega \right)(e,\dots, e).
\end{equation}
Here $e$ is the identity element of $G$, and $\varepsilon(\sigma)$ is equal to $+1$ if the number of pairs $(i,j)$ with $q+1 \leq i < j \leq p$ but $\sigma(i) > \sigma(j)$ is even, and equal to $-1$ if that number is odd. \\

The Van~Est map  does not take values on the basic elements of $\uW^{\sbullet,\sbullet} \gfrak$. 
However, there is a subalgebra of $\Omega^\bullet(G_\bullet)$ whose image under the Van Est map consists of basic elements. We set $G_{p} = G \times \cdots \times G$ ($p$ copies) and think of it as a Lie group. Moreover, we let $\gfrak_p = \gfrak \oplus \cdots \oplus \gfrak$ ($p$ copies) denote the corresponding Lie algebra. Since the actions of $G$ on $BG_{p}$ defined by \eqref{eqn:3.8} commute, we obtain an action $\gamma(g_1,\dots,g_p)$ of elements $(g_1,\dots,g_p)$ of $G_p$ on $BG_{p}$ by putting
\begin{equation}\label{eqn:3.12}
\gamma(g_1,\dots,g_p) = \gamma_1(g_1) \circ \cdots \circ \gamma_p(g_p). 
\end{equation} 
It is straightforward to check this action is transitive and free. Let us denote by $\Omega^{q}(BG_{p})^{G_p}$  the subspace of $G_p$-invariant elements of $\Omega^{q}(BG_{p})$.\footnote{We remark that $\Omega^{\sbullet}(BG_{\sbullet})^{G_{\sbullet}}$ is not a DG subalgebra of $\Omega^{\sbullet}(BG_{\sbullet})$.} Then each element of $\Omega^{q}(BG_{p})^{G_p}$ is completely and freely determined by its evaluation at $(e,\dots,e) \in BG_{p}$. Therefore, evaluation at $(e,\dots,e)$ gives an isomorphism of graded vector spaces between $\Omega^{q}(BG_{p})^{G_p}$  and $\Lambda^{q} \gfrak_{p}^*$. On the other hand, consider the residual action $\gamma_0(g)$ of elements $g$ of $G$ on $BG_{p}$ defined by
\begin{equation}\label{eqn:3.13}
\gamma_0(g) (g_1,\dots, g_p) = (g g_1, g_2,\dots, g_p). 
\end{equation}
Since this action commutes with the one given by \eqref{eqn:3.12}, we end up with an action $\zeta(g_0,g_1,\dots,g_p)$ of elements $(g_0,g_1,\dots,g_p)$ of $G_{p+1}$ on $BG_{p}$ by setting
\begin{equation} \label{eqn:3.14}
\zeta(g_0,g_1,\dots,g_p) = \gamma_0 (g_0) \circ \gamma(g_1,\dots, g_p). 
\end{equation}
We let $\Omega^{q}(BG_{p})^{G_{p+1}}$ denote the subspace of $G_{p+1}$-invariant elements of $\Omega^{q}(BG_{p})$.

\begin{lemma}\label{lem:3.1}
$\Omega^{\sbullet}(BG_{\sbullet})^{G_{\ssbullet +1}}$ is a DG subalgebra of $\Omega^{\sbullet}(BG_{\sbullet})$ and the inclusion
$$
\Omega^{\sbullet}(BG_{\sbullet})^{G_{\ssbullet +1}} \longrightarrow \Omega^{\sbullet}(BG_{\sbullet})
$$
is a quasi-isomorphism. 
\end{lemma}

\begin{proof}
First, let us verify that $\Omega^{\sbullet}(BG_{\sbullet})^{G_{\ssbullet+1}}$ is a double subcomplex of $\Omega^{\sbullet}(BG_{\sbullet})$. By definition, it is clear that $\bar{d} \colon \Omega^{q}(BG_{p}) \to \Omega^{q+1}(BG_{p})$ preserves $G_{p+1}$-invariant elements, since the exterior differential (and hence $\bar{d}$) commutes with pullback. On the other hand, it is not hard to see that $\varepsilon_i^* \omega \in \Omega^{q}(BG_{p+1})^{G_{p+2}}$ for all $\omega \in \Omega^{q}(BG_{p})^{G_{p+1}}$. Thus, from \eqref{eqn:3.4}, we conclude that $\partial \colon \Omega^{q}(BG_{p}) \to \Omega^{q}(BG_{p+1})$ also preserves the $G_{p+1}$-invariant elements. 

Next, we need to verify that $\Omega^{\sbullet}(BG_{\sbullet})^{G_{\ssbullet+1}}$ is closed with respect to the cup product \eqref{eqn:3.5}. This turns out to be a direct consequence of the following two identities
\begin{align*}
\pr  \big( \zeta(g_{0},\dots, g_{p+p'}) \big)&= \zeta(g_{0},\dots, g_{p}), \\
\pr'  \big( \zeta(g_{0},\dots, g_{p+p'}) \big) &= \zeta(g_{p},\dots, g_{p + p'}), 
\end{align*}
which follow at once from the definitions  \eqref{eqn:3.6}, \eqref{eqn:3.7} and \eqref{eqn:3.14}. 

Finally, to prove the second statement, since $G_{p+1}$ is compact and connected, a theorem of Cartan \cite{Cartan1936} asserts that the inclusion $\Omega^{\sbullet}(BG_{p})^{G_{p+1}} \to \Omega^{\sbullet}(BG_{p})$ is a quasi-isomorphism. The result then follows from the convergence of the spectral sequences for $\Omega^{\sbullet}(BG_{\sbullet})^{G_{\ssbullet+1}}$ and $\Omega^{\sbullet}(BG_{\sbullet})$, together with the fact that the inclusion $\Omega^{\sbullet}(BG_{\sbullet})^{G_{\ssbullet+1}} \to \Omega^{\sbullet}(BG_{\sbullet})$ induces an isomorphism of spectral sequences on the $E_1$-term. 
\end{proof}

For our next preparatory result, we let $\Ad_g$ be the adjoint action of elements $g$ of $G$ on $\gfrak$ and denote by the same symbol its extension to $\gfrak_p$.  

\begin{lemma}\label{lem:3.2}
The following diagram commutes
\begin{equation*}
\xymatrix@C=7ex{\Omega^{q}(BG_{p})^{G_{p}} \ar[r]^-{\gamma_0(g)^*} \ar[d] & \Omega^{q}(BG_{p})^{G_{p}} \ar[d] \\
\Lambda^{q}\gfrak_p^* \ar[r]^-{\Ad_g^*} & \Lambda^{q}\gfrak_p^*,}
\end{equation*}
where the vertical arrows denote evaluation at the element $(e,\dots,e)$. 
\end{lemma}

\begin{proof}
Take $\omega \in \Omega^{q}(BG_{p})^{G_{p}}$ and $v_1,\dots,v_p \in \gfrak_p$. We compute directly, using the definitions:
\begin{align*}
\big( \gamma_0(g)^* & \omega  \big)_{(e,\dots,e)}(v_1,\dots, v_p) \\ &= \omega_{(g,e,\dots,e)} \big( d \gamma_0(g)_{(e,\dots,e)}(v_1), \dots,  d \gamma_0(g)_{(e,\dots,e)}(v_p)  \big) \\
&= \omega_{\gamma(g^{-1},\dots,g^{-1})(e,\dots, e)} \big( d \gamma_0(g)_{(e,\dots,e)}(v_1), \dots,  d \gamma_0(g)_{(e,\dots,e)}(v_p)  \big) \\
&= \big(\gamma(g^{-1},\dots,g^{-1})^*\omega\big)_{(e,\dots,e)} \big( d\gamma(g,\dots,g)_{(g,e,\dots,e)}d \gamma_0(g)_{(e,\dots,e)}(v_1), \\
& \qquad\qquad\qquad\qquad\qquad\qquad\quad  \dots,  d\gamma(g,\dots,g)_{(g,e,\dots,e)}d \gamma_0(g)_{(e,\dots,e)}(v_p) \big) \\
&= \omega_{(e,\dots,e)} \big( d(\gamma(g,\dots, g)\circ \gamma_0(g))_{(e,\dots, e)} (v_1), \dots,  d(\gamma(g,\dots, g)\circ \gamma_0(g))_{(e,\dots, e)} (v_p) \big) \\
&=  \omega_{(e,\dots,e)} \big( d \zeta(g,\dots,g)_{(e,\dots, e)} (v_1), \dots,  d \zeta(g,\dots,g)_{(e,\dots, e)} (v_p) \big).
\end{align*}
But
$$
\zeta(g,\dots,g)(g_1,\dots,g_p) = (g g_1 g^{-1},\dots, g g_p g^{-1}), 
$$
from which it follows that $d \zeta(g,\dots,g)_{(e,\dots, e)} = \Ad_g$. Substitution gives the result claimed. 
\end{proof}

Next, we record the following observation.

\begin{lemma}\label{lem:3.3}
The restriction of the Van~Est map $\VE$ to $\Omega^{q}(BG_{p})^{G_{p}}$ vanishes unless $q = p$. 
\end{lemma}

\begin{proof}
If $\omega \in \Omega^{q}(BG_{p})^{G_{p}}$, then $L_{x^{i,\sharp}} \omega = 0$ for all $x \in \gfrak$. This, together with formula \eqref{eqn:3.11}, implies that $\VE(\omega) = 0$ unless $q = p$. 
\end{proof}

As a consequence of this, we see that the restriction of the Van~Est map $\VE$ to $\Omega^{p}(BG_{p})^{G_{p}}$, which we keep on denoting by $\VE$, is given by the following expression, for $\omega \in \Omega^{p}(BG_{p})^{G_{p}}$ and $x_1,\dots, x_p \in \gfrak$, 
\begin{equation}\label{eqn:3.15}
\VE(\omega)(x_{1},\dots,x_p) = \sum_{\sigma \in \mathfrak{S}_p}  \left( i_{x^{1,\sharp}_{\sigma(1)}} \cdots i_{x^{p,\sharp}_{\sigma(p)}}  \omega \right)(e,\dots, e).
\end{equation}
We also note that this map has its image contained in $\uS^p \gfrak^*$. 

Before we can go further, we need the following piece of notation. For each $x \in \gfrak$, we let $x^{i}$ be element of $\gfrak_p$ having its $i$th and $(i+1)$th coordinates equal to $-x$ and $x$, respectively, and all others zero. Hence, by definition, $x^{i,\sharp}(e,\dots,e) = x^{i}$. Thus, if we let $\widetilde{\VE} \colon \Lambda^p \gfrak_p^* \to \uS^p \gfrak$ be the map defined for $\xi \in \Lambda^p \gfrak_p^*$ and $x_1,\dots, x_p \in \gfrak$ by
\begin{equation}\label{eqn:3.16}
\widetilde{\VE} (\xi)(x_{1},\dots,x_p) = \sum_{\sigma \in \mathfrak{S}_p} \xi(x_{\sigma(1)}^{1},\dots, x_{\sigma(p)}^{p}), 
\end{equation}
we obtain the commutative diagram
\begin{equation*}
\xymatrix{\Omega^{p}(BG_{p})^{G_{p}} \ar[dr]^-{\VE} \ar[d]&  \\
\Lambda^{p}\gfrak_{p}^* \ar[r]_-{\widetilde{\VE}} & \uS^{p}\gfrak^*,}
\end{equation*}
where, as before, the vertical arrow denotes evaluation at $(e,\dots, e)$. We may now state and prove the following result.

\begin{proposition}\label{prop:3.4}
The restriction of the Van~Est map $\VE$ to $\Omega^{\sbullet}(BG_{\sbullet})^{G_{\ssbullet + 1}}$ has image contained in $(\uS^{\sbullet}\gfrak^*)^G$. 
\end{proposition}

\begin{proof}
By virtue of Lemma~\ref{lem:3.2} and the previous remarks, it is enough to show that the following diagram commutes
\begin{equation*}
\xymatrix@C=7ex{\Lambda^{p}\gfrak_{p}^* \ar[r]^-{\Ad_g^*} \ar[d]_-{\widetilde{\VE}} & \Lambda^{p}\gfrak_{p}^* \ar[d]^-{\widetilde{\VE}}\\
\uS^p \gfrak^* \ar[r]^-{\Ad_g^*} & \uS^p \gfrak^*.}
\end{equation*}
So let us take $\xi \in \Lambda^{p}\gfrak_{p}^*$ and $x_1, \dots , x_p \in \gfrak$. Then, attending to the definition \eqref{eqn:3.16}, we have
\begin{align*}
\widetilde{\VE}(\Ad_g^* \xi)(x_1,\dots,x_p) &= \sum_{\sigma \in \mathfrak{S}_p}  (\Ad_g^* \xi)(x_{\sigma(1)}^{1},\dots, x_{\sigma(p)}^{p}) \\
&=  \sum_{\sigma \in \mathfrak{S}_p}  \xi \Big(\!\Ad_g x_{\sigma(1)}^{1},\dots, \Ad_g x_{\sigma(p)}^{p}\Big) \\
&= \widetilde{\VE}(\xi) \Big(\Ad_g x_{\sigma(1)}^{1}, \dots,\Ad_g x_{\sigma(p)}^{p} \Big) \\
&= \big( \!\Ad_g^*\widetilde{\VE}(\xi)\big)(x_1,\dots,x_p),
\end{align*}
from which the result follows. 
\end{proof}

Next, we will show that the restricted Van~Est map $\VE \colon \Omega^{\sbullet}(BG_{\sbullet})^{G_{\ssbullet + 1}} \to (\uS^{\sbullet}\gfrak^*)^G$ is a quasi-isomorphism. For this, we need a small digression outlining some of the results of \cite{Alekseev-Meinrenken2005}. 

To begin with, recall that a $\gfrak$-\emph{DG algebra} $A$ is by definition an object of $\Rep(\TT \gfrak)$ endowed with the structure of a graded ring such that the action of $\TT \gfrak$ is by derivations. Homomorphisms of $\gfrak$-DG algebras are morphisms in $\Rep(\TT \gfrak)$ which are also homomorphisms of graded rings. Given a $\gfrak$-DG algebra $A$, an \emph{algebraic connection} is a linear map $\theta \colon \gfrak^* \to A^{1}$, which satisfy the relations
\begin{align}
\begin{split}
i_x (\theta(\xi)) &= \langle \xi, x \rangle,\\
L_x (\theta(\xi)) &=\theta (\ad_x^*\xi),
\end{split}
\end{align}
for all $x \in \gfrak$ and $\xi \in \gfrak*$. One important example of a commutative $\gfrak$-DG algebra is provided by the Weil algebra $\uW \gfrak$. It is obvious that $\uW\gfrak$ carries a ``tautological'' connection given by the map $\iota \colon \gfrak^* \to \uW^1\gfrak$. As a matter of fact, $\uW \gfrak$ is universal among commutative $\gfrak$-DG algebras with connection. Thus, given a $\gfrak$-DG algebra $A$ with connection $\theta$, there exists a $\gfrak$-DG algebra homomorphism $c^{\theta}\colon \uW \gfrak \to A$ such that $c^{\theta} \circ \iota = \theta$. Following the terminology of \cite{Alekseev-Meinrenken2005}, one refers to $c^{\theta}$ as the \emph{characteristic homomorphism} for the connection $\theta$. 

Our interest here, however, is on the De~Rham complex $\Omega^{\sbullet}(EG_{\sbullet})$ of the simplicial manifold $EG_{\sbullet}$, which is defined by exactly the same prescription that defined the Bott-Shulman-Stasheff complex $\Omega^{\sbullet}(BG_{\sbullet})$. This turns out to be a noncommutative $\gfrak$-DG algebra where the graded ring structure is again defined by the cup product, and, if we let $\rho$ denote the infinitesimal action of $\gfrak$ on $EG_{\sbullet}$, $i_x$ is the inner product of a form with $\rho(x)$, and $L_x$ is the Lie derivative of the form along $\rho(x)$. What is more, it carries a natural connection $\theta \colon \gfrak^* \to \Omega^1(EG_0)$ given by the left-invariant Maurer-Cartan form on $G$. We would like to define a characteristic homomorphism for this connection $\theta$. For this we need a universal object among noncommutative $\gfrak$-DG algebras with connection, the so called \emph{noncommutative Weil algebra} $\widetilde{\uW}\gfrak$. Its definition is as follows. 

Recall that the Weil algebra $\uW \gfrak$ may be identified with the Koszul algebra of the graded vector space $\uu \gfrak^*$. Accordingly, as a DG algebra, $\widetilde{\uW}\gfrak$ is the noncommutative Koszul algebra of $\uu \gfrak^*$. Just as in Section~\ref{sec:2.4}, we associate to each $\xi \in \gfrak^*$ a degree $1$ generator $t(\xi)$ and a degree $2$ generator $w(\xi)$, so that $\widetilde{\uW}\gfrak$ is freely generated by $t(\xi)$ and $w(\xi)$, $d_{\widetilde{\uW}} t(\xi) = w(\xi)$ and $d_{\widetilde{\uW}} w(\xi) = 0$. The formulas for the contractions $i_x$ and Lie derivatives $L_x$ are given on these generators by
\begin{align}
\begin{split}
i_x (t(\xi)) &= \langle t(\xi), x \rangle, \\
i_x (w(\xi)) &= \ad_x^* (t(\xi)), \\
L_x (t(\xi)) &= \ad_x^* (t(\xi)), \\
L_x (w(\xi)) &= \ad_x^* (w(\xi)). 
\end{split}
\end{align}
And just as in the commutative case, $\widetilde{\uW}\gfrak$ carries a ``tautological'' connection determined by the map $\widetilde{\iota}  :\gfrak^* \to \widetilde{\uW}^1\gfrak$. It can then be shown that, given an an arbitrary $\gfrak$-DG algebra $A$ with connection $\theta$, there is a $\gfrak$-DG algebra homomorphism $\widetilde{c}{}^{\,\theta} \colon \widetilde{\uW}\gfrak \to A$ such that $\widetilde{c}{}^{\,\theta} \circ \widetilde{\iota} = \theta$. We should also point out that the quotient map $\widetilde{\uW} \gfrak \to \uW \gfrak$ is a morphism in $\Rep(\TT\gfrak)$ which is a quasi-isomorphism with homotopy inverse given by symmetrisation $\sym \colon \uW \gfrak \to \widetilde{\uW}\gfrak$.

In light of the preceding discussion it is now clear that there is a Chern-Weil map
$$
c^{\theta} = \widetilde{c}{}^{\,\theta} \circ \sym \colon \uW \gfrak \longrightarrow \Omega^{\sbullet}(EG_{\sbullet}),
$$
which is defined by symmetrisation. This in turn induces a morphism of cochain complexes on the basic subspaces $c^{\theta} \colon (\uW \gfrak)_{\bas} \to \Omega^{\sbullet}(EG_{\sbullet})_{\bas}$. As $\uS^{\sbullet} \gfrak^*$ is precisely the set of elements in $\uW \gfrak$ killed by $i_x$ for $x \in \gfrak$, it follows that $(\uW \gfrak)_{\bas}$ coincides with the algebra of invariant polynomials $(\uS^{\sbullet} \gfrak^*)^G$. On the target complex we have we have on the other hand that $\Omega^{\sbullet}(EG_{\sbullet})_{\bas}$ is canonically isomorphic to the Bott-Shulman-Stasheff complex $\Omega^{\sbullet}(BG_{\sbullet})$. This latter isomorphism is induced by pullback along the right inverse $\iota \colon BG_p \to EG_p$ to the quotient map $\pi$ which is defined by the formula
\begin{equation}
\iota(g_1,\dots, g_1) = (e, g_1^{-1},  \dots, (g_1 \cdots g_p)^{-1}).
\end{equation} 
Therefore we clearly get a map
$$
\AM^{\theta} = \iota^* \circ c^{\theta} \colon (\uS^{\sbullet} \gfrak^*)^G \longrightarrow \Omega^{\sbullet}(BG_{\sbullet}),
$$
to which we refer to as the \emph{Alekseev-Meinrenken map}. The image of an invariant polynomial of degree $r$ under this map has non-vanishing components only in bidegree $p + q = 2r$ with $p \leq r$. It also induces an algebra homomorphism in cohomology, and in fact an algebra isomorphism if $G$ is compact  and connected (see Proposition~9.1 and Theorem~9.2 of \cite{Alekseev-Meinrenken2005}). 

To proceed further, let us consider the action $\overline{\gamma}_i(g)$ of elements $g$ of $G$ on $EG_p$ defined by
\begin{equation}\label{eqn:3.20a}
\overline{\gamma}_i(g)(g_0,\dots, g_p) = (g_0,\dots,g_{i-1} , g g_{i}, g_{i+1}, \dots, g_p),
\end{equation}
where $0 \leq i \leq p$. It is then a simple matter to verify that all of these actions provide lifts of the actions of $G$ on $BG_p$ determined by \eqref{eqn:3.8} and \eqref{eqn:3.13}. To be more precise, we have a commutative diagram
$$
\xymatrix{EG_p \ar[r]^-{\overline{\gamma}_i(g)} \ar[d]_-{\pi} & EG_p \ar[d]^-{\pi} \\
BG_p \ar[r]^-{\gamma_i(g)} & BG_p, }
$$
for all $0 \leq i \leq p$. This implies that if, for each $x \in \gfrak$, we let $\overline{x}^{i,\sharp}$ denote the vector field on $EG_p$ generated by the action \eqref{eqn:3.20a}, then $\overline{x}^{i,\sharp}$ and $x^{i,\sharp}$ are $\pi$-related. In particular, we have
\begin{equation}\label{eqn:3.21a}
d \pi_{(e,\dots,e)} (\overline{x}^{i,\sharp}(e,\dots,e)) = x^{i,\sharp}(e,\dots,e). 
\end{equation}
It is also worth pointing out that we get an action $\overline{\zeta}(g_0,\dots,g_p)$ of elements $(g_0,\dots,g_p)$ of $G_{p+1}$ on $EG_{p}$ by simply putting
\begin{equation}
\overline{\zeta}(g_0,\dots,g_p) = \overline{\gamma}_0(g_0) \circ \cdots \circ \overline{\gamma}_p(g_p),
\end{equation}
and that this action provides a lift of the action of $G_{p+1}$ on $BG_p$ defined by \eqref{eqn:3.14}. Let $\Omega^q(EG_p)^{G_{p+1}}$ denote the subspace of $G_{p+1}$-invariant elements of $\Omega^q(EG_p)$. By precisely the same argument as that used to prove Lemma~\ref{lem:3.1}, we have the following. 

\begin{lemma}
$\Omega^{\sbullet}(EG_{\sbullet})^{G_{\ssbullet+1}}$ is a DG subalgebra of $\Omega^{\sbullet}(EG_{\sbullet})$ and the inclusion
$$
\Omega^{\sbullet}(EG_{\sbullet})^{G_{\ssbullet+1}} \longrightarrow \Omega^{\sbullet}(EG_{\sbullet})
$$
is a quasi-isomorphism. 
\end{lemma}

The discussion in the previous paragraphs also yield the following result. 

\begin{proposition}
The Alekseev-Meinrenken map $\AM^{\theta}$ has image contained in $\Omega^{\sbullet}(BG_{\sbullet})^{G_{\ssbullet+1}}$. 
\end{proposition}

\begin{proof}
Since the Maurer-Cartan form on $G$ is left-invariant, the restriction of the Chern-Weil map $c^{\theta}$ to $(\uS^{\sbullet}\gfrak^*)^G$ has its image contained in $\Omega^{\sbullet}(EG_{\sbullet})^{G_{\ssbullet+1}}_{\bas}$. The result is thus a direct consequence of the fact that the action of $G_{p+1}$ on $EG_p$ is a lifting of the action of $G_{p+1}$ on $BG_p$. 
\end{proof}

With all of the above ingredients in place, we now let $\widehat{c}{}^{\,\theta}$ be the Chen-Weil map $c^{\theta}$ seen as a map taking values in $\Omega^{p}(EG_{p})^{G_{p+1}}_{\bas}$. We set accordingly $\widehat{\AM}{}^{\theta} = \iota^* \circ \widehat{c}{}^{\,\theta}$ and notice that $\widehat{\AM}{}^{\theta}$ is nothing but the Alekseev-Meinrenken map $\AM^{\theta}$ seen as taking values in $\Omega^{\sbullet}(BG_{\sbullet})^{G_{\ssbullet +1}} $. 

\begin{theorem}\label{thm:3.7}
The map $\widehat{\AM}{}^{\theta} \colon (\uS^{\sbullet}\gfrak^*)^G \to  \Omega^{\sbullet}(BG_{\sbullet})^{G_{\ssbullet +1}}$ is a left inverse of the Van~Est map $\VE \colon \Omega^{\sbullet}(BG_{\sbullet})^{G_{\ssbullet +1}} \to (\uS^{\sbullet}\gfrak^*)^G$. 
\end{theorem}

\begin{proof}
We will first write down an explicit formula for the map $\widehat{\AM}{}^{\theta}$. To that end, we fix a basis $e_a$ of $\gfrak$ with dual basis $e^{a}$ and recall from Section~\ref{sec:2.4} that we have set $w^{a} = w(e^{a})$, so that $\uS^{\sbullet}\gfrak$ can be identified with the polynomial algebra in these variables. We also set $\theta^{a} = \theta(e^{a})$. Notice that $\theta^{a}$ lives in bidegree $(0,1)$, $d \theta^{a}$ lives in bidegree $(0,2)$ and $\partial \theta^{a}$ lives in bidegree $(1,1)$. It follows that the image of $w^{a}$ under $\widehat{c}{}^{\,\theta}$ is $\partial \theta^{a}$. Therefore, if we pick a monomial $w^{a_1} \cdots w^{a_p}$ in $\uS^p \gfrak^*$, we get
\begin{equation*}
\widehat{c}{}^{\,\theta} (w^{a_1} \cdots w^{a_p}) = \frac{1}{p!} \sum_{\sigma \in \mathfrak{S}_p} \partial \theta^{a_{\sigma(1)}} \abxcup \cdots \abxcup \partial \theta^{a_{\sigma(p)}}. 
\end{equation*}
and, consequently,
\begin{equation}\label{eqn:3.22a}
\widehat{\AM}{}^{\theta}(w^{a_1} \cdots w^{a_p}) = \frac{1}{p!} \sum_{\sigma \in \mathfrak{S}_p} \iota^*(\partial \theta^{a_{\sigma(1)}} \abxcup \cdots \abxcup \partial \theta^{a_{\sigma(p)}}). 
\end{equation}
Next, let us determine $\VE\big(\widehat{\AM}{}^{\theta}(w^{a_1} \cdots w^{a_p})\big)$. To start, we fix $x_1,\dots, x_p \in \gfrak$. By the definition in \eqref{eqn:3.15}, and recalling that $x^{i,\sharp}_{\sigma(i)}(e,\dots,e) = x^{i}_{\sigma(i)}$ for all $1 \leq i \leq p$, we have
\begin{align*}
\VE\big(\widehat{\AM}{}^{\theta}(w^{a_1} \cdots w^{a_p})\big)(x_1,\dots,x_p) = \sum_{\sigma' \in \mathfrak{S}_p} \widehat{\AM}{}^{\theta}(w^{a_1} \cdots w^{a_p})_{(e,\dots,e)} (x^{1}_{\sigma'(1)}, \dots, x^{p}_{\sigma'(p)}).
\end{align*}
Upon using \eqref{eqn:3.22a}, this becomes
\begin{align} \label{eqn:3.24a}
\begin{split}
\VE\big(&\widehat{\AM}{}^{\theta}(w^{a_1} \cdots w^{a_p})\big)(x_1,\dots,x_p) \\
&=  \sum_{\sigma' \in \mathfrak{S}_p}\frac{1}{p!} \sum_{\sigma \in \mathfrak{S}_p} (\partial \theta^{a_{\sigma(1)}} \abxcup \cdots \abxcup \partial \theta^{a_{\sigma(p)}})_{(e,\dots,e)} \big( d\iota_{(e,\dots,e)}(x^{1}_{\sigma'(1)}),\dots, d\iota_{(e,\dots,e)}(x^{p}_{\sigma'(p)}) \big). 
\end{split}
\end{align}
Let us evaluate each of the terms inside the double sum. Firstly, attending to the definition of the cup product \eqref{eqn:3.5}, one easily verifies that
\begin{equation}\label{eqn:3.25a}
\partial \theta^{a_{\sigma(1)}} \abxcup \cdots \abxcup \partial \theta^{a_{\sigma(p)}} = \pi_{1,2}^* \partial \theta^{a_{\sigma(1)}} \wedge \cdots \wedge \pi_{p,p+1}^* \partial \theta^{a_{\sigma(p)}}, 
\end{equation}
where $\pi_{i,i+1} \colon EG_{p} \to EG_1$ is the projection onto the $i$th and $(i+1)$th factors with $1 \leq i \leq p$. Secondly, by virtue of \eqref{eqn:3.21a}, 
\begin{equation}\label{eqn:3.26a}
d\iota_{(e,\dots,e)}(x^{i}_{\sigma'(i)}) = \overline{x}^{i,\sharp}_{\sigma'(i)}(e,\dots,e),
\end{equation}
for all $1 \leq i \leq p$. Putting together \eqref{eqn:3.25a} and \eqref{eqn:3.25a}, we thus find
\begin{align}\label{eqn:3:27a}
\begin{split}
(\partial \theta^{a_{\sigma(1)}} &\abxcup \cdots \abxcup \partial \theta^{a_{\sigma(p)}})_{(e,\dots,e)} \big( d\iota_{(e,\dots,e)}(x^{1}_{\sigma'(1)}),\dots, d\iota_{(e,\dots,e)}(x^{p}_{\sigma'(p)}) \big) \\
&= \sum_{\sigma'' \in \mathfrak{S}_p} \mathrm{sgn}(\sigma'') \prod_{i=1}^{p} (\partial \theta^{a_{\sigma(\sigma''(i))}})_{(e,e)} \left((d \pi_{\sigma''(i),\sigma''(i)+1})_{(e,\dots,e)}\big(\overline{x}^{i,\sharp}_{\sigma'(i)}(e,\dots,e)\big) \right),
\end{split}
\end{align}
where $\mathrm{sgn}(\sigma'')$ denotes the sign of the permutation $\sigma''$. Next notice that $\overline{x}^{i,\sharp}_{\sigma'(i)}(e,\dots,e)$ is the element of $\gfrak_{p+1}$ having its $i$th coordinate equal to $x_{\sigma'(i)}$ and all others zero. Consequently, the only non-zero contribution to the sum in \eqref{eqn:3:27a} comes from the identity permutation. Also, it is straightforward to calculate that
\begin{equation*}
 (\partial \theta^{a_{\sigma(i)}})_{(e,e)} \left((d \pi_{i,i+1})_{(e,\dots,e)}\big(\overline{x}^{i,\sharp}_{\sigma'(i)}(e,\dots,e)\big) \right) = \theta^{a_{\sigma(i)}}_e (x_{\sigma'(i)}) = w^{a_{\sigma(i)}}(x_{\sigma'(i)}). 
\end{equation*}
In this way, \eqref{eqn:3:27a} becomes
\begin{equation*}
(\partial \theta^{a_{\sigma(1)}} \abxcup \cdots \abxcup \partial \theta^{a_{\sigma(p)}})_{(e,\dots,e)} \big( d\iota_{(e,\dots,e)}(x^{1}_{\sigma'(1)}),\dots, d\iota_{(e,\dots,e)}(x^{p}_{\sigma'(p)}) \big) = \prod_{i=1}^{p}  w^{a_{\sigma(i)}} (x_{\sigma'(i)}).
\end{equation*}
Inserting this back in \eqref{eqn:3.24a} gives
\begin{align*}
\VE\big(\widehat{\AM}{}^{\theta}(w^{a_1} \cdots w^{a_p})\big)(x_1,\dots,x_p) = \sum_{\sigma \in \mathfrak{S}_p} \frac{1}{p!} \sum_{\sigma' \in \mathfrak{S}_p} \prod_{i=1}^{p}  w^{a_{\sigma(i)}} (x_{\sigma'(i)}) = \sum_{\sigma \in \mathfrak{S}_p} \prod_{i=1}^{p} w^{a_{\sigma(i)}} (x_{i}). 
\end{align*}
This allows us to conclude that
\begin{equation*}
\VE\big(\widehat{\AM}{}^{\theta}(w^{a_1} \cdots w^{a_p})\big) = w^{a_1} \cdots w^{a_p}. 
\end{equation*}
Since any element of $(\uS^{p}\gfrak)^G$ is a linear combination of monomials $w^{a_1} \cdots w^{a_p}$ in $\uS^p\gfrak^*$,  conclusion follows at once.
\end{proof}

Combining the previous result with the above remarks immediately yields the following. 

\begin{corollary}
The restricted Van~Est map $\VE \colon \Omega^{\sbullet}(BG_{\sbullet})^{G_{\ssbullet +1}} \to (\uS^{\sbullet}\gfrak^*)^G$ is a quasi-isomorphism. 
\end{corollary}


\subsection{The De~Rham $\A_{\infty}$-quasi-isomorphism for classifying spaces}
In this subsection we establish a version of Gugenheim's $\A_{\infty}$ De~Rham theorem for the classifying space $BG$. We shall start with some general considerations concerning the totalisation of semi-cosimplicial DG algebras. 

For any positive integer $n$, let $[n]$ denote the set $\{0,1,\dots,n\}$. We then consider, for $p + q \leq n$, the map $l_{p,q}^{n} \colon [p] \to [n]$ defined by 
\begin{equation}\label{eqn:3.17}
l_{p,q}^{n}(k) = k + q. 
\end{equation}
Notice that hese maps satisfy the relations \[l_{p,q}^{n} \circ l_{p',q'}^{n'} = l_{p',q+q'}^{n}.\]

Let $A_{\sbullet} = \{A_{p}\}_{p \geq 0}$ be a semi-cosimplicial DG algebra with coface maps $\partial'_{i} \colon A_{p-1} \to A_{p}$ for $0 \leq i \leq p$. For $p \geq 0$ fixed, we write $A_{p} = \bigoplus_{q \in \ZZ} A_{p}^{q}$ for the underlying graded decomposition. Associated to $A_{\sbullet}$, there is a canonical DG algebra $\Tot(A_{\sbullet})$, constructed as follows. As a graded vector space, its $n$th degree summand is defined as
\begin{equation*}
\Tot(A_{\sbullet})^{n} = \bigoplus_{p + q = n} A_{p}^{q}. 
\end{equation*}
This becomes a cochain complex if we set $\partial = \partial' + \partial'' \colon \Tot(A_{\sbullet})^{n} \to \Tot(A_{\sbullet})^{n + 1}$, where the differential $\partial' \colon A_{p}^{q} \to A_{p+1}^{q}$ is the alternating sum 
\begin{equation}\label{eqn:3.18}
\partial' = \sum_{i=0}^{p+1} (-1)^{i} \partial'_{i},
\end{equation}
and the differential $\partial'' \colon A^{q}_{p} \to A^{q+1}_{p}$ is $(-1)^{p}$ times the differential of $A_p$. To define the product on $\Tot(A_{\sbullet})$, we take we take the map induced on $A_{\sbullet}$ by the map $l_{p,q}^{n}$ given in \eqref{eqn:3.17}, which by abuse of notation we also call $l_{p,q}^{n}$. We then have $l_{p,q}^{n} \colon A_{p} \to A_{n}$ for $p + q \leq n$. For any $a \in A_{p}^{q}$ and $a' \in A_{p'}^{q'}$, we let $a a' \in A_{p + p'}^{q + q'}$ be the element defined as
\begin{equation} \label{eqn:3.19}
a a' = (-1)^{q p'} l_{p,0}^{p + p'}(a) l_{p',p}^{p + p'}(a').
\end{equation}
With these operations, one can verify that $\Tot(A_{\sbullet})$ is in fact a DG algebra. We omit the details, but comment that this depends upon the fact that the maps $l_{p,q}^{n}$ satisfy the following relations:
 \begin{equation}\label{eqn:3.20}
 \partial'_i \circ l_{p,q}^{n} = \begin{cases} l_{p,q}^{n+1} & \text{if $i > p + q$,} \\
 l_{p+1,q}^{n+1} \circ \partial'_{i-q} & \text{if $q < i \leq p + q$,} \\
 l_{p,q+1}^{n+1} & \text{if $i \leq q$.} \end{cases}
 \end{equation}
More importantly for our purposes, the construction of $\Tot(A_{\sbullet})$ gives the following result.

\begin{proposition}\label{prop:3.5}
The assignment $A_{\sbullet} \mapsto \Tot(A_{\sbullet})$ defines a functor from the category of semi-cosimplicial DG algebras with semi-cosimplicial $\A_{\infty}$-morphisms to the category of DG algebras with $\A_{\infty}$-morphisms. 
\end{proposition}

\begin{proof}
To start with, recall that a DG algebra $A$ with differential $\partial$ can be thought of as an $\A_{\infty}$-algebra with $\A_{\infty}$-operations $d \colon \uu A \to \uu A$ and $m \colon \uu A \otimes \uu A \to \uu A$ defined by declaring
\begin{align} \label{eqn:3.21}
\begin{split}
d(\uu a) &= \uu \partial a, \\
m(\uu a \otimes \uu a') &= (-1)^{\vert a \vert + 1} \uu (a a'), 
\end{split}
\end{align}
for homogenoeus elements $a, a' \in A$. Next, let us introduce some notation to facilitate the presentation. Given a semi-cosimplicial DG algebra $A_{\sbullet} = \{A_{p}\}_{p \geq 0}$, when we write $a \in A_{p}^{q}$ we mean that $\vert a \vert = p + q$;  on the other hand, if we write $\overline{a} \in A_{p}^{q}$, we mean that $\vert \overline{a} \vert = q$. For $p \geq 0$ fixed, we also denote the $\A_{\infty}$-operations associated to $A_p$ in accord to \eqref{eqn:3.21} by $d''_p$ and $m_p$, respectively. With this notation, and the definitions \eqref{eqn:3.18} and \eqref{eqn:3.19}, we can view $\Tot(A_{\sbullet})$ as an $\A_{\infty}$-algebra by setting $d = d' + d'' \colon \uu\!\Tot(A_{\sbullet}) \to \uu\!\Tot(A_{\sbullet})$, where
\begin{align}\label{eqn:3.22}
\begin{split}
d'(\uu a) &=  \sum_{i =0}^{p+1} (-1)^{i} \uu \partial'_i \overline{a}, \\
d''(\uu a) &=   (-1)^{p} d''_p (\uu \overline{a}),
\end{split}
\end{align}
and $m \colon  \uu\!\Tot(A_{\sbullet}) \otimes  \uu\!\Tot(A_{\sbullet}) \to \uu\!\Tot(A_{\sbullet})$ to be given by
\begin{equation}\label{eqn:3.23}
m(\uu a \otimes \uu a') = (-1)^{p + qp'} m_{p+p'}\big(\uu  l_{p,0}^{p+p'}(\overline{a}), \uu l_{p',0}^{p+p'}(\overline{a}')\big),
\end{equation}
for any $a \in A_{p}^{q}$ and $a' \in A_{p'}^{q'}$.

We wish to show that the assignment $A_{\sbullet} \mapsto \Tot(A_{\sbullet})$ is functorial with respect to semi-cosimplicial $\A_{\infty}$-morphisms. So let $\phi_{\sbullet} \colon A_{\sbullet} \to B_{\sbullet}$ be one such $\A_{\infty}$-morphisms. This means that for each $p \geq 0$ we have an $\A_{\infty}$-morphism $\phi_{p} \colon A_{p} \to B_{p}$, and these commute with the coface maps of $A_{\sbullet}$ and $B_{\sbullet}$. The map $\Tot(\phi_{\sbullet}) \colon  \Tot(A_{\sbullet}) \to  \Tot(B_{\sbullet})$ is explicitly given as follows. Let $a_1 \in A_{p_1}^{q_1},\dots, a_n \in A_{p_n}^{q_n}$, and put $p = \sum_{j= 1}^{n} p_j$ and $r_i = \sum_{j= 1}^{i-1} p_j$. Then
\begin{equation}\label{eqn:3.24}
\Tot(\phi_{\sbullet})_n (\uu a_1 \otimes \cdots \otimes \uu a_n) = (-1)^{\sum_{1 \leq i < j \leq n}p_j (q_i + 1)}\phi_{p,n} \big( \uu l_{p_1,r_1}^{p}(\overline{a}_1) \otimes \cdots \otimes  \uu l_{p_n,r_n}^{p}(\overline{a}_n) \big).  
\end{equation}
We claim that $\Tot(\phi_{\sbullet})$ satisfies the required relations to be an $\A_{\infty}$-morphism, which read
\begin{align}\label{eqn:3.25}
\begin{split}
&d \circ  \Tot(\phi_{\sbullet})_n + \sum_{i + j = n} m \circ (\Tot(\phi_{\sbullet})_i \otimes \Tot(\phi_{\sbullet})_j) \\
&\qquad = \sum_{i+j+1=n} \Tot(\phi_{\sbullet})_n \circ (\id^{\otimes i} \otimes d \otimes \id^{\otimes j})  + \sum_{i+j+2=n} \Tot(\phi_{\sbullet})_{n-1} \circ (\id^{\otimes i} \otimes m \otimes \id^{\otimes j}). 
\end{split}
\end{align}
To verify the claim, we note that to prove \eqref{eqn:3.25} it is enough to prove that
\begin{equation}\label{eqn:3.26}
d' \circ \Tot(\phi_{\sbullet})_n  = \sum_{i+j+1=n} \Tot(\phi_{\sbullet})_n \circ (\id^{\otimes i} \otimes d' \otimes \id^{\otimes j}),
\end{equation}
and 
\begin{align}\label{eqn:3.27}
\begin{split}
&d'' \circ  \Tot(\phi_{\sbullet})_n + \sum_{i + j = n} m \circ (\Tot(\phi_{\sbullet})_i \otimes \Tot(\phi_{\sbullet})_j) \\
&\qquad = \sum_{i+j+1=n} \Tot(\phi_{\sbullet})_n \circ (\id^{\otimes i} \otimes d'' \otimes \id^{\otimes j})  + \sum_{i+j+2=n} \Tot(\phi_{\sbullet})_{n-1} \circ (\id^{\otimes i} \otimes m \otimes \id^{\otimes j}),
\end{split}
\end{align}
separately.

We begin with \eqref{eqn:3.26}. To that end, we write $\widetilde{\partial}'_i = \uu \circ \partial'_i \circ \us$ and set $s = \sum_{1 \leq i < j \leq n}p_j (q_i + 1)$. From \eqref{eqn:3.18} and the fact that $\phi_{p}$ commutes with the $\widetilde{\partial}'_{i}$, it follows that
\begin{gather*}
\begin{align*}
d' \left( \Tot(\phi_{\sbullet})_n (\uu a_1 \otimes \cdots \otimes \uu a_n) \right)  & = \sum_{i = 0}^{p+1} (-1)^{s + i} \widetilde{\partial}'_i \left(  \phi_{p,n} \big( \uu l_{p_1,r_1}^{p}(\overline{a}_1) \otimes \cdots \otimes  \uu l_{p_n,r_n}^{p}(\overline{a}_n) \big) \right) \\
&=(-1)^{s} \widetilde{\partial}'_{0} \left(  \phi_{p,n} \big( \uu l_{p_1,r_1}^{p}(\overline{a}_1) \otimes \cdots \otimes  \uu l_{p_n,r_n}^{p}(\overline{a}_n) \big) \right) \\
&\quad+ (-1)^{s + p +1} \widetilde{\partial}'_{p+1} \left(  \phi_{p,n} \big( \uu l_{p_1,r_1}^{p}(\overline{a}_1) \otimes \cdots \otimes  \uu l_{p_n,r_n}^{p}(\overline{a}_n) \big) \right) \\
&\quad + \sum_{i = 1}^{n}  \sum_{j = 1}^{p_{i}}(-1)^{s + r_{i} + j} \widetilde{\partial}'_{r_i + j} \left(  \phi_{p,n} \big( \uu l_{p_1,r_1}^{p}(\overline{a}_1) \otimes \cdots \otimes  \uu l_{p_n,r_n}^{p}(\overline{a}_n) \big) \right)  \\
&=(-1)^{s}   \phi_{p+1,n} \big( \uu \partial'_0 l_{p_1,r_1}^{p}(\overline{a}_1) \otimes \cdots \otimes  \uu  \partial'_0 l_{p_n,r_n}^{p}(\overline{a}_n) \big) \\
&\quad  + (-1)^{s + p +1}  \phi_{p+1,n} \big( \uu \partial'_{p+1} l_{p_1,r_1}^{p}(\overline{a}_1) \otimes \cdots \otimes  \uu \partial'_{p+1} l_{p_n,r_n}^{p}(\overline{a}_n) \big)  \\
&\quad + \sum_{i = 1}^{n}  \sum_{j = 1}^{p_{i}}(-1)^{s + r_{i} + j}   \phi_{p,n} \big( \uu \partial'_{r_i + j} l_{p_1,r_1}^{p}(\overline{a}_1) \otimes \cdots \otimes  \uu \partial'_{r_i + j} l_{p_n,r_n}^{p}(\overline{a}_n) \big).
\end{align*}
\end{gather*}
Taking into account \eqref{eqn:3.20}, the left-hand side of the last equality becomes
\begin{gather*}
\begin{align*}
&\phantom{=}(-1)^{s}   \phi_{p+1,n} \big( \uu  l_{p_1,r_1 + 1}^{p+1}(\overline{a}_1) \otimes \cdots \otimes  \uu   l_{p_n,r_n + 1}^{p+1}(\overline{a}_n) \big)  + (-1)^{s + p + 1}  \phi_{p+1,n} \big( \uu  l_{p_1,r_1}^{p+1}(\overline{a}_1) \otimes \cdots \otimes  \uu  l_{p_n,r_n}^{p+1}(\overline{a}_n) \big) \\
& \phantom{=}+ \sum_{i = 1}^{n}  \sum_{j = 1}^{p_{i}}(-1)^{s + r_{i} + j}   \phi_{p+1,n} \big( \uu  l_{p_1,r_1}^{p +1 }(\overline{a}_1) \otimes \cdots  \otimes \uu  l_{p_{i-1},r_{i-1}}^{p+1}(\overline{a}_{i-1}) \otimes \uu  l_{p_i + 1,r_i}^{p + 1}(\partial'_j \overline{a}_i) \\
&\phantom{=+ \sum_{i = 1}^{n}  \sum_{j = 1}^{p_{i}}(-1)^{s + r_{i} + j}   \phi_{p+1,n} \big( aaa \otimes \cdots} \,\, \otimes \uu  l_{p_{i+1},r_{i+1} + 1}^{p+1}(\overline{a}_{i+1})  \otimes \cdots \otimes  \uu  l_{p_n,r_n + 1}^{p + 1}(\overline{a}_n) \big).
\end{align*}
\end{gather*}
A simple calculation shows that the first and second term of this expression cancel each other out. Thus, upon using \eqref{eqn:3.18} and \eqref{eqn:3.22} and putting all together,
\begin{gather*}
\begin{align*}
& d' \left( \Tot(\phi_{\sbullet})_n (\uu a_1 \otimes \cdots \otimes \uu a_n) \right) \\
    &\qquad\qquad\qquad = \sum_{i = 1}^{n}  (-1)^{s + r_{i}}   \phi_{p+1,n} \big( \uu  l_{p_1,r_1}^{p +1 }(\overline{a}_1) \otimes \cdots  \otimes \uu  l_{p_{i-1},r_{i-1}}^{p+1}(\overline{a}_{i-1}) \otimes \uu  l_{p_i + 1,r_i}^{p + 1}(\partial' \overline{a}_i) \\
&\qquad\qquad\qquad \phantom{= \sum_{i = 1}^{n}  (-1)^{s + r_{i}}aaaaaaaaaaaaall  } \,\, \otimes \uu  l_{p_{i+1},r_{i+1} + 1}^{p+1}(\overline{a}_{i+1})  \otimes \cdots \otimes  \uu  l_{p_n,r_n + 1}^{p + 1}(\overline{a}_n) \big) \\
&\qquad\qquad\qquad = \sum_{i+j+1 = n} \Tot(\phi_{\sbullet})_n \left( (\id^{\otimes i} \otimes d' \otimes \id^{\otimes j})(\uu a_1 \otimes \cdots \otimes \uu a_n) \right),
\end{align*} 
\end{gather*}
from which \eqref{eqn:3.26} follows.  

Let us now tackle \eqref{eqn:3.27}. By attending to \eqref{eqn:3.22} and \eqref{eqn:3.23}, for the terms on the left-hand side of \eqref{eqn:3.27} we have
\begin{gather*}
\begin{align*}
d'' \left( \Tot(\phi_{\sbullet})_n (\uu a_1 \otimes \cdots \otimes \uu a_n) \right) = (-1)^{s+p} d''_p \left( \phi_{p,n} \big(\uu l_{p_1,r_1}^p(\overline{a}_1) \otimes \cdots \otimes \uu l_{p_n,r_n}^{p}(\overline{a}_n) \big) \right), \end{align*}
\end{gather*}
and 
\begin{gather*}
\begin{align*}
&m \left( (\Tot(\phi_{\sbullet})_i \otimes \Tot(\phi_{\sbullet})_j) (\uu a_1 \otimes \cdots \otimes \uu a_n) \right) \\
&\qquad\qquad\qquad\qquad\qquad\qquad\qquad = (-1)^{s+p} m_p \left( (\phi_{p,i} \otimes \phi_{p,j}) \big(\uu l_{p_1,r_1}^p(\overline{a}_1) \otimes \cdots \otimes \uu l_{p_n,r_n}^{p}(\overline{a}_n) \big)\right).
\end{align*}
\end{gather*}
Similarly, for the terms on the right-hand side \eqref{eqn:3.27}, one obtains
\begin{gather*}
\begin{align*}
&\Tot(\phi_{\sbullet})_{n} \left( (\id^{\otimes i} \otimes d'' \otimes \id^{\otimes j}) (\uu a_1 \otimes \cdots \otimes \uu a_n) \right)  \\
&\qquad \qquad\qquad\qquad\qquad\qquad = (-1)^{s+p} \phi_{p,n} \left( (\id^{\otimes i} \otimes d''_p \otimes \id^{\otimes j})\big(\uu l_{p_1,r_1}^p(\overline{a}_1) \otimes \cdots \otimes \uu l_{p_n,r_n}^{p}(\overline{a}_n)\right),
\end{align*}
\end{gather*}
and  
\begin{gather*}
\begin{align*}
&\Tot(\phi_{\sbullet})_{n-1} \left( (\id^{\otimes i} \otimes m \otimes \id^{\otimes j}) (\uu a_1 \otimes \cdots \otimes \uu a_n) \right) \\
&\qquad \qquad\qquad\qquad\qquad\qquad = (-1)^{s+p} \phi_{p,n-1} \left( (\id^{\otimes i} \otimes m_p \otimes \id^{\otimes j})\big(\uu l_{p_1,r_1}^p(\overline{a}_1) \otimes \cdots \otimes \uu l_{p_n,r_n}^{p}(\overline{a}_n)\right).
\end{align*}
\end{gather*}
The desired conclusion is therefore a consequence of the fact that $\phi_p$ is an $\A_{\infty}$-morphism. 

In order to finish the proof, one needs to check that $A_{\sbullet} \mapsto \Tot(A_{\sbullet})$ preserves compositions. This is a straightforward verification which we omit.
\end{proof}

The following result should also be noted.

\begin{lemma} \label{lem:3.10a}
Let $A_{\sbullet}$ and $B_{\sbullet}$ be two semi-cosimplicial positively graded DG algebras. If $\phi_{\sbullet} \colon A_{\sbullet} \to B_{\sbullet}$ is a semi-cosimplicial $A_{\infty}$-quasi-isomorphism then $\Tot(\phi_{\sbullet}) \colon \Tot(A_{\sbullet}) \to \Tot(B_{\sbullet})$ is an $\A_{\infty}$-quasi-isomorphism. 
\end{lemma}

\begin{proof}
Since $A_{\sbullet}$ and $B_{\sbullet}$ are assumed to be positively graded, the map $\Tot(\phi_{\sbullet})_1$ is a morphism of first-quadrant double complexes. Besides, by our hypothesis on $\phi_{\sbullet}$, we see that $\Tot(\phi_{\sbullet})_1$ induces an isomorphism on the vertical directions. We conclude therefore that the map of spectral sequences is an isomorphism at the first page, and hence $\Tot(\phi_{\sbullet})_1$ induces an isomorphism in cohomology. 
\end{proof}

With these preliminaries out of the way, we may now formulate the version of Gugenheim's $\A_{\infty}$ De~Rham theorem for $BG$ we are after.  As in the previous section, consider the simplicial manifold $BG_{\sbullet}$. Then, the Bott-Shulman-Stasheff complex gives us a semi-cosimplicial DG algebra $\Omega^{\sbullet}(BG_{\sbullet}) = \{\Omega^{\sbullet}(BG_{p})\}_{p \geq 0}$. Also, by taking singular cochains, we get a second semi-cosimplicial DG algebra $\uC^{\sbullet}(BG_{\sbullet})=\{\uC^{\sbullet}(BG_{p})\}_{p \geq 0}$. Invoking Theorem~\ref{thm:2.1}, for each $p \geq 0$, there is an $\A_{\infty}$-morphism $\DR_p \colon \Omega^{\sbullet}(BG_{p}) \to  \uC^{\sbullet}(BG_{p})$ induced by Gugenheim's construction. Since the latter is natural with respect to the simplicial operations, this $\A_{\infty}$-morphisms commutes with the coface maps of $\Omega^{\sbullet}(BG_{\sbullet})$ and $\uC^{\sbullet}(BG_{\sbullet})$. Thus, we actually get a semi-cosimplicial $\A_{\infty}$-morphism $\DR_{\sbullet} \colon \Omega^{\sbullet}(BG_{\sbullet}) \to \uC^{\sbullet}(BG_{\sbullet})$. One obtains the following.

\begin{theorem}
The induced $\A_{\infty}$-morphism $\Tot(\DR_{\sbullet}) \colon \Tot(\Omega^{\sbullet}(BG_{\sbullet})) \to \Tot(\uC^{\sbullet}(BG_{\sbullet}))$ is an $\A_{\infty}$-quasi-isomorphism. 
\end{theorem}

\begin{proof}
It follows from Theorem~\ref{thm:2.1} that $\DR_{\sbullet} \colon \Omega^{\sbullet}(BG_{\sbullet}) \to \uC^{\sbullet}(BG_{\sbullet})$ is a semi-cosimplicial $\A_{\infty}$-quasi-isomorphism. Hence  conclusion is a consequence of  Lemma~\ref{lem:3.10a}. 
\end{proof}

\subsection{The Hochschild-De~Rham $\A_{\infty}$-quasi-isomorphism}\label{sec:3.3}
The goal of this subsection is to construct an $\A_{\infty}$-quasi-isomorphism between $\Omega^{\sbullet}(BG_{\sbullet})$ and the DG algebra of Hochschild cochains on the space of singular chains $\uC_{\sbullet}(G)$. We begin with some generalities. 

Let $A$ be a DG Hopf algebra and let $\varepsilon \colon A \to \RR$ be its counit. For each $p \geq 0$, we set $\Delta_{p}(A) = A^{\otimes p}$ and define the maps $\partial_i \colon \Delta_{p}(A) \to \Delta_{p-1}(A)$ by
\begin{equation} \label{eqn:3.28}
\partial_{i} (a_1 \otimes \cdots \otimes a_p) = \begin{cases}  \varepsilon(a_1) a_2 \otimes \cdots \otimes a_p & \text{if $i=0$,} \\
a_{1} \otimes \cdots \otimes a_{i-1} \otimes a_{i}a_{i+1} \otimes a_{i+2} \otimes  \cdots \otimes a_p & \text{if $0 < i < p$,} \\
a_2 \otimes \cdots \otimes a_{p-1} \varepsilon(a_p) & \text{if $i=p$.}
 \end{cases}
\end{equation}
The following lemma es crucial. 

\begin{lemma}
The collection $\Delta_{\sbullet}(A) = \{\Delta_{p}(A)\}_{p \geq 0}$ is a semi-simplicial DG coalgebra with face maps $\partial_i$. 
\end{lemma}

\begin{proof}
Since, by hypothesis, $A$ is a DG Hopf algebra, its the product map is a morphism of DG coalgebras, which means that the maps $\partial_i$ for $0 < i < p$ are indeed morphisms of DG coalgebras. On the other hand, the fact that $\varepsilon \colon A \to \RR$ is a morphism of DG Hopf algebras implies that both $\partial_0$ and $\partial_p$ are also morphisms of DG coalgebras. It remains only to check that $\partial_{i} \circ \partial_{j} = \partial_{j-1} \circ \partial_{i}$ for $i < j$. This is a routine calculation which we leave to the reader.
\end{proof}

We next let $\Delta^{\sbullet}(A)$ be the semi-cosimplicial DG algebra obtained by dualising $\Delta_{\sbullet}(A)$. 
The following observation should be made. 

\begin{lemma}\label{lem:3.13a}
The product in $\Tot(\Delta^{\sbullet}(A))$ is given, for homogeneous elements $\varphi$ and $\psi$, by the formula
$$
\us^p \varphi \bigcdot \us^{q}\psi = (-1)^{\vert \varphi \vert q} \us^{p+q} (\varphi \abxcup \psi),
$$
where $\abxcup$ designates the cup product in the Hochschild cochain complex $\HC^{\sbullet}(A)$. 
\end{lemma}

\begin{proof}
As earlier, for $p + q \leq n$, let $l_{p,q}^{n}\colon \Delta^{p}(A) \to \Delta^{n}(A)$ denote the map induced by that given in \eqref{eqn:3.28}. This induces a corresponding map $l_{p,q}^{n*}\colon A^{\otimes n} \to A^{\otimes p}$. We claim that, if $\Delta$ denotes the coproduct in $A$,  
\begin{equation}\label{eqn:3.40a}
\left(l_{p,0}^{p+q*} \otimes l_{q,p}^{p+q*}\right) \circ \Delta^{\otimes (p+q)} = \id_{A^{\otimes (p+q)}}. 
\end{equation}
To substantiate our claim, we fix homogeneous elements $a_1,\dots, a_{p+q} \in A$, and notice that
\begin{equation}\label{eqn:3.41a}
l_{p,0}^{p+q*}(a_1 \otimes \cdots \otimes a_{p+q}) = \varepsilon(a_{p+1} \cdots a_{p+q}) a_1 \otimes \cdots \otimes a_{p},
\end{equation}
and 
\begin{equation}\label{eqn:3.42a}
l_{q,p}^{p+q*}(a_1 \otimes \cdots \otimes a_{p+q}) = \varepsilon(a_{1} \cdots a_{p}) a_{p+1} \otimes \cdots \otimes a_{p+q}. 
\end{equation}
We also make use of Sweedler's notation and write, for each $1 \leq i \leq p+q$,
\begin{equation}\label{eqn:3.43a}
\Delta(a_i) = \sum a_{i(1)} \otimes a_{i (2)}.  
\end{equation}
Using \eqref{eqn:3.41a}, \eqref{eqn:3.42a} and \eqref{eqn:3.43a}, we find that
\begin{align*}
&\left(l_{p,0}^{p+q*} \otimes l_{q,p}^{p+q*}\right) \left( \Delta^{\otimes (p+q)} (a_1 \otimes \cdots \otimes a_{p+q})\right) \\ 
&\,\,= \sum (-1)^{\sum_{ i < j } \vert a_{i(1)} \vert\vert a_{j(2)}\vert } \left(l_{p,0}^{p+q} \otimes l_{q,p}^{p+q}\right) (a_{1(1)} \otimes \cdots \otimes a_{p+q(1)} \otimes a_{1(2)} \otimes \cdots \otimes a_{p+q(2)}) \\
&\,\,= \sum (-1)^{\sum_{ i < j } \vert a_{i(1)} \vert\vert a_{j(2)}\vert } a_{1(1)} \otimes \cdots \otimes a_{p(1)} \varepsilon(a_{p+1(1)} \cdots a_{p+q(1)} a_{1(2)} \cdots a_{p(2)}) a_{p+1(2)} \otimes \cdots \otimes a_{p+q(2)} \\
&\,\,= \sum  a_{1(1)} \otimes \cdots \otimes a_{p(1)} \varepsilon(a_{p+1(1)} \cdots a_{p+q(1)} a_{1(2)} \cdots a_{p(2)}) a_{p+1(2)} \otimes \cdots \otimes a_{p+q(2)} \\
&\,\,= \left(\sum a_{1(1)} \varepsilon(a_{1(2)}) \right) \otimes \cdots  \otimes  \left(\sum a_{p+q(1)} \varepsilon(a_{p+q(2)}) \right) \\
&\,\, = a_1 \otimes \cdots \otimes a_{p+q},
\end{align*}
where in the third equality we have used the fact that $\varepsilon$ vanishes on elements of positive degree in order to set the signs to zero, and, in the last one, that $\varepsilon$ is a counit for the coproduct. Thus \eqref{eqn:3.40a} is true.   

Next, let us denote by $\mu^{p,q}$ the natural map from $A^{\otimes p*} \otimes A^{\otimes q*}$ to $(A^{\otimes p} \otimes A^{\otimes q})^*$. It is a simple matter to verify that
\begin{equation}\label{eqn:3.44a}
\mu^{p+q,p+q} \circ \left(l_{p,0}^{p+q} \otimes l_{p,q}^{p+q}\right) = \left(l_{p,0}^{p+q*} \otimes l_{p,q}^{p+q*}\right)^* \circ \mu^{p,q}.
\end{equation}
Moreover, we also have that
\begin{equation}\label{eqn:3.45a}
\mu^{p,q}(\varphi \otimes \psi) = \varphi \abxcup \psi,
\end{equation}
for $\varphi \in \Delta^p(A)$ and $\psi \in \Delta^{q}(A)$. Now, the product in $\Tot(\Delta^{\sbullet}(A))$ is defined, up to a sign, as the composition
$$
\Delta^{\otimes(p+q)*}\circ \mu^{p+q,p+q} \circ \left(l_{p,0}^{p+q} \otimes l_{p,q}^{p+q}\right). 
$$
More explicitly, for homogeneous elements $\varphi \in \Delta^p(A)$ and $\psi \in \Delta^{q}(A)$, we may write
\begin{equation}\label{eqn:3.46a}
\us^{p} \varphi \bigcdot \us^{q}\psi = (-1)^{\vert \varphi\vert q} \us^{p+q}\Delta^{\otimes(p+q)*} \left( \mu^{p+q,p+q} \left(l_{p,0}^{p+q}\varphi \otimes l_{p,q}^{p+q} \psi \right) \right). 
\end{equation}
Using \eqref{eqn:3.44a} and \eqref{eqn:3.45a}, the left hand side of \eqref{eqn:3.46a} becomes
\begin{align*}
&(-1)^{\vert \varphi\vert q} \us^{p+q}\Delta^{\otimes(p+q)*}\left(\left(l_{p,0}^{p+q*} \otimes l_{p,q}^{p+q*}\right)^* \left(\mu^{p,q} (\varphi \otimes \psi)\right) \right)  \\
&\qquad\qquad\qquad= (-1)^{\vert \varphi\vert q} \us^{p+q}\Delta^{\otimes(p+q)*}\left(\left(l_{p,0}^{p+q*} \otimes l_{p,q}^{p+q*}\right)^* \left(\varphi \abxcup \psi \right) \right) \\
& \qquad\qquad\qquad= (-1)^{\vert \varphi\vert q} \us^{p+q}  \left( \left(l_{p,0}^{p+q*} \otimes l_{p,q}^{p+q*}\right) \circ \Delta^{\otimes (p+q)} \right)^* \left(\varphi \abxcup \psi \right). 
\end{align*}
Combining this with \eqref{eqn:3.40a} gives
$$
\us^{p} \varphi \bigcdot \us^{q}\psi = (-1)^{\vert \varphi\vert q} \us^{p+q} \left(\varphi \abxcup \psi \right), 
$$
as we wished to show. 
\end{proof}

\begin{lemma}\label{lem:3.14a}
There is an isomorphism of DG algebras $\Theta \colon \Tot(\Delta^{\sbullet}(A)) \to \HC^{\sbullet}(A)$, which is explicitly given by
$$
\Theta(\us^{p}(\varphi)) (\uu a_1 \otimes \cdots \otimes \uu a_{p}) = (-1)^{p \vert \varphi\vert + \frac{p(p-1)}{2}+ \sum_{i=1}^{p-1} \vert a_i \vert (p-i)}\varphi(a_1 \otimes \cdots \otimes a_p),
$$
for homogeneous elements $\varphi \in \Delta^{p}(A)$ and $a_1,\dots, a_p \in A$. 
\end{lemma}

\begin{proof}
It is obvious that $\Theta$ is a linear isomorphism. We have to show that it is also an algebra homomorphism. For this, let us fix homogeneous elements $\varphi  \in \Delta^{p}(A)$, $\psi \in \Delta^{q}(A)$ and $a_1,\dots,a_{p+q} \in A$. Then, one the one hand, by virtue of Lemma~\ref{lem:3.13a}, 
\begin{align*}
&\Theta\left( \us^p \varphi \bigcdot \us^q \psi \right) (\uu a_1 \otimes \cdots \otimes \uu a_{p+q}) \\
&\,\, =(-1)^{\vert\varphi\vert q} \Theta(\us^{p+q}(\varphi \abxcup \psi)) (\uu a_1 \otimes \cdots \otimes \uu a_{p+q}) \\
&\,\, =(-1)^{\vert\varphi\vert q+ (p+q)(\vert\varphi\vert + \vert\psi\vert) + \frac{(p+q)(p+q-1)}{2} + \sum_{i=1}^{p+q-1}\vert a_i \vert (p+q-i)} (\varphi \abxcup \psi)(\uu a_1 \otimes \cdots \otimes \uu a_{p+q}) \\
&\,\, = (-1)^{\vert\varphi\vert q+ (p+q)(\vert\varphi\vert + \vert\psi\vert) + \frac{(p+q)(p+q-1)}{2} + \sum_{i=1}^{p+q-1}\vert a_i \vert (p+q-i) + \vert\varphi\vert \vert\psi\vert} \varphi(a_1 \otimes \cdots \otimes a_p)\psi(a_{p+1} \otimes \cdots \otimes a_{p+q}) \\
&\,\, = (-1)^{p(\vert\varphi\vert + \vert\psi\vert) + (q + \vert\varphi\vert)\vert\psi\vert + \frac{(p+q)(p+q-1)}{2} +  \sum_{i=1}^{p+q-1}\vert a_i \vert (p+q-i)} \varphi(a_1 \otimes \cdots \otimes a_p)\psi(a_{p+1} \otimes \cdots \otimes a_{p+q}). 
\end{align*}
On the other hand,
\begin{align*}
&\left(\Theta(\us^p \varphi) \abxcup \Theta(\us^{q}\psi)\right) (\uu a_1 \otimes \cdots \otimes \uu a_{p+q}) \\
&\,\, = (-1)^{(q + \vert\psi\vert)(p + \vert\varphi\vert)} \Theta(\us^p \varphi)(\uu a_1 \otimes \cdots \otimes \uu a_{p}) \Theta(\us^q \psi)(\uu a_{p+1} \otimes \cdots \otimes \uu a_{p+q}) \\
&\,\, = (-1)^{(q + \vert\psi\vert)(p + \vert\varphi\vert) +  \frac{p(p-1)+q(q-1)}{2}  + p \vert\varphi\vert + \sum_{i =1}^{p}\vert a_i \vert (p+q-i) + q \vert\psi\vert + \sum_{i=p+1}^{p+q}\vert a_i \vert (p+q-i)}  \\
&\qquad \qquad \qquad \qquad\qquad \qquad\qquad \qquad\qquad \qquad\qquad \qquad \,\,\,\,\, \times \varphi(a_1 \otimes \cdots \otimes a_p)\psi(a_{p+1} \otimes \cdots \otimes a_{p+q}) \\
&= \,\, (-1)^{ p(\vert\varphi\vert + \vert\psi\vert) + (q + \vert\varphi\vert)\vert\psi\vert + \frac{(p+q)(p+q-1)}{2} +  \sum_{i=1}^{p+q-1}\vert a_i \vert (p+q-i)} \varphi(a_1 \otimes \cdots \otimes a_p)\psi(a_{p+1} \otimes \cdots \otimes a_{p+q}).
\end{align*}
By comparing the last two equalities above we find that
$$
\Theta\left( \us^p \varphi \bigcdot \us^q \psi \right) = \Theta(\us^p \varphi) \abxcup \Theta(\us^{q}\psi).
$$
We conclude that $\Theta$ is indeed an algebra homomorphism. We leave it to the reader the task of checking that it also preserves the differentials. 
\end{proof}

Before moving forward, we need some definitions. An $n$-singular simplex $\sigma \colon \Delta_{n} \to G$ is said to be \emph{degenerate} if it can be factored as $\sigma = \sigma' \circ \eta_i$, where $\sigma' \colon \Delta_{n-1} \to G$ is an $(n-1)$-singular simplex and $\eta_i \colon \Delta_{n-1} \to \Delta_{n}$ is a degeneracy map. It is not hard to see that the vector space $\mathfrak{I}$ generated by degenerate simplices on $G$ is both a DG ideal and a DG coideal of $\uC_{\sbullet}(G)$. Thus, taking the quotient vector space $\overline{\uC}_{\sbullet}(G) = \uC_{\sbullet}(G) / \mathfrak{I}$, we obtain canonically a DG Hopf algebra structure on $\overline{\uC}_{\sbullet}(G)$ such that the projection $\uq \colon \uC_{\sbullet}(G) \to \overline{\uC}_{\sbullet}(G)$ is a morphism of DG Hopf algebras. It can be shown that $\uq$ is in fact a quasi-isomorphism. We shall refer to the DG Hopf algebra $\overline{\uC}_{\sbullet}(G)$ as the algebra of \emph{normalized} singular chains on $G$. 

Now let us again consider the simplicial manifold $BG_{\sbullet}$ from Section~\ref{sec:3.1}. As the construction above is functorial in $G$, it defines a semi-simplicial DG coalgebra $\overline{\uC}_{\sbullet}(BG_{\sbullet})$ and a corresponding projection $\uq_{\sbullet} \colon \uC_{\sbullet}(BG_{\sbullet}) \to \overline{\uC}_{\sbullet}(BG_{\sbullet})$. For each $p \geq 0$,  we let $\EZ_{p} \colon \uC_{\sbullet}(G)^{\otimes p} \to \uC_{\sbullet}(BG_{p})$ be the Eilenberg-Zilber map defined as reviewed in Section~\ref{sec:2.4}. We further let $\overline{\EZ}_{p} \colon \uC_{\sbullet}(G)^{\otimes p} \to \overline{\uC}_{\sbullet}(BG_{p})$ be defined as the composition
\begin{equation}
\overline{\EZ}_{p} = \uq_{p} \circ \EZ_{p}.
\end{equation}
Owing to Proposition~\ref{prop:2.2} and the preceding discussion, the map $\overline{\EZ}_{p}$ is a quasi-isomorphism of DG coalgebras. Also, we have the following. 

\begin{lemma}
The collection $\{\overline{\EZ}_{p}\}_{p \geq 0}$ determines a morphism of semi-simplicial DG coalgebras $\overline{\EZ}_{\sbullet} \colon \Delta_{\sbullet}(\uC_{\sbullet}(G)) \to \overline{\uC}_{\sbullet}(BG_{\sbullet})$. 
\end{lemma}

\begin{proof}
We need merely to show that $\overline{\EZ}_{p-1} \circ \partial_{i} = \partial_{i} \circ \overline{\EZ}_{p}$ for $0 \leq i \leq p$. First, consider the case $i =0$. For every collection $\sigma_{1},\dots, \sigma_{p}$ of singular simplices on $G$, we see from \eqref{eqn:3.28} that
\begin{align*}
\overline{\EZ}_{p-1} \left( \partial_{0} (\sigma_1 \otimes \cdots \otimes \sigma_{p})\right) &= \varepsilon(\sigma_1)\overline{\EZ}_{p-1} \left( \sigma_2 \otimes \cdots \otimes \sigma_{p}\right) \\
& = \begin{cases} 0 & \text{if $\vert \sigma_1 \vert > 0$,} \\ 
\overline{\EZ}_{p-1} \left( \sigma_2 \otimes \cdots \otimes \sigma_{p}\right) & \text{if $\vert \sigma_1 \vert = 0$.}
\end{cases}  
\end{align*}
On the other hand, we can see from \eqref{eqn:2.18} that
\begin{align*}
\partial_{0} \left(\overline{\EZ}_{p} (\sigma_1 \otimes \cdots \otimes \sigma_{p}) \right) &= \partial_{0}  \left( \sum_{\chi \in \mathfrak{S}_{\vert \sigma_1 \vert,\dots, \vert \sigma_p \vert }}(-1)^{\chi} \uq_{p} \circ (\sigma_1 \times \cdots \times \sigma_{p}) \circ \chi_{*} \right) \\
&=   \sum_{\chi \in \mathfrak{S}_{\vert \sigma_1 \vert,\dots, \vert \sigma_p \vert }}(-1)^{\chi} \uq_{p-1} \circ \partial_{0} \circ (\sigma_1 \times \cdots \times \sigma_{p}) \circ \chi_{*} . 
\end{align*}
If $\vert \sigma_1 \vert = 0$, then $\partial_{0} \circ (\sigma_1 \times \cdots \times \sigma_p) = \sigma_2 \times \cdots \times \sigma_p$ and therefore
\begin{equation*}
\partial_{0} \left(\overline{\EZ}_{p} (\sigma_1 \otimes \cdots \otimes \sigma_{p}) \right) = \overline{\EZ}_{p-1} (\sigma_2 \otimes \cdots \otimes \sigma_{p}). 
\end{equation*}
If $\vert \sigma_1 \vert = 0$, then $\uq_{p-1} \circ \partial_{0} \circ (\sigma_1 \times \cdots \times \sigma_{p}) \circ \chi_{*} = 0$, since $\partial_{0} \circ (\sigma_1 \times \cdots \times \sigma_{p}) \circ \chi_{*}$ is degenerate, from which it follows that
\begin{equation*}
\partial_{0} \left(\overline{\EZ}_{p} (\sigma_1 \otimes \cdots \otimes \sigma_{p}) \right) = 0,
\end{equation*}
and hence the result. The case $i = p$ is completely analogous. So there only remains the case $0 < i < p$. On the one hand, using \eqref{eqn:3.28} gives
\begin{equation*}
\overline{\EZ}_{p-1} \left( \partial_{i} (\sigma_1 \otimes \cdots \otimes \sigma_{p})\right) = \uq_{p-1} \left(\EZ_{p-1}(\sigma_{1} \otimes \cdots \otimes \sigma_{i-1} \otimes \sigma_{i}  \sigma_{i+1} \otimes \sigma_{i+2} \otimes  \cdots \otimes \sigma_p) \right).
\end{equation*}
On the other hand, observing that, in the notation employed at the end Section~\ref{sec:2.4}, the face map $\partial_{i} \colon \uC_{\sbullet}(BG_{p}) \to \uC_{\sbullet}(BG_{p-1})$ is given by $\partial_{i}=(\id^{\times (i-1)} \times \mu \times \id^{\times (p-i)})_*$ and that, therefore,
\begin{equation*}
\partial_{i} \circ \EZ_{p} = \EZ_{p-1} \circ \big( \id^{\otimes(i-1)} \otimes (\mu_{*} \circ \EZ_2) \otimes \id^{\otimes(p-1)}\big),
\end{equation*}
we find that
\begin{align*}
\partial_{i} \left( \overline{\EZ}_{p}(\sigma_1 \otimes \cdots \otimes \sigma_p )\right) & = \partial_i \left( \uq_{p} \left( \EZ_{p}(\sigma_1 \otimes \cdots \otimes \sigma_p )\right) \right) \\
&= \uq_{p-1} \left( \partial_i \left( \EZ_{p}(\sigma_1 \otimes \cdots \otimes \sigma_p ) \right)\right) \\
&= \uq_{p-1} \left(\EZ_{p-1}(\sigma_{1} \otimes \cdots \otimes \sigma_{i-1} \otimes \sigma_{i}  \sigma_{i+1} \otimes \sigma_{i+2} \otimes  \cdots \otimes \sigma_p) \right),
\end{align*}
as wished. 
\end{proof}

Let us now write $\overline{\uC}{}^{\sbullet}(BG_{\sbullet})$ to denote the semi-cosimplicial DG algebra obtained by dualising $\overline{C}_{\sbullet}(BG_{\sbullet})$. Also, let us denote by $\ue_{\sbullet} \colon \overline{\uC}{}^{\sbullet}(BG_{\sbullet}) \to \uC^{\sbullet}(BG_{\sbullet})$ the inclusion dual to the projection $\uq_{\sbullet}\colon \uC_{\sbullet}(BG_{\sbullet}) \to \overline{\uC}_{\sbullet}(BG_{\sbullet})$. We make an observation to be applied in the subsequent argument.

\begin{lemma}\label{lem:3.10}
The semi-cosimplicial $\A_{\infty}$-morphism $\DR_{\sbullet} \colon \Omega^{\sbullet}(BG_{\sbullet}) \to \uC^{\sbullet}(BG_{\bullet})$ factors thorugh $\overline{\uC}{}^{\sbullet}(BG_{\sbullet})$, that is, there is a semi-cosimplicial $\A_{\infty}$-morphism $\overline{\DR}_{\sbullet} \colon \Omega^{\sbullet}(BG_{\sbullet}) \to \overline{\uC}{}^{\sbullet}(BG_{\sbullet})$ such that $\DR_{\sbullet} = \ue_{\sbullet} \circ \overline{\DR}_{\sbullet}$. 
\end{lemma}

\begin{proof}
It suffices to show that $\DR_{\sbullet} \colon \Omega^{\sbullet}(BG_{\sbullet}) \to \uC^{\sbullet}(BG_{\bullet})$ takes values in those singular cochains that vanish on degenerate singular simplices. But this holds by Proposition~3.26 of \cite{Abad-Schatz2013}. 
\end{proof}

Next we consider the semi-cosimplicial DG algebra $\Delta^{\sbullet}(\uC_{\sbullet}(G))$ dual to $\Delta^{\sbullet}(\uC_{\sbullet}(G))$. By taking the dual of $\overline{\EZ}_{\sbullet} \colon \Delta_{\sbullet}(\uC_{\sbullet}(G)) \to \overline{\uC}_{\sbullet}(BG_{\sbullet})$, we obtain a morphism of semi-cosimplicial DG algebras $\overline{\EZ}_{\usbullet}^{\ast} \colon  \overline{\uC}^{\sbullet}(BG_{\sbullet}) \to \Delta^{\sbullet}(\uC_{\sbullet}(G))$. We let $\DR_{\sbullet}^{\Delta} \colon \Omega^{\sbullet}(BG_{\sbullet}) \to  \Delta^{\sbullet}(\uC_{\sbullet}(G))$ be the semi-cosimplicial $\A_{\infty}$-morphism defined as the composition
\begin{equation} \label{eqn:3.48a}
\DR_{\sbullet}^{\Delta} =\overline{\EZ}_{\usbullet}^{\ast} \circ \overline{\DR}_{\sbullet},
\end{equation}
where $\overline{\DR}_{\sbullet} \colon \Omega^{\sbullet}(BG_{\sbullet}) \to \overline{\uC}{}^{\sbullet}(BG_{\sbullet})$ is the semi-cosimplicial $\A_{\infty}$-morphism from Lemma~\ref{lem:3.10}. By Proposition~\ref{prop:3.5}, the latter induces an $\A_{\infty}$-morphism $\Tot(\DR_{\sbullet}^{\Delta}) \colon \Tot (\Omega^{\sbullet}(BG_{\sbullet})) \to \Tot(\Delta^{\sbullet}(\uC_{\sbullet}(G)))$. We then apply Lemma~\ref{lem:3.14a}, to get and $\A_{\infty}$-morphism $\DR^{\Theta} \colon \Tot (\Omega^{\sbullet}(BG_{\sbullet})) \to \HC^{\sbullet}(\uC_{\sbullet}(G))$ defined as the composition
\begin{equation}
\DR^{\Theta} = \Theta \circ \Tot(\DR_{\sbullet}^{\Delta}). 
\end{equation}
These observations taken together with the preceding results yield the following. 

\begin{theoremD}
The induced $\A_{\infty}$-morphism $\DR^{\Theta} \colon \Tot (\Omega^{\sbullet}(BG_{\sbullet})) \to \HC^{\sbullet}(\uC_{\sbullet}(G))$ is an $\A_{\infty}$-quasi-isomorphism. 
\end{theoremD}

\begin{proof}
Because of Lemma~\ref{lem:3.14a}, it is enought to prove that $\Tot(\DR_{\sbullet}^{\Delta})$ is an $\A_{\infty}$-quasi-isomorphism. In order to do so, notice that both $\overline{\EZ}_{\usbullet}^{\ast}$ and $\overline{\DR}_{\sbullet}$ are semi-cosimplicial $\A_{\infty}$-quasi-isomorphism. Thus, taking note of the definition \eqref{eqn:3.48a}, we conclude that $\DR_{\sbullet}^{\Delta}$ is also a semi-cosimplicial $\A_{\infty}$-quasi-isomorphism. The desired assertion now follows from Lemma~\ref{lem:3.10a}.
\end{proof}

In the remaining part of this section, we will prove a vanishing result for the $\A_{\infty}$-morphism $\DR^{\Theta}$, which will be needed later. First a little terminology. We say that an $r$-singular simplex $\sigma \colon \Delta_{r} \to BG_{p}$ is \emph{decomposable} if there is a collection of $r_i$-singular simplices $\sigma_i \colon \Delta_{r_i} \to G$ with $i = 1,\dots, p$ and $r = \sum_{i=1}^{p} r_i$, together with an $(r_1,\dots,r_p)$-shuffle $\chi$ such that
\begin{equation}
\sigma = (\sigma_1 \times \cdots \times \sigma_p) \circ \chi_*.
\end{equation} 
It is immediately apparent from \eqref{eqn:2.18} that, for each $p \geq 0$, the image of the Eilenberg-Zilber map $\EZ_{p} \colon \uC_{\sbullet}(G)^{\otimes p} \to \uC_{\sbullet}(BG_p)$ is generated by decomposable singular simplices. 

\begin{proposition}\label{prop:3.18}
Let $n > 1$ and consider differential forms $\omega_1 \in \Omega^{q_1}(G), \dots, \omega_n \in \Omega^{q_n}(G)$. Then
\begin{equation*}
\DR^{\Theta}_n (\uu \omega_1 \otimes \cdots \otimes \uu \omega_n) = 0. 
\end{equation*}
\end{proposition}
 
 \begin{proof}
We will in fact show that
\begin{equation*}
\Tot(\DR_{\sbullet}^{\Delta})_n (\uu \omega_1 \otimes \cdots \otimes \uu \omega_n) = 0,
\end{equation*}
which clearly suffices thanks to Lemma~\ref{lem:3.14a}. To begin with, by Proposition~\ref{prop:3.5}, we know that $\Tot(\DR_{\sbullet}^{\Delta}) = \Tot(\overline{\EZ}_{\usbullet}^{\ast}) \circ \Tot(\overline{\DR}_{\sbullet})$. 
Thus, in view of our previous remark, it will be enough to show that if $\sigma \colon \Delta_{r} \to BG_{n}$ with $r = \sum_{i=1}^{n} q_i - n +1$ is an decomposable singular simplex, then
 \begin{equation*}
 \Tot(\overline{\DR}_{\sbullet})_n (\uu \omega_1 \otimes \cdots \otimes \uu \omega_n) (\sigma) = 0. 
 \end{equation*}
 According to definition \eqref{eqn:3.24}, this means that
 \begin{equation}\label{eqn:3.32}
 \DR_{n,n} \big(\uu l_{1,0}^{n} (\omega_{1}) \otimes \cdots \otimes \uu l_{1,n-1}^{n} (\omega_{n}) \big)(\sigma) = 0. 
 \end{equation}
 Next, notice that if, for $i = 1,\dots, n$, we let $\pr_{i} \colon BG_{n} \to G$ be the projection onto the $i$th factor, then $l_{1,i-1}^{n} = \pr_{i}^*$. Thus \eqref{eqn:3.32} becomes
  \begin{equation}\label{eqn:3.33}
 \DR_{n,n} \big(\uu \pr_1^*\omega_{1} \otimes \cdots \otimes \uu \pr_n^*\omega_{n} \big)(\sigma) = 0. 
 \end{equation}
 On the other hand, using the notation for the $\A_{\infty}$ version of De Rham map from \S\ref{sec:2.3}, we have that
 \begin{align*}
 \DR_{n,n} \big(\uu \pr_1^*\omega_{1} \otimes \cdots \otimes \uu \pr_n^*\omega_{n} \big)(\sigma) &= \pm \int_{I^{r-1}} \theta_r^*(\Pcal \sigma)^* \int_{\Delta_n} \ev^* \big( \pi_1^*\pr_1^*\omega_1 \wedge \cdots \wedge  \pi_n^*\pr_n^*\omega_1\big) \\
 &= \pm \int_{I^{r-1}} \theta_r^*   \int_{\Delta_n}  (\id \times \Pcal \sigma)^*\ev^* \big( \pi_1^*\pr_1^*\omega_1 \wedge \cdots \wedge  \pi_n^*\pr_n^*\omega_1\big).
  \end{align*} 
 Consequently, to show \eqref{eqn:3.33}, it is sufficient to show that
 \begin{equation} \label{eqn:3.34}
(\id \times \Pcal \sigma)^*\ev^* \big( \pi_1^*\pr_1^*\omega_1 \wedge \cdots \wedge  \pi_n^*\pr_n^*\omega_1\big) = 0. 
 \end{equation}
 Now let us use the fact that $\sigma$ is decomposable. By definition, this means that
 \begin{equation*}
 \sigma = (\sigma_1 \times \cdots \times \sigma_n) \circ \chi_*,
 \end{equation*}
 for a collection of $r_i$-singular simplices $\sigma_i \colon \Delta_{r_i} \to G$ with $i=1,\dots, n$ and $r = \sum_{i=1}^{n} r_i$, and for an $(r_1,\dots,r_n)$-shuffle $\chi$. Therefore, 
 \begin{equation*}
 \id \times \Pcal \sigma = (\id \times \Pcal(\sigma_1 \times \cdots \times \sigma_n)) \circ (\id \times \Pcal \chi_*),  
 \end{equation*}
 and hence
 \begin{equation*}
 \ev \circ (\id \times \Pcal \sigma) = \ev \circ (\id \times \Pcal(\sigma_1 \times \cdots \times \sigma_n)) \circ (\id \times \Pcal \chi_*) = (\sigma_1 \times \cdots \times \sigma_n)^{\times n} \circ \ev \circ \Pcal \chi_*.
 \end{equation*}
 Thus, to show \eqref{eqn:3.34}, it will be enough to show that
 \begin{equation}\label{eqn:3.35}
 \ev^* ((\sigma_1 \times \cdots \times \sigma_n)^{\times n})^*\big( \pi_1^*\pr_1^*\omega_1 \wedge \cdots \wedge  \pi_n^*\pr_n^*\omega_1\big) = 0.
 \end{equation}
 Now a simple calculation reveals that
 \begin{equation} \label{eqn:3.36}
 \ev^* ((\sigma_1 \times \cdots \times \sigma_n)^{\times n})^*\big( \pi_1^*\pr_1^*\omega_1 \wedge \cdots \wedge  \pi_n^*\pr_n^*\omega_1\big) = \ev^* \big( \sigma_1^*\omega_1 \wedge \cdots \wedge  \sigma_n^*\omega_1\big). 
 \end{equation}
On the other hand, since $n > 1$, we have that $r = \sum_{i=1}^{n} q_i - n + 1 < \sum_{i=1}^{n} q_i$. But $r = \sum_{i=1}^{n} r_i$, so there must exists a $k \in \{ 1,\dots, n\}$ such that $r_k < q_k$. This implies that $\sigma_k^*\omega_k = 0$, and as a result \eqref{eqn:3.35} follows from \eqref{eqn:3.36}. 
 \end{proof}


\section{$\A_{\infty}$-quasi-equivalence of DG categories}\label{sec:4}
In this section we prove the main result of the paper, which is the construction of a zig-zag of $\A_{\infty}$-quasi-equivalences between the DG enhancements of the categories $\Rep(\TT\gfrak)$ and $\Mod(\uC_{\sbullet}(G))$.  

\subsection{DG enhancement of the category $\Rep(\TT \gfrak)$}\label{sec:4.1}
In this subsection we describe a DG enhancement of the category $\Rep(\TT \gfrak)$. 
Let $V$ be an object of $\Rep(\TT \gfrak)$. For $x \in \gfrak$, by a slight abuse of notation, we will indistinctly write $i_x$ and $L_x$ for the contraction and Lie derivative operators acting on $\uW\gfrak$ or $V$. An element $\alpha \in \uW\gfrak \otimes V$ will be called \emph{basic} if
\begin{align}\label{eqn:4.1}
\begin{split}
(i_x \otimes 1) \alpha &= 0, \\
(L_x \otimes 1 + 1 \otimes L_x)\alpha &= 0,
\end{split}
\end{align}
for every $x \in \gfrak$. Since the operators $i_x \otimes 1$ and $L_x \otimes 1 + 1 \otimes L_x$ are derivations, the basic elements form a graded subspace of $\uW\gfrak \otimes V$. It will be denoted by $(\uW\gfrak \otimes V)_{\bas}$.

Next, consider the DG algebra $\uW\gfrak \otimes \End(V)$ with multiplication induced by the composition operation $\End(V)$ and the differential $\dW  + \delta$. Fix a basis $e_a$ of $\gfrak$ with structure constant $f^{a}_{\phantom{a}bc}$ and recall from Section~\ref{sec:2.4} that $t^{a}$ stands for the degree $1$ generators of $\Lambda^1\gfrak$ and $w^{a}$ stands for the degree $2$ generators of $\uS^1 \gfrak$. 

\begin{lemma}
The element $t^{a} \otimes L_{a} - w^{a} \otimes i_{a}$ is a Maurer-Cartan element of $\uW\gfrak \otimes \End(V)$. 
\end{lemma}

\begin{proof}
On the one hand, according to \eqref{eqn:2.36} and the relations \eqref{eqn:2.22}, 
\begin{align*}
\dW (t^{a} \otimes L_{a} - w^{a} \otimes i_{a} ) &=  \dW t^{a} \otimes L_a - \dW w^{a} \otimes i_a \\
&=   w^{a} \otimes L_a - \frac{1}{2} f^{a}_{bc} t^{b} t^{c} \otimes L_a - f^{a}_{bc} w^{b} t^{c} \otimes i_a,
\end{align*}
and
\begin{align*}
\delta( t^{a} \otimes L_{a} - w^{a} \otimes i_{a} ) &=- t^{a} \otimes [\delta, L_a] - w^{a} \otimes [\delta, i_a] =- w^{a} \otimes L_a .
\end{align*}
Hence,
$$
(\dW +\delta) (t^{a} \otimes L_{a} - w^{a} \otimes i_{a} ) =  -f^{a}_{bc} w^{b} t^{c} \otimes i_a - \frac{1}{2} f^{a}_{bc}t^{b} t^{c} \otimes L_a
$$
On the other hand, again using the relations \eqref{eqn:2.22}, we find that
\begin{align*}
(t^{b} \otimes L_{b}& -w^{b} \otimes i_{b})(t^{c} \otimes L_{c}- w^{c} \otimes i_{c})  \\
&=  t^{b} t^{c} \otimes L_bL_c - t^b w^c \otimes L_b i_c + w^b t^c \otimes i_b L_c + w^{b} w^{c} \otimes i_b i_c \\
&= \frac{1}{2} t^{b} t^{c} \otimes [L_b,L_c] - t^b w^c \otimes [L_b,i_c] - t^b w^c \otimes i_c L_b + w^b t^c \otimes i_b L_c + \frac{1}{2} w^{b} w^{c} \otimes [i_b,i_c]   \\
&= \frac{1}{2} f_{bc}^{a} t^b t^c \otimes L_a - f_{bc}^{a} t^b w^c \otimes i_a   \\
&= \frac{1}{2} f_{bc}^{a} t^b t^c \otimes L_a +  f_{bc}^{a} w^b t^c \otimes i_a.
\end{align*}
In conclusion, we obtain
$$
(\dW +  \delta) (t^{a} \otimes L_{a} - w^{a} \otimes i_{a}) + (t^{b} \otimes L_{b} -w^{b} \otimes i_{b})(t^{c} \otimes L_{c}- w^{c} \otimes i_{c}) = 0 ,
$$
as required. 
\end{proof}

This result has the following important consequence.  

\begin{corollary}
The operator $D$ in $\uW\gfrak \otimes V$ given by
$$
D = \dW  +  \delta + t^{a} \otimes L_{a} - w^{a} \otimes i_{a},
$$
is a derivation of homogenous degree $1$ that satisfies $D^2 = 0$.  
\end{corollary}
 
Also, the following property holds true. 

\begin{lemma}
The differential $D$ preserves the graded subspace $(\uW\gfrak \otimes V)_{\bas}$. 
\end{lemma}

\begin{proof}
It suffices to show that
$$
[D, L_{c} \otimes 1 + 1 \otimes L_{c}] = 0,
$$
and 
$$
[D, i_{c} \otimes 1] = L_{c} \otimes 1 + 1 \otimes L_{c}.
$$
Fix an element of $\uW\gfrak \otimes V$ of the form $\xi \otimes v$. Then a straightforward computation gives
\begin{align*}
D\left( (L_{c} \otimes 1 + 1 \otimes L_{c})(\xi \otimes v)\right) = & \, \dW(L_c \xi) \otimes v + (-1)^{\vert \xi \vert} L_{c}\xi \otimes \delta v \\
& +  \dW \xi \otimes L_c v  + (-1)^{\vert \xi \vert} \xi \otimes \delta( L_c  v) \\
& + (t^{a} L_{c} \xi) \otimes L_a v -(-1)^{\vert \xi \vert} (w^{a} L_c \xi) \otimes i_a v  \\
&+ (t^{a} \xi ) \otimes L_a(L_c v) - (-1)^{\vert \xi \vert} (w^{a} \xi) \otimes i_a(L_c v) ,
\end{align*}
and 
\begin{align*}
(L_{c} \otimes 1 + 1 \otimes L_{c})\left( D(\xi \otimes v)\right) = & \, L_c(\dW \xi) \otimes v + (-1)^{\vert \xi \vert} L_{c}\xi \otimes \delta v \\
&+  \dW \xi \otimes L_c v  + (-1)^{\vert \xi \vert} \xi \otimes L_c (\delta v) \\
& - f_{cb}^{a} ( t^{b}  \xi) \otimes L_a v+ (t^{a} L_{c} \xi) \otimes L_a v  \\
& +(-1)^{\vert \xi \vert} f_{cb}^{a}( w^{b} \xi) \otimes i_a v -(-1)^{\vert \xi \vert} (w^{a} L_c \xi) \otimes i_a v \\
& + (t^{a} \xi ) \otimes L_c(L_a v) - (-1)^{\vert \xi \vert} (w^{a} \xi) \otimes L_c(i_a v).
\end{align*}
Therefore, from \eqref{eqn:2.22}, it follows that
\begin{align*}
[D, L_{c} \otimes 1 + 1 \otimes L_{c}] (\xi \otimes v) &= [\dW,L_c]\xi \otimes v + (-1)^{\vert \xi \vert} \xi \otimes [\delta,L_c]v \\
&\phantom{=}\, +  f_{cb}^{a} ( t^{b} \xi) \otimes L_a v - (-1)^{\vert \xi \vert} f_{cb}^{a}( w^{b} \xi) \otimes i_a v \\
&\phantom{=}\, +  (t^{a} \xi ) \otimes [L_a,L_c]v + (-1)^{\vert \xi \vert} (w^{a} \xi) \otimes [L_c,i_a]v \\
&=  f_{cb}^{a} ( t^{b}  \xi) \otimes L_a v - (-1)^{\vert \xi \vert} f_{cb}^{a}( w^{b} \xi) \otimes i_a v  \\
&\phantom{=}\, + f_{ac}^{b} ( t^{a}  \xi) \otimes L_b v + (-1)^{\vert \xi \vert} f_{ca}^{b}( w^{a} \xi) \otimes i_b v \\
&= 0.
\end{align*}
Thus the first identity is established. On the other hand, again by a direct computation,
\begin{align*}
D \left( (i_c \otimes 1)(\xi \otimes v) \right) &= \dW (i_c \xi) \otimes v - (-1)^{\vert \xi \vert} i_c \xi \otimes \delta v  \\
&\phantom{=}\, + (t^{a} i_c \xi) \otimes L_a v + (-1)^{\vert \xi \vert} (w^{a} i_c \xi) \otimes i_a v,
\end{align*}
and
\begin{align*}
(i_c \otimes 1)\left( D(\xi \otimes v)\right) &= i_c (\dW \xi) \otimes v + (-1)^{\vert \xi \vert} i_c \xi \otimes \delta v  \\
&\phantom{=}\, +\delta^{a}_{c}\xi L_a v - (t^{a} i_c \xi) \otimes L_a v \\
&\phantom{=}\, -(-1)^{\vert \xi \vert} (w^{a} i_c \xi) \otimes i_a v.
\end{align*}
Hence, using \eqref{eqn:2.22} again, this gives 
\begin{align*}
[D, i_{c} \otimes 1] (\xi \otimes v) &= [\dW,i_c] \xi \otimes v + \delta^{a}_{c}\xi L_a v \\
&= L_c \xi \otimes v + \xi L_c v \\
&= (L_c \otimes 1 + 1 \otimes L_c) (\xi \otimes v),
\end{align*}
and, consequently, the second identity also holds. 
\end{proof}

The preceding discussion allows us to define a DG category, which provides a DG enhancement of the category $\Rep(\TT\gfrak)$, by the following data. The objects of this DG category are the same as those of $\Rep(\TT\gfrak)$. For any two objects $V$ and $V'$, with corresponding differentials $D$ and $D'$, the space of morphisms is the graded vector space $(\uW\gfrak \otimes \Hom(V,V'))_{\bas}$ with the differential $\partial_{D,D'}$ acting according to the formula
\begin{equation}\label{eqn:4.2aa}
\partial_{D,D'}\varphi = D' \circ \varphi - (-1)^k \varphi \circ D,
\end{equation}
for any homogeneous element $\varphi$ of degree $k$. The DG category given by this data will be denoted by $\DGRep(\TT\gfrak)$.

\subsection{The Bott-Shulman-Stasheff DG category}\label{sec:4.2}
In this subsection we introduce a DG category canonically associated to the Lie group $G$, which is based on the Bott-Shulman-Stasheff model discussed in \S\ref{sec:3.1}. This DG category will play an essential intermediate role in the proof of our main result. 

Let $V$ be object of $\Rep(\TT\gfrak)$ and consider the DG algebra $\Omega^{\sbullet}(BG_{\sbullet})\otimes \End(V)$ with multiplication induced by the composition operation on $\End(V)$ and the differential $\bar{d} + \partial + \bar{\delta}$,  where $\bar{\delta}$ here is defined as $(-1)^{p}$ times the differential $\delta$ when acting on $\Omega^{\sbullet}(BG_{p})\otimes \End(V)$. Let also $\Phi_V$ be the left-equivariant representation form associated to $V$. We note that $\Phi_V$ may be thought of as an element of $\Omega^{\sbullet}(BG_{\sbullet})\otimes \End(V)$ of homogeneous of degree $1$ with respect to the total degree. 

\begin{lemma}
The element $\Phi_V - \id_V$ is a Maurer-Cartan element of $\Omega^{\sbullet}(BG_{\sbullet})\otimes \End(V)$. 
\end{lemma}

\begin{proof}
We must show that
$$
( -d + \partial - \delta) (\Phi_V - \id_V) + (\Phi_V - \id_V) \abxcup (\Phi_V - \id_V) = 0. 
$$
To prove this, we first notice that $d (\id_V) = 0$ and $\delta(\id_V) = 0$. Moreover, by decomposing $\Phi_V = \sum_{k \geq 0} \Phi_V^{(k)}$ and bringing to mind the ``descent equations'' \eqref{eqn:2.34}, we obtain that
$$
(d + \delta) \Phi_V = 0.  
$$
We are thus left to show that
$$
\partial (\Phi_V - \id_V) + (\Phi_V - \id_V) \abxcup (\Phi_V - \id_V) = 0. 
$$
Toward this end, we notice that in the present situation $\partial = \varepsilon_0^* - \varepsilon_1^* + \varepsilon_2^*$. Furthermore, according to \eqref{eqn:3.2}, the face maps $\varepsilon_0$, $\varepsilon_1$ and $\varepsilon_2$ coincide with the projection onto the second component $\pi_2$, the multiplication map $\mu$ and the projection onto the first component $\pi_1$, respectively. Therefore,
$$
\partial(\Phi_V - \id_V) = \pi_2^*\Phi_V - \mu^*\Phi_V + \pi_1^*\Phi_V - \id_V. 
$$
On the other hand, taking note of the condition \eqref{eqn:2.35}, we have, for the first cup product term,
\begin{align*}
\Phi_V \abxcup \Phi_V &= \sum_{k \geq 2} \sum_{i + j = k} \Phi_V^{(i)} \abxcup \Phi_V^{(j)} =  \sum_{k \geq 2} \sum_{i + j = k} (-1)^{i(1+ j)}(-1)^{i} \pi_1^*\Phi_V^{(i)} \wedge \pi_2^*\Phi_V^{(j)} \\
&= \sum_{k \geq 2} \sum_{i + j = k} (-1)^{ij} \pi_1^*\Phi_V^{(i)} \wedge \pi_2^*\Phi_V^{(j)} = \sum_{k \geq 2} \mu^* \Phi_V^{(k)} = \mu^* \Phi_V. 
\end{align*}
For the remaining cup product terms, we compute
\begin{align*}
\Phi_V \abxcup \id_V &= \sum_{k \geq 2} \Phi_V^{(k)} \abxcup \id_V = \sum_{k \geq 2}(-1)^{-k}(-1)^{k} \pi_1^*\Phi_V^{(k)} \wedge \pi_2^*(\id_V) = \sum_{k \geq 2} \pi_1^*\Phi_V^{(k)} =  \pi_1^*\Phi_V, \\
\id_V \abxcup \Phi_V &= \sum_{k \geq 2} \id_V \abxcup \Phi_V^{(k)} = \sum_{k \geq 2} \pi_1^*(\id_V) \wedge \pi_2^*\Phi_V^{(k)} = \sum_{k \geq 2} \pi_2^*\Phi_V^{(k)} = \pi_2^*\Phi_V, \\
\id_V \abxcup \id_V &= \pi_1^*(\id_V) \wedge \pi_2^*(\id_V)  = \id_V.
\end{align*}
Consequently,
\begin{align*}
\partial (\Phi_V &- \id_V) + (\Phi_V - \id_V) \abxcup (\Phi_V - \id_V) \\
&= \partial (\Phi_V - \id_V) + \Phi_V  \abxcup \Phi_V - \Phi_V  \abxcup \id_V - \id_V  \abxcup \Phi_V + \id_V  \abxcup \id_V \\
&= \pi_2^*\Phi_V - \mu^*\Phi_V + \pi_1^*\Phi_V - \id_V +  \mu^* \Phi_V^{(k)} -  \pi_1^*\Phi_V -  \pi_2^*\Phi_V +  \id_V \\
&=0 ,
\end{align*}
as we wished to show. 
\end{proof}

As in the previous section, we have the following direct consequence of this result. 

\begin{corollary}\label{cor:4.5}
The operator $D$ in $\Omega^{\sbullet}(BG_{\sbullet}) \otimes V$ given by
$$
D = \bar{d} + \partial + \bar{\delta} + \Phi_V - \id_V,
$$
is a derivation of homogeneous degree $1$ that satisfies $D^2 = 0$. 
\end{corollary}

In light of the above discussion, we can define a DG category by the following data. The objects of this DG category are the same as those of $\Rep(\TT \gfrak)$. For any two objects $V$ and $V'$, with corresponding differentials $D$ and $D'$, the space of morphisms is the graded vector space $\Omega^{\sbullet}(BG_{\sbullet}) \otimes \Hom(V,V')$ with the differential $\partial_{D,D'}$ given by the same formula as the one for $\DGRep(\TT \gfrak)$. This DG category will be called the \emph{Bott-Shulman-Stasheff DG category} and will be denoted by $\BSS(G)$.  

\subsection{The invariant Bott-Shulman-Stasheff DG category}\label{sec:4.3}
Our aim now is to consider an invariant version of the Bott-Shulman-Stasheff DG category we have just introduced. It is this DG category  that is linked to the ``infinitesimal'' DG category $\DGRep(\TT \gfrak)$ discussed in \S\ref{sec:4.1}. We start with some preliminary remarks. The notation is the same as in \S\ref{sec:3.1}. 

Let $V$ be an object of $\Rep(\TT \gfrak)$ with associated left-equivariant representation form $\Phi_V$. Recall that, with respect to the decomposition $\Phi_V = \sum_{k \geq 0} \Phi^{(k)}_V$, the zeroth component $\Phi^{(0)}_V$ is a representation of $G$ on $V$. With this understanding, let us consider the action $\widehat{\gamma}_0(g)$ of elements $g$ of $G$ on $\Omega^{q}(BG_{p}) \otimes V$ defined by
\begin{equation}\label{eqn:4.2}
\widehat{\gamma}_0(g) (\omega \otimes v) = \gamma_0(g)^{*} \omega \otimes \big( \Phi^{(0)}_V(g) (v) \big),
\end{equation}
for $\omega \in \Omega^{q}(BG_{p})$ and $v \in V$. We should also consider the action $\widehat{\gamma}(g_1,\dots,g_p)$ of elements $(g_1,\dots,g_p)$ of $G_{p}$ on $\Omega^{q}(BG_{p}) \otimes V$ given by
\begin{equation}\label{eqn:4.3}
\widehat{\gamma}(g_1,\dots,g_p) (\omega \otimes v) = \gamma(g_1,\dots,g_p)^{*} \omega \otimes v ,
\end{equation}
for $\omega \in \Omega^{q}(BG_{p})$ and $v \in V$. Noting that these two actions commute, we obtain an action $\widehat{\zeta}(g_0,g_1,\dots,g_p)$ of elements $(g_0,g_1,\dots,g_p)$ of $G_{p+1}$ on $\Omega^{q}(BG_{p}) \otimes V$ by simply putting
\begin{equation}\label{eqn:4.4}
\widehat{\zeta}(g_0,g_1,\dots,g_p) = \widehat{\gamma}_0(g_0) \circ \widehat{\gamma}(g_1,\dots,g_p). 
\end{equation}
For what follows, we let $[\Omega^{q}(BG_{p}) \otimes V]^{G_{p+1}}$ denote the subspace of $G_{p+1}$-invariant elements of $\Omega^{q}(BG_{p}) \otimes V$. 

\begin{lemma}\label{lem:4.6}
$[\Omega^{\sbullet}(BG_{\sbullet}) \otimes V]^{G_{\ssbullet+1}}$ is a subcomplex of $\Omega^{\sbullet}(BG_{\sbullet}) \otimes V$ and the inclusion
$$
[\Omega^{\sbullet}(BG_{\sbullet}) \otimes V]^{G_{\ssbullet+1}} \longrightarrow \Omega^{\sbullet}(BG_{\sbullet}) \otimes V
$$
is a quasi-isomorphism. 
\end{lemma}

\begin{proof}
In view of the second relation in \eqref{eqn:2.22}, we deduce that $\bar{\delta}$ preserves $G_{p + 1}$-invariant elements.
Consequently, the result follows from the first part of the proof of Lemma~\ref{lem:3.1}.
\end{proof}

Consider next the derivation $D$ of $\Omega^{\sbullet}(BG_{\sbullet}) \otimes V$ as defined in Corollary~\ref{cor:4.5}. We have the following important observation.

\begin{lemma}\label{lem:4.7}
$D$ preserves the subcomplex $[\Omega^{\sbullet}(BG_{\sbullet}) \otimes V]^{G_{\ssbullet+1}}$. 
\end{lemma}

\begin{proof}
From the definition, it is clearly sufficient to show that the left-equivariant representation form $\Phi_V$ preserves $[\Omega^{\sbullet}(BG_{\sbullet}) \otimes V]^{G_{\ssbullet+1}}$. To prepare for this, we first observe that the $k$th component $\Phi^{(k)}_V$ of $\Phi_V$ satisfies
\begin{equation}\label{eqn:4.5}
L_{g}^* \Phi^{(k)}_V = \Phi^{(0)}_V(g) \circ \Phi^{(k)}_V,
\end{equation}
where $L_{g}$ indicates the left translation determined by the group element $g$; see Lemma~3.15 of \cite{AriasAbad2019}.  
Using this, we claim that
\begin{equation}\label{eqn:4.6}
R_{g}^*  \Phi^{(k)}_V = \Phi^{(k)}_V \circ \Phi^{(0)}_V(g),
\end{equation}
where $R_{g}$ indicates the right translation determined by the group element $g$. Indeed, let $x_1,\dots,x_k \in \gfrak$. Then, attending to the definition of $\Phi^{(k)}_V$ in \eqref{eqn:2.33}, we get
\begin{align*}
(R_{g}^*  \Phi^{(k)}_V)(e) (x_1,\dots,x_k) &=  \Phi^{(k)}_V(g) \big( (dR_g)_e(x_1),\dots, (dR_g)_e(x_k) \big) \\
&= (L_{g}^*  \Phi^{(k)}_V)(e) \big( d(L_{g^{-1}} \circ R_g)_e(x_1),\dots, d(L_{g^{-1}} \circ R_g)_e(x_k) \big) \\
&= (L_{g}^*  \Phi^{(k)}_V)(e) \big( \Ad_{g^{-1}} x_1,\dots, \Ad_{g^{-1}} x_k \big) \\
&= \Phi^{(0)}_V(g) \left(  \Phi^{(k)}_V(e) \big( \Ad_{g^{-1}} x_1,\dots, \Ad_{g^{-1}} x_k \big) \right) \\
&= \Phi^{(0)}_V(g) \left( \Phi^{(0)}_V(g^{-1}) \circ \Phi^{(k)}_V(e) \big(  x_1,\dots, x_k \big) \circ \Phi^{(0)}_V(g) \right) \\
&= \Phi^{(k)}_V(e) \big(  x_1,\dots, x_k \big) \circ \Phi^{(0)}_V(g), 
\end{align*}
as we wished. Next, let us take an invariant element $\eta \in [\Omega^{q}(BG_{p}) \otimes V]^{G_{p + 1}}$. On account of  \eqref{eqn:4.2}, \eqref{eqn:4.3} and \eqref{eqn:4.4}, this means that 
\begin{equation}\label{eqn:4.7}
\zeta(g_0,\dots,g_{p})^* \eta =\Phi^{(0)}_V(g_0) ( \eta).
\end{equation}
We need to show that $\Phi^{(k)}_V \abxcup \eta \in [\Omega^{k+ q}(BG_{p+1}) \otimes V]^{G_{p + 2}}$. On this purpose we notice firstly that
\begin{equation}\label{eqn:4.8}
\Phi^{(k)}_V \abxcup \eta = \pi_1^* \Phi^{(k)}_V \wedge \pi_{(p)}^* \eta, 
\end{equation}
where $\pi_1 \colon G_{p+1} \to G$ is the projection onto the first factor and $\pi_{(p)}\colon G_{p+1} \to G_{p}$ is the projection onto the remaining $p$ factors. Notice, secondly, that 
\begin{align}\label{eqn:4.9}
\begin{split}
\pi_1 \circ \zeta(g_0,\dots, g_{p+1}) &= L_{g_0} \circ R_{g_1^{-1}}  \circ \pi_1, \\
\pi_{(p)} \circ \zeta(g_1,\dots, g_{p+1}) &= \zeta(g_1,\dots, g_{p+1}) \circ \pi_{(p)}.
\end{split}
\end{align}
By using \eqref{eqn:4.5} ,\eqref{eqn:4.6}, \eqref{eqn:4.7}, \eqref{eqn:4.8} and \eqref{eqn:4.9}, we find
\begin{align*}
\zeta(g_0,\dots, g_{p+1})^* (\Phi^{(k)}_V \abxcup \eta) &= \zeta(g_0,\dots, g_{p+1})^* \big( \pi_1^* \Phi^{(k)}_V \wedge \pi_{(p)}^* \eta\big) \\
&= (\pi_1 \circ \zeta(g_0,\dots, g_{p+1}) )^* \Phi^{(k)}_V \wedge (\pi_{(p)} \circ \zeta(g_0,\dots, g_{p+1}) )^* \eta \\
&= (L_{g_0} \circ R_{g_1^{-1}}  \circ \pi_1)^*\Phi^{(k)}_V  \wedge (\zeta(g_1,\dots, g_{p+1}) \circ \pi_{(p)})^* \eta \\
&= \pi_1^* R_{g_1^{-1}}^* L_{g_0}^* \Phi^{(k)}_V \wedge \pi_{(p)}^* \zeta(g_1,\dots, g_{p+1})^* \eta \\
&= \pi_1^* \big(\Phi^{(0)}_V(g_0) \circ \Phi^{(k)}_V \circ \Phi^{(0)}_V(g_1^{-1}) \big) \wedge \pi_{(p)}^* \big( \Phi^{(0)}_V(g_1) (\eta) \big) \\
&=  \Phi^{(0)}_V(g_0) \big( \pi_1^* \Phi^{(k)}_V \wedge \pi_{(p)}^* \eta\big) \\
&= \Phi^{(0)}_V(g_0) (\Phi^{(k)}_V \abxcup \eta),
\end{align*}
which implies what we want. 
\end{proof}

With this result in hand, we can now define the equivariant version of the Bott-Shulman-Stasheff DG category, which we denote by $\BSS^{G}(G)$. Its objects are the same as those of $\BSS(G)$, and, as such, they are just objects in the category $\Rep(\TT\gfrak)$. For any two objects $V$ and $V'$, the space of morphisms is the graded vector space $[\Omega^{\sbullet}(BG_{\sbullet}) \otimes \Hom(V,V')]^{G_{\ssbullet+1}}$ with differential $\partial_{D.D'}$ given by exactly the same formula as that of $\BSS(G)$. Note that Lemmas~\ref{lem:4.6} and \ref{lem:4.7} ensure that this is well-defined. We would also like to highlight the following key result. 

\begin{proposition}\label{prop:4.8aa}
The inclusion DG functor from $\BSS^{G}(G)$ to $\BSS(G)$ is a quasi-equivalence.  
\end{proposition}

\begin{proof}
For any pair of objects $V$ and $V'$ in $\BSS^{G}(G)$, since $G_{p+1}$ is compact and connected, we know that the inclusion $[\Omega^{\sbullet}(BG_{p}) \otimes \Hom(V,V')]^{G_{p+1}} \to \Omega^{\sbullet}(BG_{p}) \otimes \Hom(V,V')$ is a quasi-isomorphism. The result thus follows by an argument entirely similar to that of the proof of Lemma~\ref{lem:3.1}. 
\end{proof}


\subsection{The Van~Est DG functor}
In this subsection we describe the construction of a DG functor between the equivariant Bott-Shulman-Stasheff DG category $\BSS^{G}(G)$ and the DG enhanced category $\DGRep(\TT\gfrak)$, which is a quasi-equivalence when $G$ is compact. We use freely the definitions and notation from \S\ref{sec:3.1}. 

Let $V$ be and object of $\Rep(\TT\gfrak)$ and consider again the cochain complex $\Omega^{\sbullet}(BG_{\sbullet}) \otimes V$. For fixed $p$ and $q$, we let $[\Omega^{q}(BG_{p}) \otimes V]^{G_{p}}$ denote the subspace of $G_{p}$-invariant elements of $\Omega^{q}(BG_{p}) \otimes V$ with respect to the action \eqref{eqn:4.3}. From the definition it is obvious that $[\Omega^{q}(BG_{p}) \otimes V]^{G_{p}}$ coincides with $\Omega^{q}(BG_{p})^{G_{p}} \otimes V$. Thus, evaluation at $(e,\dots, e)$ induces an isomorphism of graded vector spaces from $[\Omega^{q}(BG_{p}) \otimes V]^{G_{p}}$ onto $\Lambda^{q}\gfrak_{p}^* \otimes V$. On the latter, we consider the action $\widehat{\gamma}'_0(g)$ of elements $g$ of $G$ defined by  
\begin{equation}
\widehat{\gamma}'_0(g) (\xi \otimes v) = \Ad_g^*\xi \otimes \big(\Phi_V^{(0)}(g) (v) \big),
\end{equation}
for $\xi \in \Lambda^{q}\gfrak_{p}^*$ and $v \in V$. The following result, which is a direct consequence of Lemma~\ref{lem:3.2}, will be needed below. 

\begin{lemma}\label{lem:4.9}
The following diagram commutes
\begin{equation*}
\xymatrix@C=7ex{[\Omega^{q}(BG_{p}) \otimes V]^{G_{p}} \ar[r]^-{\widehat{\gamma}_0(g)} \ar[d] & [\Omega^{q}(BG_{p})\otimes V]^{G_{p}} \ar[d] \\
\Lambda^{q}\gfrak_p^* \otimes V \ar[r]^-{\widehat{\gamma}'_0(g)} & \Lambda^{q}\gfrak_p^* \otimes V,}
\end{equation*}
where the vertical arrows denote evaluation at the element $(e,\dots,e)$. 
\end{lemma}

Next we consider the morphism of cochain complexes defined by
$$
\VEscr_V = \VE \otimes \id_{V} \colon \Omega^{\sbullet}(G_{\sbullet}) \otimes V \longrightarrow \uW^{\sbullet,\sbullet} \gfrak \otimes V,
$$  
where $\VE \colon \Omega^{\sbullet}(G_{\sbullet}) \to \uW^{\sbullet,\sbullet} \gfrak$ is the Van~Est map. By virtue of Lemma~\ref{lem:3.3}, the restriction of $\VEscr_V$ to $[\Omega^{q}(BG_{p}) \otimes V]^{G_{p}}$ vanishes unless $q = p$. From this it follows at once that this restriction, which we still denote by $\VEscr_V$, has its image contained in $\uS^{p} \gfrak^* \otimes V$. It is also worth pointing out that, if we consider the morphism of graded vector spaces defined by
$$
\widetilde{\VEscr}_V = \widetilde{\VE} \otimes \id_V  \colon \Lambda^{\sbullet} \gfrak_{p}^* \otimes V \to \uS^{\sbullet} \gfrak^* \otimes V,
$$
we get a commutative diagram
\begin{equation*}
\xymatrix{[\Omega^{p}(BG_{p}) \otimes V]^{G_{p}} \ar[dr]^-{\VEscr_V} \ar[d]&  \\
\Lambda^{p}\gfrak_{p}^* \otimes V \ar[r]_-{\widetilde{\VEscr}_V} & \uS^{p}\gfrak^* \otimes V,}
\end{equation*}
with the vertical arrow being the evaluation at $(e,\dots,e)$. This instructs us to introduce yet one more action $\widehat{\gamma}''_0(g)$ of elements $g$ of $G$ on $\uS^{p}\gfrak^* \otimes V$ defined by
\begin{equation}\label{eqn:4.11}
\widehat{\gamma}''_0(g) (\xi \otimes v) = \Ad_g^*f \otimes \big(\Phi_V^{(0)}(g) (v) \big),
\end{equation}
for $f \in \uS^{p}\gfrak^*$ and $v \in V$. The corresponding subspace of $G$-invariants elements of $\uS^{p}\gfrak^* \otimes V$ will be denoted by $(\uS^{p}\gfrak^* \otimes V)^{G}$. Then we have the following result. 

\begin{proposition}\label{prop:4.10}
The restriction of the morphism $\VEscr_V$ to $[\Omega^{\sbullet}(BG_{\sbullet}) \otimes V]^{G_{\ssbullet+1}}$ has its image contained in $(\uW^{\sbullet,\sbullet} \gfrak \otimes V)_{\bas}$. 
\end{proposition}

\begin{proof}
The first thing to notice is that, owing to the definitions in \eqref{eqn:4.1} and \eqref{eqn:4.11}, the graded subspace $(\uW^{\sbullet,\sbullet} \gfrak \otimes V)_{\bas}$ coincides with $(\uS^{\sbullet}\gfrak^* \otimes V)^{G}$. Therefore, in light of Lemma~\ref{lem:4.9}, it will suffice to show that the following diagram commutes
\begin{equation*}
\xymatrix@C=7ex{\Lambda^{p}\gfrak_{p}^* \otimes V \ar[r]^-{\widehat{\gamma}'_0(g)}  \ar[d]_-{\widetilde{\VEscr}_V} & \Lambda^{p}\gfrak_{p}^* \otimes V \ar[d]^-{\widetilde{\VEscr}_V}\\
\uS^p \gfrak^* \otimes V \ar[r]^-{\widehat{\gamma}''_0(g)} & \uS^p \gfrak^* \otimes V.}
\end{equation*}
But this is an easy consequence of the commutativity of the diagram we established in the course of the proof of Proposition~\ref{prop:3.4}. 
\end{proof}

We also note the following result here.

\begin{proposition}\label{prop:4.11}
The Maurer-Cartan element $\Phi_V - \id_V$ of $\Omega^{\sbullet}(BG_{\sbullet}) \otimes \End(V)$ is sent by the morphism $\VEscr_{\End(V)}$ to the Maurer-Cartan element $t^{a} \otimes L_{a} - w^{a} \otimes i_{a}$ of $\uW^{\sbullet,\sbullet} \gfrak \otimes \End(V)$. 
\end{proposition}

\begin{proof}
Let us first write $\Phi_V$ as a sum $\sum_{k \geq 0} \Phi_{V}^{(k)}$. Next, let us observe that, by definition,
$$
\VE \colon \Omega^k(G)  \longrightarrow \uW^{1,k} \gfrak = \Lambda^{1-k}\gfrak^* \otimes \uS^{k}\gfrak^* = \begin{cases} \Lambda^{1}\gfrak^* & \text{if $k=0$,} \\
\uS^{1}\gfrak^* & \text{if $k=1$,} \\
0 & \text{otherwise.}\end{cases}
$$
This implies that $\VEscr_{\End(V)} (\Phi_V^{(k)})= 0$ for $k \geq 2$. On the other hand, attending to the definitions, for each $x \in \gfrak$, we have
$$
\VEscr_{\End(V)} (\Phi_V^{(0)})(x) = (L_{x^{\sharp}} \Phi_V^{(0)})(e) = L_x \circ \Phi_V^{(0)}(e) = L_x \circ \id_V = L_x,
$$
and 
$$
\VEscr_{\End(V)} (\Phi_V^{(1)})(x) = (i_{-x^{\sharp}} \Phi_V^{(1)})(e) = \Phi_V^{(1)}(e)(-x^{\sharp}(e)) = -\Phi_V^{(1)}(e)(x)  = - i_x.
$$
Moreover,  $\VEscr_{\End(V)}(\id_V)= 0$. Since for any $x \in \gfrak$, $(t^{a} \otimes L_{a})(x) = L_x$ and $(w^{a} \otimes i_{a})(x) = i_x$,  conclusion follows. 
\end{proof}

With this preparatory work completed, we come now to the definition that concerns us. Take two object $V$ and $V'$ of $\Rep(\TT\gfrak)$ and let us write $\VEscr_{V,V'}$ for $\VEscr_{\Hom(V,V')}$. Applying Proposition~\ref{prop:4.10} yields that
$$
\VEscr_{V,V'} \colon [\Omega^{\sbullet}(BG_{\sbullet}) \otimes \Hom(V,V')]^{G_{\ssbullet+1}} \longrightarrow  (\uW^{\sbullet,\sbullet} \gfrak \otimes \Hom(V,V'))_{\bas}. 
$$
When combined with Proposition~\ref{prop:4.11}, this shows that the collection of morphisms $\VEscr_{V,V'}$ defines a DG functor $\VEscr \colon \BSS^{G}(G) \to \DGRep(\TT\gfrak)$ which is the identity on objects. We shall henceforth refer to this as the \emph{Van~Est DG functor}. The following result is our main finding in this subsection.  

\begin{theoremB}\label{thm:4.12aa}
The Van~Est DG functor $\VEscr \colon \BSS^{G}(G) \to \DGRep(\TT\gfrak)$ is a quasi-equivalence.  
\end{theoremB}

\begin{proof}
Following the discussion in the last part of \S \ref{sec:3.1}, let us again consider the map $\widehat{\AM}{}^{\theta}$ . We notice that, by definition, the following diagram commutes 
$$
\xymatrix@C=7ex{\uS^{p} \gfrak^* \ar[r]^-{\Ad_g^*}\ar[d]_-{\widehat{\AM}{}^{\theta}} & \uS^{p} \gfrak^* \ar[d]^-{\widehat{\AM}{}^{\theta}} \\
\Omega^p(BG_p)^{G_p} \ar[r]^-{\gamma_0(g)^*} & \Omega^p(BG_p)^{G_p}. }
$$
Next, for any two objects $V$ and $V'$ of $\Rep(\TT\gfrak)$, we set
$$
\widehat{\AMsrc}{}^{\theta}_{V,V'} =\widehat{\AM}{}^{\theta} \otimes \id_{\Hom(V,V')} \colon \uS^{p}\gfrak^* \otimes \Hom(V,V') \longrightarrow \Omega^{p}(BG_{p})^{G_p} \otimes \Hom(V,V'). 
$$   
Then, from the definitions \eqref{eqn:4.2} and \eqref{eqn:4.11}, it follows on use of the above that the following diagram commutes
$$
\xymatrix@C=7ex{\uS^{p} \gfrak^* \otimes \Hom(V,V') \ar[r]^-{\widehat{\gamma}''_0(g)}\ar[d]_-{\widehat{\AMsrc}{}^{\theta}_{V,V'}} & \uS^{p} \gfrak^* \otimes \Hom(V,V') \ar[d]^-{\widehat{\AMsrc}{}^{\theta}_{V,V'}} \\
\Omega^p(BG_p)^{G_p} \otimes \Hom(V,V') \ar[r]^-{\widehat{\gamma}_0(g)} & \Omega^p(BG_p)^{G_p} \otimes \Hom(V,V'). }
$$
Recalling that $(\uS^{\sbullet} \gfrak^* \otimes \Hom(V,V'))^G$ coincides with $(\uW^{\sbullet,\sbullet}\gfrak \otimes \Hom(V,V'))_{\bas}$ and $[\Omega^{\sbullet}(BG_{\sbullet})^{G_{\sbullet}} \otimes \Hom(V,V')]^G$ coincides with $[\Omega^{\sbullet}(BG_{\sbullet})\otimes \Hom(V,V')]^{G_{\ssbullet + 1}}$, we thus get a morphism of cochain complexes
$$
\widehat{\AMsrc}{}^{\theta}_{V,V'} \colon (\uW^{\sbullet,\sbullet}\gfrak \otimes \Hom(V,V'))_{\bas} \longrightarrow [\Omega^{\sbullet}(BG_{\sbullet})\otimes \Hom(V,V')]^{G_{\ssbullet + 1}}. 
$$
Furthermore, invoking Theorem~\ref{thm:3.7}, we infer that $\widehat{\AMsrc}{}^{\theta}_{V,V'}$ is a left inverse of $\VEsrc_{V,V'}$. This, clearly, yields the result. 
\end{proof}

\begin{theorem}
Let $G$ be a compact simply connected Lie group. The categories $\BSS(G)$ and $\DGRep(\TT\gfrak)$ are $\A_\infty$ quasi-equivalent.
\end{theorem}


\subsection{The De~Rham-Hochschild $\A_{\infty}$-functor}
In this subsection we shall construct an $\A_{\infty}$-quasi-equivalence which connects the Bott-Shulman-Stasheff DG category $\BSS(G)$ to the DG enhanced category $\DGMod(\uC_{\sbullet}(G))$. We will use the De~Rham-Hochschild $\A_{\infty}$-quasi-isomorphism from \S\ref{sec:3.3}. 

Let $V$ be an object of $\Rep(\TT \gfrak)$. 

 By tensoring the $\A_\infty$ map $\DR^{\Theta}:  \Tot(\Omega^{\sbullet}(BG_{\sbullet}))  \to \HC^{\sbullet}(\uC_{\sbullet}(G))$ with the 
 identity on $\End(V)$ one obtains a map:

\[\DRsrc \colon \Tot(\Omega^{\sbullet}(BG_{\sbullet})) \otimes \End(V) \to \HC^{\sbullet}(\uC_{\sbullet}(G))\otimes \End(V),\]which
is an $\A_{\infty}$-quasi-isomorphism. 

To proceed further, let us denote by $\rho \colon \TT \gfrak \to \End(V)$ the structure homomorphism associated with $V$. Following the discussion of \S~\ref{sec:2.6}, we will designate by $\Iscr(\rho) \colon \uC_{\sbullet}(G) \to \End(V)$ the structure homomorphism associated with $V$ when viewed as object of $\Mod(\uC_{\sbullet}(G))$.  We will also write $\II_V \colon \uC_{\sbullet}(G) \to \End(V)$ to denote the homomorphism which associates $\id_V$ to every singular $0$-simplex on $G$, that is, the trivial module. 

\begin{lemma}
The element $\Iscr(\rho)- \II_V$ is a Maurer-Cartan element of $\HC^{\sbullet}(\uC_{\sbullet}(G)\otimes \End(V)$. 
\end{lemma}

\begin{proof}
Applying Proposition~\ref{prop:3.18}, we see that, for each $n > 1$, and for homogeneous elements $\omega_1 \in \Omega^{q_1}(G), \dots, \omega_n \in \Omega^{q_n}(G)$ and $f_1,\dots, f_n \in \End(V)$, 
$$
\DRsrc_n \left(\uu(\omega_1 \otimes f_1) \otimes \cdots \otimes \uu(\omega_n \otimes f_n)\right) = 0. 
$$
Putting this together with the definitions in \eqref{eqn:2.12} and \eqref{eqn:2.38}, we conclude that the Maurer-Cartan element $\Phi_V - \id_V$ of $\Tot(\Omega^{\sbullet}(BG_{\sbullet})) \otimes \End(V)$ is sent by the $\A_{\infty}$-morphism $\DRsrc$ to the element $\Iscr(\rho)- \II_V$ of $\HC^{\sbullet}(\uC_{\sbullet}(G),\End(V))$. Since  $\A_{\infty}$-morphisms preserve Maurer-Cartan elements, the result follows.  
\end{proof}

As usual, an immediate consequence of this is the following. 

\begin{corollary}
The operator $D$ in $\HC^{\sbullet}(\uC_{\sbullet}(G),\End(V))$ given by
$$
D = b + \delta + \Iscr(\rho)- \II_V,
$$
is a derivation of homogeneous degree $1$ that satisfies $D^2 = 0$. 
\end{corollary}

With this point in mind, let us now take $V'$ to be another object of $\Rep(\TT\gfrak)$ with associated structure homomorphism $\rho' \colon \TT \gfrak \to \End(V')$ and differential $D'$, and consider the graded vector space $\HC^{\sbullet}(\uC_{\sbullet}(G),\End(V))$ endowed with the differential $\partial_{D,D'}$ given by the same formula as \eqref{eqn:4.2aa} above. The following result demonstrates the basic link between the latter and the the Hochschild differential on $\HC^{\sbullet}(\uC_{\sbullet}(G),\End(V))$ obtained by the prescription \eqref{eqn:2.6aa}. 

\begin{proposition}\label{prop:4.16}
The differential $\partial_{D,D'}$ coincides with the differential on the  Hochschild cochain complex $\HC^{\sbullet}(\uC_{\sbullet}(G),\End(V))$. 
\end{proposition}

\begin{proof}
Explicitly, we may write $\partial_{D,D'}$ as
\begin{align*}
\begin{split}
&\partial_{D,D'} \varphi (\uu c_1 \otimes \cdots \otimes \uu c_n) = \delta(\varphi(\uu c_1 \otimes \cdots \otimes \uu c_n)) - (-1)^{\vert \varphi \vert} \varphi(b(\uu c_1 \otimes \cdots \otimes \uu c_n)) \\
& \qquad\qquad\qquad\quad\,\,\, + (\Iscr(\rho') \abxcup \varphi)(\uu c_1 \otimes \cdots \otimes \uu c_n)  - (-1)^{\vert\varphi\vert} (\varphi \abxcup \Iscr(\rho))(\uu c_1 \otimes \cdots \otimes \uu c_n), 
\end{split}  
\end{align*}
for homogeneous elements $\varphi \in \Hom((\uu \uC_{\sbullet}(G))^{\otimes n}, \End(V))$ and $c_1,\dots, c_n \in \uC_{\sbullet}(G)$. On the other hand, from the definition of the cup product \eqref{eqn:2.7aa}, we have
\begin{align*}
 (\Iscr(\rho') \abxcup \varphi)(\uu c_1 \otimes \cdots \otimes \uu c_n)  = (-1)^{\vert \varphi \vert (\vert c_1 \vert +1)} \Iscr(\rho')(c_1) \circ \varphi (\uu c_2 \otimes \cdots \otimes \uu c_n), 
\end{align*}
and
\begin{align*}
(\varphi \abxcup \Iscr(\rho))(\uu c_1 \otimes \cdots \otimes \uu c_n)  = (-1)^{\sum_{j=1}^{n-1}\vert c_j \vert + n-1}  \varphi (\uu c_1 \otimes \cdots \otimes \uu c_{n-1})\circ  \Iscr(\rho)(c_n).
\end{align*}
Using this together with \eqref{eqn:2.6aa} gives the desired conclusion. 
\end{proof}

Now we proceed to define the $\A_{\infty}$-functor from $\BSS(G)$ to $\DGMod(\uC_{\sbullet}(G))$, which we also denote by $\DRsrc$ and will refer to as the \emph{Hochschild-De~Rham $\A_{\infty}$-functor}. On objects $\DRsrc$ acts as the integration functor $\Iscr \colon \Rep(\TT \gfrak) \to \Mod(\uC_{\sbullet}(G))$ which gives a module over singular chains given a representation of $\TT \gfrak$. For each $n \geq 1$ and for every collection of objects $V_0, \dots, V_n$ of $\Rep(\TT\gfrak)$, we let
\begin{align*}
&\DRsrc_n \colon \uu \left(\Tot(\Omega^{\sbullet}(BG_{\sbullet}) \otimes \Hom(V_{n-1},V_{n})\right) \otimes \cdots \otimes  \uu \left(\Tot(\Omega^{\sbullet}(BG_{\sbullet}) \otimes \Hom(V_{0},V_{1})\right)  \\
&\qquad\qquad\qquad\qquad\qquad\qquad\qquad\qquad\qquad\qquad\qquad\qquad\qquad\qquad\longrightarrow \HC^{\sbullet}(\uC_{\sbullet}(G),\Hom(V_0,V_n))
\end{align*}
be defined by
\begin{align}
\begin{split}
&\DRsrc_n (\uu(\omega_{n-1} \otimes f_{n-1}) \otimes \cdots \otimes \uu(\omega_{0} \otimes f_{0}) )\\
&\qquad\qquad\qquad\qquad = \uu \left(\us \DR^{\Theta}_n (\uu \omega_{n-1} \otimes \cdots \otimes \uu \omega_{0}) \otimes (f_{n-1} \circ \cdots \circ f_{0}) \right),
\end{split}
\end{align}
for homogeneous elements $\omega_0,\dots, \omega_{n-1} \in \Tot(\Omega^{\sbullet}(BG_{\sbullet})$ and for a composable chain of homogeneous homomorphisms $f_0 \in \Hom(V_{0},V_{1}), \dots, f_{n-1} \in \Hom(V_{n-1},V_{n})$. A straightforward computation, which takes into account Proposition~\ref{prop:4.16}, shows that the sequence of maps $\DRsrc_n$ indeed defines an $\A_{\infty}$-functor $\DRscr \colon \BSS(G) \to \DGMod(\uC_{\sbullet}(G))$. 
\begin{theoremC}\label{thm:4.17aa}
The Hochschild-De~Rham $\A_{\infty}$-functor $\DRscr \colon \BSS(G) \to \DGMod(\uC_{\sbullet}(G))$ is an $\A_{\infty}$-quasi-equivalence. 
\end{theoremC}
\begin{proof}
Since the functor coincides with  $\Iscr \colon \Rep(\TT \gfrak) \to \Mod(\uC_{\sbullet}(G))$ on objects, we know that it is essentially surjective. It is also quasi-fully faithful because the map 
\[\DRsrc \colon \Tot(\Omega^{\sbullet}(BG_{\sbullet})) \otimes \End(V) \to \HC^{\sbullet}(\uC_{\sbullet}(G))\otimes \End(V),\]
is a quasi-isomorphism.
\end{proof}
\subsection{The main theorem}
We are at last in a position to state and prove the principal result of the paper. 

\begin{theoremA}\label{thm:4.18aa}
Suppose that $G$ is compact and simply connected. There exists a zig-zag of $\A_\infty$ quasi-equivalences of DG categories connecting $\DGRep(\TT\gfrak)$ and $\DGMod(\uC_{\sbullet}(G))$. 
\end{theoremA}

\begin{proof}
It follows directly from Proposition~\ref{prop:4.8aa},  Theorem B and Theorem C.
\end{proof}

We finish the paper with a couple of examples that illustrate the content of our main result.

\begin{example}
Let us write $\RR$ for the trivial representation of $\TT\gfrak$. Then, on the one hand,
$$
\mathrm{H}^{\sbullet}\left(\End_{\DGRep(\TT\gfrak)}(\RR)\right) = \mathrm{H}^{\sbullet}\left((\uW\gfrak)_{\bas}\right) \cong (\uS^{\sbullet}\gfrak^*)^G \cong \mathrm{H}^{\sbullet}(BG). 
$$
On the other hand,
$$
\mathrm{H}^{\sbullet}\left(\End_{\DGMod(\uC_{\sbullet}(G))}(\RR)\right) \cong \mathrm{H}^{\sbullet}\left(\HC^{\sbullet}(\uC_{\sbullet}(G))\right) = \HH^{\sbullet}\left(\uC_{\sbullet}(G)\right), 
$$
where ``$\HH^{\sbullet}$'' stands for Hochschild cohomology. Invoking Theorem~\ref{thm:4.18aa}, we conclude that
$$
\mathrm{H}^{\sbullet}(BG)\cong (\uS^{\sbullet}\gfrak^*)^G\cong \HH^{\sbullet}\left(\uC_{\sbullet}(G)\right).
$$
This recovers two models for computing the cohomology of the classifying space $BG$ with coefficients in the trivial local system.
\end{example}
 
\begin{example}
Let us consider the Chevalley-Eilenberg complex $\CE(\gfrak)$ viewed as a representation of $\TT\gfrak$.   Then, on the one hand,
$$
\mathrm{H}^{\sbullet}\left(\Hom_{\DGRep(\TT\gfrak)}(\RR, \CE(\gfrak))\right) = \mathrm{H}^{\sbullet}\left((\uW\gfrak \otimes \CE(\gfrak))_{\bas}\right) \cong \mathrm{H}^{\sbullet}\left((\uS^{\sbullet}\gfrak^* \otimes \CE(\gfrak))^G\right).
$$
On the other hand,
$$
\mathrm{H}^{\sbullet}\left(\Hom_{\DGMod(\uC_{\sbullet}(G))}(\RR, \CE(\gfrak))\right) = \mathrm{H}^{\sbullet}(\HC^{\sbullet}(\uC_{\sbullet}(G), \CE(\gfrak))) = \HH^{\sbullet}(\uC_{\sbullet}(G), \CE(\gfrak)). 
$$
Since the latter is known to be isomorphic to the cohomology of the free loop space $\mathcal{L}BG$ of $BG$, Theorem~\ref{thm:4.18aa} tells us that
$$
\mathrm{H}^{\sbullet}(\mathcal{L}BG) \cong \mathrm{H}^{\sbullet}\left((\uS^{\sbullet}\gfrak^* \otimes \CE(\gfrak))^G\right). 
$$
On recovers the fact that the equivariant cohomology of $G$ acting on itself by conjugation is the cohomology of the loop space of $BG$.
This means that the Chevalley-Eilenberg complex $\CE(\gfrak)$ corresponds to the Gauss-Manin local system for the loop space fibration $\pi: \mathcal{L}BG \to BG$.
\end{example}



\end{document}